\documentclass[a4paper,11pt]{article}
\usepackage{amsmath,amssymb,amsthm,graphicx,inputenc}
\usepackage{hyperref,url,enumitem,bm,bbm,esint,color}
\usepackage[margin=30mm]{geometry}

\newcommand{\sym}[1]{\big ( #1 \big )^{\mathrm{sym}}}
\newcommand{\veceta}{\bm{\eta}}

\renewcommand{\epsilon}{\varepsilon}

\newcommand{\vecS}{\bm{S}}

\newcommand{\I}{\mathbb{I}}

\theoremstyle{plain}
\newtheorem{thm}{Theorem}[section]
\newtheorem{lem}{Lemma}[section]

\newtheorem{remark}{Remark}[section]

\newtheorem{ex}{Example}[section]
\numberwithin{equation}{section}

\newcommand{\bnu}{\bm{\nu}}
\newcommand{\eps}{\varepsilon}
\newcommand{\So}{\bm{S}}
\newcommand{\vp}{\varphi}

\renewcommand{\P}{\mathcal{P}}
\newcommand{\C}{\mathcal{C}}
\newcommand{\R}{\mathbb{R}}
\newcommand{\E}{\mathcal{E}}
\newcommand{\K}{\mathcal{K}}
\newcommand{\N}{\mathbb{N}}
\newcommand{\F}{\mathcal{F}}
\newcommand{\J}{\mathcal{J}}
\newcommand{\G}{\mathcal{G}}
\newcommand{\V}{\mathcal{V}}

\newcommand{\D}{\mathrm{D}}
\newcommand{\BV}{\mathrm{BV}}
\newcommand{\TV}{\mathrm{TV}}

\newcommand{\dz}{\, \mathrm{d}z}
\newcommand{\dr}{\, \mathrm{d}r}
\newcommand{\dt}{\, \mathrm{d}t}
\newcommand{\dbz}{\, \mathrm{d}\bm{z}}

\newcommand{\bA}{{\bf{A}}}
\newcommand{\BB}{{\bf{B}}}
\newcommand{\GG}{{\bf{G}}}
\newcommand{\Gl}{\GG_l}

\renewcommand{\div}{\mathrm{div}}
\newcommand{\dx}{\, \mathrm{dx}}
\newcommand{\dH}{\, \mathrm{d}\mathcal{H}^{d-1}}
\newcommand{\Haus}{\mathcal{H}^{d-1}}

\newcommand{\p}{\bm{p}}
\newcommand{\rr}{\bm{r}}
\newcommand{\x}{\bm{x}}
\newcommand{\sss}{\bm{s}}
\newcommand{\q}{\bm{q}}
\newcommand{\be}{\bm{\eta}}
\newcommand{\bxi}{\bm{\xi}}
\renewcommand{\u}{\bm{u}}
\newcommand{\w}{{\bm{w}}}

\newcommand{\0}{\bm{0}}
\newcommand{\f}{\bm{f}}
\newcommand{\g}{\bm{g}}
\newcommand{\bv}{\bm{v}}
\newcommand{\opt}{\overline{\vp}}
\newcommand{\uopt}{\overline{\u}}

\newcommand{\hh}{\an{\mathbbm{h}}}

\newcommand{\pd}{\partial}

\newcommand{\inn}[1]{\langle #1 \rangle}
\newcommand{\no}[1]{\| #1 \|}
\newcommand{\tr}{\mathrm{tr}}

\newcommand{\HD}{H^1_D(\Omega, \R^d)}

\definecolor{rosso}{rgb}{0.85,0,0}
\def\an #1{{\color{rosso}#1}}

\def\an #1{{#1}}

\title{Overhang penalization in additive manufacturing via phase field structural topology optimization with anisotropic energies}
\author{Harald Garcke \footnotemark[1] \and Kei Fong Lam \footnotemark[2] \and Robert N\"urnberg \footnotemark[3] \and Andrea Signori \footnotemark[4]}
\date{ }
\begin{document}
\maketitle

\renewcommand{\thefootnote}{\fnsymbol{footnote}}
\footnotetext[1]{Fakult{\"a}t f\"ur Mathematik, Universit{\"a}t Regensburg, 93040 Regensburg, Germany
({\tt Harald.Garcke@mathematik.uni-regensburg.de}).}
\footnotetext[2]{Department of Mathematics, Hong Kong Baptist University, Kowloon Tong, Hong Kong ({\tt akflam@math.hkbu.edu.hk}).}
\footnotetext[3]{Department of Mathematics, University of Trento, Trento, Italy
({\tt robert.nurnberg@unitn.it}).}
\footnotetext[4]{Dipartimento di Matematica ``F. Casorati'',  Universit\`a di Pavia, 27100 Pavia, Italy ({\tt andrea.signori01@unipv.it}).}

\begin{abstract} 
A phase field approach for structural topology optimization with application to additive manufacturing is analyzed. The main novelty is the penalization of {\it overhangs} (regions of the design that require underlying support structures during construction) with anisotropic energy functionals. Convex and non-convex examples are provided, with the latter showcasing oscillatory behavior along the object boundary termed \an{the} {\it dripping effect} in the literature. We provide a rigorous mathematical analysis for the structural topology optimization problem with convex and non-continuously-differentiable anisotropies, deriving the first order necessary optimality condition using subdifferential calculus. Via formally matched asymptotic expansions we connect our approach with previous works in the literature based on a sharp interface shape optimization description. Finally, we present several numerical results to demonstrate the advantages of our proposed approach in penalizing overhang developments.

\end{abstract}

\noindent {\bf \an{Keywords}:}
Topology optimization, phase field, anisotropy, linear elasticity, optimal control, additive manufacturing, overhang penalization

\vskip3mm
\noindent {\bf AMS (MOS) Subject Classification:} {
		49J20, 
		49K40, 
		49J50  
		}

\section{Introduction}
Additive manufacturing (AM) is an innovative building technique that produces objects in a layer-by-layer fashion through fusing or binding raw materials in powder and resin forms. Since its introduction in the 1970s, it has shown great versatility in allowing for the creation of highly complex geometries, immediate modifications and redesign, thus making it an ideal process for rapid prototyping and testing. But despite such advantages over traditional manufacturing technologies, AM still has not seen widespread integration in serial production, and there are still many limitations that have yet to be \an{overcome}. For an overview of the technologies involved and the challenges encountered by AM, we refer the reader to the review article \cite{AAAM} as well as the report \cite{Del}.

In this work, we focus on a recurring issue encountered when practitioners employ AM to construct objects. Some categories of AM technologies, such as Fused Deposition Modeling and Laser Metal Deposition, construct objects in an upwards direction (hereafter referred to as the build direction) by repeatedly depositing and then fusing a new layer of raw materials on the surface of the object. {\it Overhangs} are regions of the constructed object that when placed in a certain orientation extend outwards without any underlying support. For example, the top horizontal bar of a T-shaped structure placed vertically will be classified as an overhang. The main issue with overhangs is that they can deform under their own weight or by thermal residual stresses from the construction process, and if not supported from below, there is a risk that the object itself can display unintended deformations and the entire printing process can fail.

One natural solution is to identify certain orientations of the object whose overhang regions are minimized prior to printing, and choose the build direction to be one of these orientations. Going back to the T-shape structure example, we can simply rotate the shape by $180^{\circ}$ so that there are no overhang regions \an{(see, e.g., \cite[Fig.~5]{Jiang})}. However, in general, there is no guarantee that orientations with no overhang regions exist. In this direction we mention the works \cite{Morgan,Zhang} that incorporate various geometric information such as contact surface area and support volume expressed as functions of the build direction within an optimization procedure.  On the other hand, simultaneous optimization of build direction and topology has been considered in \cite{WangQ}.

Another remedy is to employ {\it support structures} that are built concurrently with the object whose purpose is to reduce the potential deformation of overhang regions through mechanical loads or thermal residual stresses. These supplementary structures act like scaffolding to overhangs, and after a successful print are then removed in a post-processing step. Along with increased material costs and printing time, there is a risk of damaging delicate features of the objects during the removal step. Nevertheless, for certain AM technologies such as the aforementioned Fused Deposition Modeling, some form of support structures will always be necessary in order to mitigate the potential undesirable deformations of the finished object.  Thus, recent mathematical research \an{has} concentrated on the optimizing support structures in an effort to reduce material waste and minimize contact area with the object. Many such works explore optimal support in the framework of shape and topology optimization \cite{All,Gardan,Kuo,Lang,Mirz}, as well as proposing sparse cellular designs that have low solid volume and contact area in the form of lattices \cite{Hussein}, honeycomb \cite{Lu}, and trees \cite{Vanek}.  Further details can be found in the review article \cite{Jiang}. 

A related approach would be to allow some modifications to the object, as long as the altered design retains the intended functionality of the original, leading to creation of {\it self-supporting} objects that do not require support structures at all \cite{Cacace,Leary13,Leary,Liu}. Our present work falls roughly into this category, where we employ a well-known phase field methodology in structural topology optimization with the aim of identifying optimal designs fulfilling the so-called {\it overhang angle constraint}.  In mathematical terms, let us consider an object $\Omega_1 \subset \R^{d}$, $d=2,3,$ being built within a hold-all rectangular domain $\Omega = [0,1]^d$ in a layer-by-layer fashion with build direction $\bm{e}_{d} = (0,\dots, 1)^{\top}$. The {\it base plate} is the region $\mathcal{B} := [0,1]^{d-1}\times \{0\}$ where $\Omega_1$ is assembled on. For any $\p \in \pd \Omega_1 \setminus \mathcal{B}$, the overhang angle $\alpha$ is defined as the angle from the base plate to the tangent line at $\p$ (denoted as $\alpha_1$ and $\alpha_2$ in Figure \ref{fig:overhang}). 
\begin{figure}[h]
\centering
\includegraphics[width=0.6\textwidth]{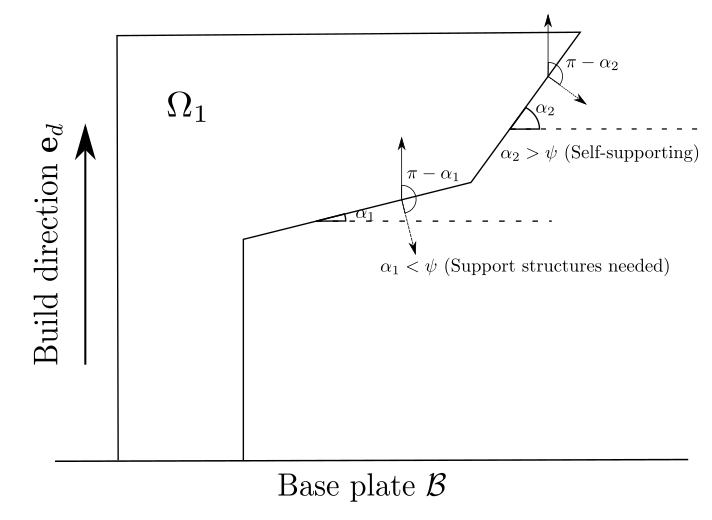}
\caption{Definition of overhang angle measured from the base plate to the tangent line on the boundary.}
\label{fig:overhang}
\end{figure}
Then, there exists a critical threshold angle $\psi$ where the portion of $\Omega_1$ is self-supporting (resp.~requires support structures) if the overhang angle there is greater (resp.~smaller) than $\psi$. Conventional wisdom from practitioners puts the critical angle $\psi$ at $45^{\circ}$ \cite{Dapogny,Lang,Thomas} with the reasoning that every layer then has approximately 50\% contact with the previous layer and thus it is well-supported from below, although there are some authors that propose smaller values for $\psi$ (e.g., $40^{\circ}$ in \cite{Leary13,Leary}). Exact values of this critical angle also depend on the setting of the 3D printer as well as physical properties of the raw materials used. We mention that other works \cite{Mirz} consider an alternate definition, which is given as the angle from the vertical build direction to the outward unit normal of $\pd \Omega_1$. Denoting this as $\beta$, a simple calculation shows that $\beta = 180^{\circ} - \alpha$. Then, regions where $\beta > 180^{\circ} - \psi$ (typically $135^{\circ}$ if $\psi = 45^{\circ}$) will require support structures.

In this work we adopt the latter definition, but choose to define the overhang angle as the angle measured from the negative build direction $-\bm{e}_d$ to the outer unit normal, see Figure \ref{fig:overhang_def} and also \cite{AllaireDEFM17,Cacace}. Then, given a critical angle threshold $\psi$ (expressed now in radians), the overhang angle constraint in the current context of shape and topology optimization would demand that an optimal geometry for the object $\Omega_1$ has minimal regions where the overhang angles there do not lie in the interval $[\psi, 2\pi - \psi]$ (i.e., $\Omega_1$ should be self-supporting as much as possible), in addition to other mechanical considerations such as minimal compliance.

\begin{figure}[h]
\centering
\includegraphics[width=0.8\textwidth]{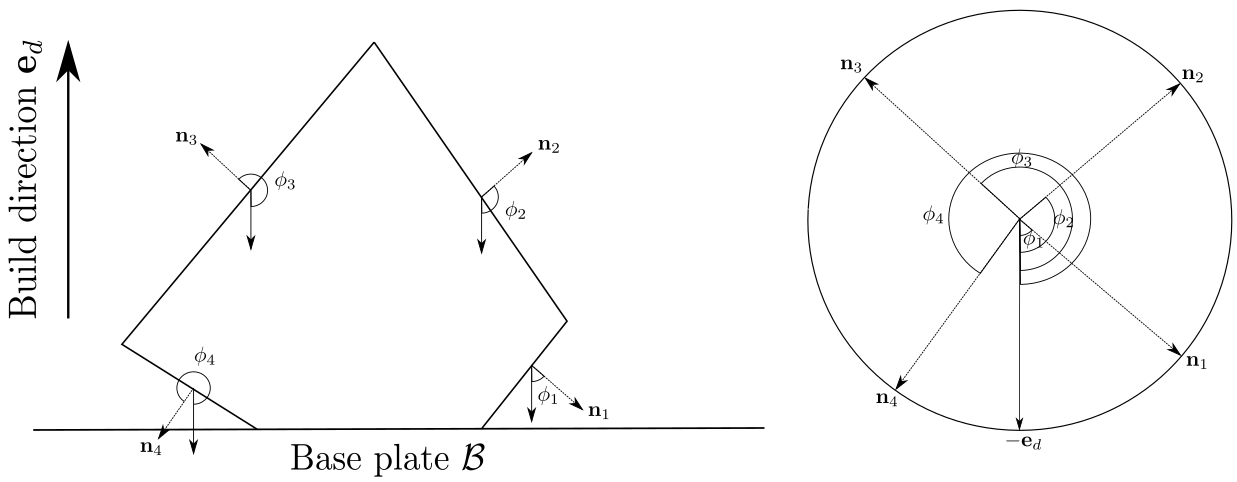}
\caption{The definition of overhang angle used in this work. (Left) A structure with four angles measured from the negative build direction $-\bm{e}_d$ to the outward unit normal on the boundary. (Right) Visualization of the angles on the unit circle, where the angle $\phi_i$ is associated to unit normal $\bm{n}_i$.}
\label{fig:overhang_def}
\end{figure}

For structural topology optimization we employ the phase field methodology proposed in \cite{Bourdin}, later popularized by many authors to other applications such as multi-material structural topology optimization \cite{BGFS,WangZ04}, compliance optimization \cite{BlankGSSSV12,Take}, topology optimization with local stress constraints \cite{Burger}, nonlinear elasticity \cite{Penzler}, elastoplasticity \cite{Almi}, eigenfrequency maximization \cite{GHK,Take}, graded-material design \cite{Carr}, shape optimization in fluid flow \cite{GHechtNS,GHechtStokes,GHHKL}, as well as resolution strategies for some inverse identification problems \cite{Beretta,LY}. The key idea is to cast the structural topology optimization problem as a constrained minimization problem for a phase field variable $\vp$, where the geometry of an optimal design $\Omega_1$ for the object can be realized as a certain level-set of $\vp$. In the language of optimal control theory, $\vp$ acts as a control and influences a response function, e.g., the elastic displacement $\u$ of the object, that is the solution to a system of partial differential equations (see \eqref{state} below). With an appropriate objective functional, for instance, a weighted sum of the mean compliance and a phase field formulation of the overhang constraint (see \eqref{obj} below), we analyze the PDE-constrained minimization problem for an optimal design variable $\vp$. 

In the above references, previous authors employ an isotropic Ginzburg--Landau functional that has the effect of perimeter penalization (see Section \ref{sec:aniso} for more details), and in the current context this means no particular directions are preferred/discouraged in the optimization process. In this work we follow the ideas first proposed in \cite{AllaireDEFM17,Dapogny} that use an anisotropic perimeter functional as a geometric constraint for overhangs (see also similar ideas in \cite{All_sppt} for support structures), and introduce a suitable anisotropic Ginzburg--Landau functional for the phase field optimization problem to enforce the overhang angle constraint. In one sense, our first novelty is that we generalize earlier works in phase field structural topology optimization by considering anisotropic energies. The precise formulation and details of the problem are given in the next sections. For more details regarding the use of anisotropy in phase field models we refer the reader to \cite{eck,vch,GNS,Taylor}.

Our second novelty is the analysis of the anisotropic phase field optimization problem. We establish the existence of minimizers (i.e., optimal designs) and derive first-order necessary optimality conditions. It turns out that our proposed approach to the overhang angle constraint requires an anisotropy function $\gamma$ that is not continuously differentiable. In turn, the derivation of the associated necessary optimality conditions becomes non-standard. We overcome this difficulty with subdifferential calculus, and provide a characterization result for the subdifferential of convex functionals whose arguments are weak gradients. Then, we perform a formally matched asymptotic analysis for the optimization problem with a differentiable $\gamma$ to infer the corresponding sharp interface limit.

Lastly, let us provide a non-exhaustive overview on related works that integrate the overhang angle constraint within a topology optimization procedure. Our closest counterpart is the work \cite{AllaireDEFM17} which employs an anisotropic perimeter functional in a shape optimization framework implemented numerically with the level-set method. The authors observed that oscillations along the object boundary would develop even though the design complied with the angle constraint. This so-called {\it dripping effect} is attributed to instability effects of the anisotropic perimeter functional, brought about by the non-convexity of the anisotropy used in \cite{AllaireDEFM17}, see also Example \ref{eg:Frank2} and Figure \ref{fig:2dwigglebig32pi_newnc1_ani} below and the discussion in \cite[Sec. 7.3]{Gurtin}. An alternative mechanical constraint based on modeling the layer-by-layer construction process is then proposed to provide a different treatment of the overhangs and seems to suppress the dripping effect (see also \cite{Amir} for similar ideas), but is computationally much more demanding. These boundary oscillations have also been observed earlier in \cite{Qian}, which used a Heaviside projection based integral to encode the overhang angle constraint within a density based topology optimization framework, and can be suppressed by means of an adaptive anisotropic filter introduced in \cite{Mezz}.  Similar projection techniques for density filtering are used in \cite{Gara}, which was combined with an edge detection algorithm to evaluate feasible and non-feasible contours during optimization iterations. In \cite{Gaynor}, control of overhang angles is achieved by means of filtering with a wedge-shaped support, while in \cite{Zhao} an explicit quadratic function with respect to the design density is used to formulate the self-support constraint. A different approach was proposed in \cite{Lang0,Lang1} using a spatial filtering which checks element densities row-by-row and remove all elements not supported by the previous row. This idea is then extended in \cite{Lang,Pellens,Thore} to construct new filtering schemes that include other relevant manufacturing constraints. Another method was proposed in \cite{Guo} using the frameworks of moving morphable components and moving morphable voids to provide a more explicit and geometric treatment of the problem, and is capable of simultaneously optimizing structural topology and build orientation. Finally, let us mention \cite{Ven1,Ven2}, where a filter is developed based on front propagation in unstructured meshes that has a flexibility in enforcing arbitrary overhang angles. Focusing only on overhang angles in enclosed voids, \cite{Luo} combined a nonlinear virtual temperature method for the identification of enclosed voids with a logarithmic-type function to constraint the area of overhang regions to zero.

The rest of this paper is organized as follows: in Section \ref{sec:prob} we formulate the phase field structural optimization problem to be studied, and in Section \ref{sec:aniso} we describe our main idea of introducing anisotropy, along \an{with} some examples and relevant choices, as well as a useful characterization of the subdifferential of the anisotropic functional. The analysis of the structural optimization problem is carried out in Section \ref{sec:ana}, where we establish analytical results concerning minimizers and optimality conditions. The connection between our work and that of \cite{AllaireDEFM17} is explored in Section \ref{sec:SI} where we look into the sharp interface limit, and, finally, in Section \ref{sec:num} we present the numerical discretization and several simulations of our approach.

\medskip

\section{Problem formulation}\label{sec:prob}
In a bounded domain $\Omega \subset \R^d$ with Lipschitz boundary $\Gamma :=\pd \Omega$ that exhibits a decomposition $\Gamma = \Gamma_D \, \cup\,  \Gamma_g \, \cup \, \Gamma_0$, we consider a linear elastic material that does not fully occupy $\Omega$. We describe the material location with the help of a phase field variable $\vp: \Omega \to [-1,1]$. In the phase field methodology, we use the level set $\{\vp = -1\} = \{\an{\x \in \Omega : \vp(\x) =- 1}\}$ to denote the region occupied by the elastic material, and $\{\vp = 1\}$ to denote the complementary void region.  These two regions are separated by a diffuse interface layer $\{|\vp|<1\}$ whose thickness is proportional to a small parameter $\eps > 0$. Since $\vp$ describes the material distribution within $\Omega$, complete knowledge of $\vp$ allows us to determine the shape and topology of the elastic material.

\paragraph{Notation.} For a Banach space $X$, we denote its topological dual by $X^*$ and the corresponding duality pairing by $\inn{\cdot,\cdot}_X$. For any $p \in [1,\infty]$ and $k >0$, the standard Lebesgue and Sobolev spaces over $\Omega$ are denoted by $L^p := L^p(\Omega)$ and $W^{k,p} := W^{k,p}(\Omega)$ with the corresponding norms 
$\no{\cdot}_{L^p}$ and $\no{\cdot}_{W^{k,p}}$.
In the special case $p = 2$, these become Hilbert spaces and we employ the notation $H^k := H^k(\Omega) = W^{k,2}(\Omega)$ with the corresponding norm $\no{\cdot}_{H^k}$. For convenience, the norm and inner product of $L^2(\Omega)$
are simply denoted by $\no{\cdot}$ and $(\cdot,\cdot)$, respectively. For our subsequent analysis, we introduce the space
\begin{align*}
	\HD := \{ \bv \in H^1(\Omega, \R^d) \, : \, \bv = \0 \text{ on } \Gamma_D \}.
\end{align*}

Let us remark that, despite we employ bold symbols to denote vectors, matrices, and vector- or matrix-valued functions, we do not introduce a special notation to indicate the corresponding Lebesgue and Sobolev spaces. Thus, when those terms occur in the estimates, the corresponding norm is to be intended in its natural setting.

\subsection{State system}
To obtain a mathematical formulation that can be further analyzed, we employ the ersatz material approach, see \cite{AllaireErsatz}, and model the complementary void region as a very soft elastic material.  This allows us to consider a notion of elastic displacement on the entirety of $\Omega$, leading to the displacement vector $\u: \Omega \to \R^d$ and the associated linearized strain tensor
\[
\E(\u) := \tfrac{1}{2}(\nabla \u + (\nabla \u)^{\top}).
\]
Let $\C_0$ and $\C_1$ be the fourth order elasticity tensors corresponding to the ersatz soft material and the linear elastic material, respectively, that satisfy the standard symmetric conditions 
\[
\C_{ijkl} = \C_{jikl} = \C_{ijlk} \quad \text{ for } i,j,k,l \in \{1, \dots, d\},
\]
and assume there exist positive constants $\theta$ and $\Lambda$ such that for any non-zero symmetric matrices $\bA$, $\BB \in \R^{d \times d}$:
\[
\theta |\bA|^2 \leq \C \bA : \bA \leq \Lambda |\bA|^2,
\]
where $\bA : \BB = \sum_{i,j=1}^d \bA_{ij} \BB_{ij}$ denotes the Frobenius inner product between two matrices.  With the help of the phase field variable $\vp$, we define an interpolation fourth order elasticity tensor $\C(\vp)$ as
\[
\C(\vp) = \tfrac{1}{2}g(\vp)( \C_0 - \C_1) + \tfrac{1}{2}(\C_0 + \C_1)  \quad \text{ for } \vp \in [-1,1],
\]
where $g:\R \to [-1,1]$ is a monotone function satisfying $g(-1) = -1$ and $g(1) = 1$. Then, it is clear that in the region $\{\vp = -1\}$ (resp.~$\{\vp = 1\}$) we obtain the elasticity tensor $\C_1$ (resp.~$\C_0$) by substituting $\vp = -1$ (resp.~$\vp = 1$) into the definition of $\C(\vp)$. As an example, we can consider 
\[
g(\vp) = \vp, \quad \C_0 = \eps^2 \C_1,
\]
where $0 < \eps \ll 1$ is a constant and for any symmetric matrix $\bA \in \R^{d \times d}$,
\begin{align}\label{Lame}
\C_1 \bA = 2 \mu_1 \bA  + \lambda_1 \tr(\bA) {\bf I}
\end{align}
with identity matrix ${\bf I}$ and Lam\'e constants $\lambda_1$ and $\mu_1$ to obtain a linear interpolating elasticity tensor. Another example for $\C(\vp)$ uses a quadratic interpolation function
\begin{align}\label{eq:gphic}
g(\vp) =1 - \frac{1}{2}(1-\vp)^2.
\end{align}
Let $\f: \Omega \to \R^d$ denote a body force and $\g : \Gamma_g \to \R^d$ denote a surface traction force.  On $\Gamma_D \subset \Gamma$ we assign a zero Dirichlet boundary condition for the displacement $\u$ and on $\Gamma_0$ we assign a traction-free boundary condition. This leads to the following elasticity  system for the displacement $\u$, representing the {\it state system} of the problem:
\begin{subequations}\label{state}
\begin{alignat}{2}
- \div (\C(\vp) \E(\u)) & = \hh(\vp) \f && \quad \text{ in } \Omega, \\
\u &= \0 && \quad \text{ on } \Gamma_D, \\
(\C(\vp) \E(\u)) \bm{n} &= \g && \quad \text{ on } \Gamma_g, \\
(\C(\vp) \E(\u)) \bm{n} &= \0 && \quad \text{ on } \Gamma_0,
\end{alignat}
\end{subequations}
where $\bm{n}$ denotes the outward unit normal to $\Gamma = \Gamma_D \cup \Gamma_g \cup \Gamma_0$, and $\hh: [-1,1] \to [0,1]$ is defined as $\hh(\vp) = \frac{1}{2}(1-\vp)$ is a function introduced so that the body force $\f$ only acts on the elastic material $\{\vp = -1\}$ (where $\hh(\vp) = 1$) and not on the ersatz material $\{\vp = 1\}$ (where $\hh(\vp) = 0$). We extend $\hh(\vp)$ as $1$ if $\vp<-1$ and as $\an{0}$ if $\vp>1$.

The well-posedness of \eqref{state} is a simple consequence of \cite[Thms.~3.1, 3.2]{BGFS}, which we summarize as follows:
\begin{lem}\label{lem:state}
For any $\vp \in L^\infty(\Omega)$ and $(\f, \g) \in L^2(\Omega, \R^d) \times L^2(\Gamma_g, \R^d)$, there exists a unique solution $\u \in \HD $ to \eqref{state} 
satisfying
\begin{align}\label{state:weak}
\int_\Omega \C(\vp) \E(\u) : \E(\bv) \dx = \int_\Omega \hh(\vp) \f \cdot \bv \dx + \int_{\Gamma_g} \g \cdot \bv \dH \text{ for all } \bv \in \HD,
\end{align}
where $\mathcal{H}^{d-1}$ denotes the $(d-1)$-dimensional Hausdorff measure. Moreover, there exists a positive constant $C$, independent of $\vp$, such that 
\begin{align}\label{state:bdd}
\| \u \|_{H^1} \leq C ( 1 + \| \vp \|_{L^\infty}).
\end{align}
In addition, let $M>0$ and $\vp_i \in L^\infty(\Omega)$ with $\|\vp_i\|_{L^\infty} \leq M$, $i=1,2,$ and $\u_i$ be the associated solution to \eqref{state:weak}. Then, there also exists a 
positive constant $C$, depending on the data of the system and $M$, but independent of the difference $\vp_1-\vp_2$, such that 
\begin{align*}
 \| \u_1 - \u_2\|_{H^1} \leq C \| \vp_1 - \vp_2 \|_{L^\infty}.
\end{align*}
\end{lem}
The above result provides a notion of a solution operator, also referred to as the control-to-state operator, $\So : \vp \mapsto \u$ where $\u \in \HD$ is the unique solution to \eqref{state:weak} corresponding to $\vp \in L^\infty(\Omega)$. As such, we may view the phase field variable $\vp$ as a design variable that encodes the elastic response of the associated material distribution through the operator $\So$ and seek an optimal material distribution $\opt$ that fulfills suitable constraints and minimizes some cost functional.

\subsection{Cost functional and design space}
We define the design space, i.e., the set of admissible design variables, as
\begin{align}\label{V}
	\V_m := \Big \{ f \in H^1(\Omega) \, : \, f \in [-1,1] \text{ a.e.~in } \Omega, \, \int_\Omega f \dx = m|\Omega| \Big \} \subset H^1(\Omega) \cap L^\infty(\Omega),
\end{align}
where $m \in (-1,1)$ is a fixed constant for the mass constraint, and motivated by the context of additive manufacturing, we propose the following cost functional to be minimized:
\begin{align}\label{obj}
	J(\vp,\u) = \widehat \alpha \int_\Omega \frac{\eps}{2} |\gamma(\nabla \vp)|^2 + \frac{1}{\eps} \Psi(\vp) \dx + \beta \Big ( \int_\Omega \hh(\vp) \f \cdot \u \dx + \int_{\Gamma_g} \g \cdot \u \dH \Big ){.}
\end{align}
In \eqref{obj}, the second term premultiplied by $\beta > 0$ is the mean compliance functional, while the first term premultiplied by $\widehat \alpha > 0$ is an anisotropic Ginzburg--Landau functional with anisotropy function $\gamma$.  We delay the detailed discussion on $\gamma$ to the next section and remark here that for the isotropic case $\gamma(\x) = |\x|$, $\x \in \R^d$, it is well-known that the Ginzburg--Landau functional is an approximation of the perimeter functional.  Therefore, \eqref{obj} can be viewed as a weighted sum between (anisotropic) perimeter penalization and mean compliance.

Lastly, for the non-negative potential function $\Psi$ in \eqref{obj}, we require that it has $\pm 1$ as its global minima. While many choices are available, in light of the design space $\V_m$, we consider $\Psi$ as the double obstacle potential
\begin{align}\label{obs}
\Psi(s) = \begin{cases}
\Psi_0(s):=\frac{1}{2}(1-s^2) & \text{ if } s \in [-1,1], \\
+\infty & \text{ otherwise}.
\end{cases}
\end{align}
It is worth pointing out that it holds $\Psi(\vp) = \Psi_0(\vp)$ for $\vp \in \V_m$. Then, the structural optimization problem we study can be expressed as the following:
\begin{align}\tag{\bf P}\label{opt}
	\hbox{Minimize $J(\vp, \u)$ subject to $\vp \in \V_m$ and $\u$ solving \eqref{state}. }
\end{align}

\section{Anisotropic Ginzburg--Landau functional}\label{sec:aniso}
In the phase field methodology, the (isotropic) Ginzburg--Landau functional reads as
\begin{align}\label{E}
	E_\eps({\vp}) = \int_\Omega \frac \eps  2 |\nabla \vp|^2 + \frac 1\eps \Psi(\vp) \dx,
\end{align}
where $0 < \eps \ll 1$ and $\Psi$ is a non-negative double-well potential that has $\pm 1$ as its minima, i.e., $\{s \in \R \, : \, \Psi(s) = 0 \} = \{\pm 1\}$. Heuristically, the minimization of \eqref{E} in $H^1(\Omega)$ results in minimizers that take near constant values close to $\pm 1$ in large regions of the domain $\Omega$, which are separated by thin interfacial regions with thickness scaling with $\eps$ over which the functions \an{transit} smoothly from one value to the other. Formally in the limit $\eps \to 0$ these minimizers converge to functions that only take values in $\{\pm 1\}$. This is made rigorous in the framework of $\Gamma$-convergence by the seminal work of Modica and Mortola \cite{Modica,MM}.

To facilitate the forthcoming discussion, we review some basic properties for functions of bounded variations. For a more detailed introduction we refer the reader to \cite{Ambro,EG}.  A function $u \in L^1(\Omega)$ is a function of bounded variation in $\Omega$ if its distributional gradient $\D u$ is a finite Radon measure. The space of all such functions is denoted as $\BV(\Omega)$ and its endowed with the norm $\no{\cdot}_{\BV(\Omega)} = \no{\cdot}_{L^1(\Omega)} + \TV(\cdot)$, where for $u \in \BV(\Omega)$, its total variation $\TV(u)$ is defined as
\[
\TV(u) := |\D u|(\Omega) = \sup \Big \{ \int_\Omega u \, \div {\boldsymbol \phi} \dx \text{ s.t. } {\boldsymbol \phi} \in C^1_0(\Omega, \R^d), \, \no{{\boldsymbol \phi}}_{\an{L^\infty(\Omega,\R^d)}} \leq 1 \Big \}.
\]
The space $\BV(\Omega, \{a,b\})$ denotes the space of all $\BV(\Omega)$ functions taking values in $\{a,b\}$. We say that a set $U \subset \Omega$ is a set of finite perimeter, or a Caccioppoli set, if its characteristic function $\chi_U$, where $\chi_U(\x) = 1$ if $\x \in U$ and $\chi_U(\x) = 0$ if $\x \notin U$, belongs to $\BV(\Omega,\{0,1\})$. The perimeter of a set of finite perimeter $U$ in $\Omega$ is defined as 
\[
\P_\Omega(U) := |\D \chi_U|(\Omega) = \TV(\chi_U),
\]
while its reduced boundary $\pd^*U$ is the set of all points ${\an{\bf y}}\in \R^d$ such that $|\D \chi_{U}|(B_r({\an{\bf y}})) > 0$ for all $r > 0$, with $B_r({\an{\bf y}})$ denoting the ball of radius $r$ centered at ${\an{\bf y}}$, and 
\[
\bnu_U({\an{\bf y}}) := \lim_{r \to 0} \frac{\D \chi_U(B_r({\an{\bf y}}))}{|\D \chi_U|(B_r({\an{\bf y}}))} \text{ exists and } |\bnu_U({\an{\bf y}})| = 1.
\]
The unit vector $\bnu_U({\an{\bf y}})$ is called the measure theoretical unit inner normal to $U$ at ${\an{\bf y}}$, a theorem by De Giorgi yields the connection $\P_\Omega(U) = \Haus(\pd^* U)$, see, e.g., \cite{Ambro}. Then, the result of Modica and Mortola \cite{Modica,MM} can be expressed as follows: The $\Gamma$-limit of the extended functional
\[
\E_\eps(u) := \begin{cases}
\displaystyle \int_\Omega \frac{\eps}{2} |\nabla u|^2 + \frac{1}{\eps} \Psi(u) \dx & \text{ if } u \in H^1(\Omega), \\[10pt]
+\infty & \text{ elsewhere in } L^1(\Omega),
\end{cases} 
\]
is equal to 
\[
\E_0(u) := \begin{cases}
c_\Psi \P_\Omega (\{u = 1\}) & \text{ if } u \in \BV(\Omega, \{-1,1\}), \\[10pt]
+\infty & \text{ elsewhere in } L^1(\Omega),
\end{cases}
\]
with the constant $c_\Psi :=  \int_{-1}^1 \sqrt{2\Psi(s)} ds$.

\begin{remark}
For $u \in \BV(\Omega, \{-1,1\})$, setting $A := \{u = 1\}$ leads to the relation $u(\x) = 2 \chi_{A}(\x) - 1$, and hence 
\[
|\D u|(\Omega) = 2|\D \chi_A|(\Omega) = 2 \P_\Omega(\{u = 1\}).
\]
\end{remark}

Of particular interest is the following property of $\Gamma$-convergence, which states that if (i) $\E_0$ is the $\Gamma$-limit of $\E_\eps$, (ii) $u_\eps$ is a minimizer to $\E_\eps$ for every $\eps > 0$, (iii) $\{u_\eps\}_{\eps > 0}$ is a precompact sequence, then every limit of a subsequence of $\{u_\eps\}_{\eps > 0}$ is a minimizer for $\E_0$. This provides a methodology to construct minimizers of $\E_0$ as limits of minimizers to $\E_\eps$, provided the associated $\Gamma$-limit is precisely $\E_0$.

Returning to our discussion and problem in additive manufacturing, in \cite{AllaireDEFM17} it was proposed to use anisotropic perimeter functionals to model the overhang angle constraint. In our notation, for a set $U$ of bounded variation with $\bnu_U$ as the measure theoretical inward unit normal, these functionals take the form
\begin{align}\label{anis:peri}
\P_\gamma(U) := \int_{\pd^*U} \gamma(\bnu_U) \dH
\end{align}
with a $C^1$ function $\gamma: \R^d \to \R$. Two choices were suggested in \cite{AllaireDEFM17}:
\begin{align*}
\gamma_a(\bnu) := [\min(0,\bnu \cdot \bm{e}_d + \cos \psi )]^2, \quad \gamma_b(\bnu) := \prod_{i=1}^m (\bnu - \bnu_{\psi_i})^2,
\end{align*}
where $\psi$ is a fixed angle threshold, $\bm{e}_d$ denotes the build direction, and $\psi_i:\R^d \to \R$, for $i = 1, \dots, m$, are given pattern functions with $\bnu_{\psi_i} := \nabla \psi_i/|\nabla \psi_i|$. The first choice $\gamma_a$ penalizes the regions of the boundary $\pd^* U$ where the angle between the outward normal $(-\bnu_U)$ and the negative build direction $(-\bm{e}_d)$ is smaller than $\psi$, while the second choice $\gamma_b$ compels the unit normal $\bnu_U$ to be close to at least one of the directions $\bnu_{\psi_i}$.

A phase field approximation of the anisotropic perimeter functional \eqref{anis:peri} is the following anisotropic Ginzburg--Landau functional
\begin{align}\label{aniso:GL}
E_{\gamma,\eps}(\vp) := \int_\Omega \frac{\eps}{2} |\gamma(\nabla \vp)|^2 + \frac{1}{\eps} \Psi(\vp) \dx,
\end{align}
where as before, $\Psi$ is a double well potential with $\pm 1$ as its minima. Then, for a convex $\gamma: \R^d \to \R$ that is positively homogeneous of degree one (see \eqref{gamma:1}), one has the analogue of the result by Modica and Mortola for anisotropic energies (see \cite{Barroso,Bellettini,Bouchitte,Owen}, and also Lemma \ref{lem:Gamma} below), that is
\begin{align*}
\Gamma-\lim_{\eps \to 0} \E_{\gamma,\eps}(u) = \E_{\gamma,0}(u),
\end{align*}
where the extended functionals \an{$\E_{\gamma,\eps}$ and $\E_{\gamma,0}$} are defined as
\begin{subequations}
\begin{alignat}{2}
\label{E:gam:eps} \E_{\gamma,\eps}(u) & := \begin{cases}
\displaystyle \int_\Omega \frac{\eps}{2} |\gamma(\nabla u)|^2 + \frac{1}{\eps} \Psi(u) \dx & \text{ if } u \in H^1(\Omega), \\[10pt]
+\infty & \text{ elsewhere in } L^1(\Omega),
\end{cases} \\
\label{E:gam:0} \E_{\gamma,0}(u) & := \begin{cases}
c_{\Psi} P_\gamma(\{u = 1\}) & \text{ if } u \in \BV(\Omega, \{-1,1\}), \\[10pt]
+\infty & \text{ elsewhere in } L^1(\Omega).
\end{cases}
\end{alignat}
\end{subequations}
This motivates our consideration of the objective functional \eqref{obj} and of the study of the related minimization problem for the overhang angle constraint. Furthermore, let us formally state the corresponding {\it sharp interface limit} $(\eps \to 0)$ of the structural optimization problem as:
\begin{align}\tag{$\mathbf{P}_0$}\label{opt:SI}
	\hbox{Minimize $J_0(\vp, \u)$ subject to $\vp \in \BV_m(\Omega, \{-1,1\})$ and $\u$ solving \eqref{state}, }
\end{align}
where $\BV_m(\Omega, \{-1,1\}) = \{ f \in \BV(\Omega, \{-1,1\}) \, : \, \int_\Omega f \dx = m |\Omega| \}$ and
\[
J_0(\vp,\u) = \an{\widehat\alpha} c_{\Psi} {\cal P}_\gamma(\{\vp=1\}) + \beta \Big ( \int_\Omega \hh(\vp) \f \cdot \u \dx + \int_{\Gamma_g} \g \cdot \u \dH \Big ).
\]
Note that by Lemma \ref{lem:state}, the solution operator $\So : \vp \mapsto \u$ is well-defined for $\vp \in \BV(\Omega, \{-1,1\})$. The connection between \eqref{opt} and \eqref{opt:SI} will be explored in Section \ref{sec:SI}.

\subsection{Anisotropy function, Wulff shape and Frank diagram}
Consider an anisotropic density function $\gamma :\R^d \to \R$ satisfying
\begin{enumerate}[label={\bf (A\arabic{*})}, ref={\bf A\arabic{*}}]
\item \label{gamma:1}
	$\gamma$ is positively homogeneous of degree one:
\begin{align*}
	\gamma(\lambda \q) = \lambda \gamma(\q) \quad \text{ for all } \q \in \R^d \setminus \{\0\}, \, \lambda \geq  0,
\end{align*}
which immediately implies $\gamma(\0) = 0$.
\item \label{gamma:2}
	 $\gamma$ is positive for non-zero vectors:
\begin{align*}
	\gamma(\q) >0 \quad  \text{ for all } \q \in \R^d  \setminus \{\0\}.
\end{align*}
\item \label{gamma:3}
	 $\gamma$ is convex:
\[
\gamma(s \p + (1-s) \q) \leq s \gamma(\p) + (1-s) \gamma(\q) \quad \text{ for all } \p, \q \in \R^d, \, s \in [0,1].
 \]
\end{enumerate}
Note that it is sufficient to assign values of $\gamma$ on the unit sphere $\partial B_1(\0)$ in $\R^d$, since by the one-homogeneity property \eqref{gamma:1} we can define for any $\p \neq \0$ with $\hat{\p} = \p/ |\p|$
\[
\gamma(\p) := |\p| \gamma \big ( \hat{\p} \big ).
\]
A consequence of the convexity assumption \eqref{gamma:3} is that $\gamma$ is continuous (in fact it is even locally Lipschitz continuous, see, e.g., \cite[E4.6, p.~129]{Alt}).
Then, from \eqref{gamma:2} we have that $\gamma$ has a positive minimum $\tilde{c}$ on the compact set $\partial B_1(\0)$, and consequently by \eqref{gamma:1}, $\gamma(\q)=\gamma\big(|\q| \hat{\q} \big)=|\q|\gamma\big( \hat{\q}\big)\geq \tilde{c}|\q|$ for $\q\neq \0$. Thus, \eqref{gamma:1}--\eqref{gamma:3} yield the following property:
\begin{enumerate}[label={\bf (A\arabic{*})}, ref={\bf A\arabic{*}}, start=4]
\item \label{gamma:4} $\gamma$ is Lipschitz continuous and there exists a constant $\tilde c > 0$ such that 
\begin{align*}
	\gamma(\q) \geq  \tilde c |\q| \quad \text{ for all }  \q \in \R^d.
\end{align*}
\end{enumerate}
  Moreover, with $\lambda = 2$ in \eqref{gamma:1} and $s = \frac{1}{2}$ in \eqref{gamma:3}, it is not difficult to verify that $\gamma$ satisfies the triangle inequality.  Hence, provided that $\gamma(\q)=\gamma(-\q)$ for all $\q \in \R^d$, any $\gamma$ satisfying \eqref{gamma:1}--\eqref{gamma:3} defines a norm on $\R^d$. 

For such anisotropy density functions and smooth hypersurfaces $\Gamma$ with normal vector field $\bnu$, we define the anisotropic interfacial energy as
\begin{align}\label{aniso}
\F^\gamma(\Gamma) := \int_\Gamma \gamma(\bnu) \dH.
\end{align}
Then, the isoperimetric problem involves finding a hypersurface $\Gamma^*$ that minimizes $\F^\gamma$ under a volume constraint.  This problem has been well-studied by many authors (see for instance \cite{Fonseca1,Fonseca2,Taylor1,Taylor2} and \cite{Gurtin} and the references therein) and the solution is given as the boundary of the region called the Wulff shape \cite{Wulff}
\begin{align}
	\label{Wulff}
	W = \{ \rr \in \R^d \,:\, \gamma^*(\rr) \leq 1 \},
\end{align}
where $\gamma^*$ is the dual function of $\gamma$ defined as
\begin{align}
	\label{dualnorm}
	\gamma^*(\rr) = \sup_{\q \in \R^d \setminus \{\0\}} \frac {\rr \cdot \q}{\gamma(\q)} 
	\quad 
	\text{ for all } \rr \in \R^d.
\end{align}
The dual function $\gamma^*$ also satisfies \eqref{gamma:1}--\eqref{gamma:3} and hence we can view the Wulff shape $W$ as the 1-ball of $\gamma^*$.  Besides the Wulff shape, another region of interest that is used to visualize the effects of the anisotropy is the Frank diagram \cite{Frank}, which is  defined as the 1-ball of $\gamma$:
\begin{align}
	\label{Frank}
	F = \{ \q \in \R^d \,:\, \gamma(\q) \leq 1 \},
\end{align}
which, due to \eqref{gamma:3}, is always a convex subset of $\R^d$. To see how the shape of the boundary of $F$ determines which directions of unit sphere in $\R^d$ is preferred by the anisotropic density function $\gamma$, let us consider three examples.

\begin{ex}[Isotropic case]
Consider $\gamma(\q) = |\q|$ for $\q \in \R^d$.  Then, the associated Frank diagram is just the unit ball in $\R^d$, with boundary $\{\gamma(\q) = |\q| = 1\}$. As all points on the boundary are equidistant to the origin, all directions of the unit sphere in $\R^d$ are equally preferable. 
\end{ex}

\begin{ex}[Convex example]\label{eg:Frank1}
Consider the Frank diagram shown in the left of Figure \ref{fig:Frankeg}, whose boundary is composed of a circular arc $C$ and a horizontal line $L$.  The black dot denotes the origin in $\R^2$.  For any unit vector in $\R^2$, we denote by $\phi \in [0,2\pi)$ the angle it makes with the negative $y$-axis measured anticlockwise (see also Figure \ref{fig:overhang_def}).  Then, there exists $\theta > 0$ such that all unit vectors with angle in $[0,\theta] \cup [2\pi - \theta, 2\pi)$ are associated with the horizontal line $L$ in Figure \ref{fig:Frankeg}, while all unit vectors with angle in $(\theta, 2 \pi - \theta)$ are associated with the circular arc $C$. Notice, if the origin lies on $L$, then $\theta = \frac{\pi}{2}$, and $C$ is the upper semicircle.

Let $\p \in C$ and $\q \in L$ be arbitrary, and set $\hat \p = \frac{\p}{|\p|}$, $\hat \q = \frac{\q}{|\q|}$ as their unit vectors. It is clear from the figure that $|\p| \geq |\q|$, and since $\gamma(\p) = \gamma(\q) = 1$, by \eqref{gamma:1} we see that
\begin{align*}
1 = |\p| \gamma \big ( \hat \p \big ) = |\q| \gamma \big ( \hat \q \big ).
\end{align*}
This implies that $\gamma(\hat \p) \leq \gamma (\hat \q)$, and from the viewpoint of minimizing the interfacial energy $\F^\gamma$ in \eqref{aniso}, directions $\hat \p$ are preferable to directions $\hat \q$. Consequently, from the Frank diagram in Figure \ref{fig:Frankeg} we can see that the associated anisotropy density function $\gamma$ prefers directions with angle in $(\theta, 2\pi - \theta)$ over directions with angle in $[0,\theta] \cup [2 \pi - \theta, 2 \pi)$.
\begin{figure}[h]
\centering
\includegraphics[width=0.85\textwidth]{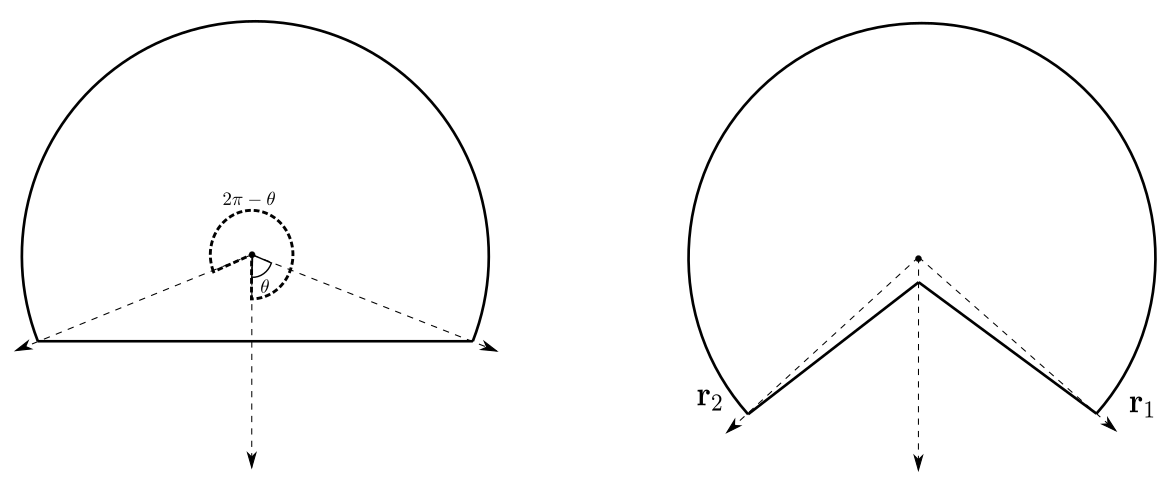}
\caption{Illustrations of the Frank diagram for Examples \ref{eg:Frank1} and \ref{eg:Frank2}. (Left) Convex case and (Right) non-convex case.}
\label{fig:Frankeg}
\end{figure}
\end{ex}

\begin{ex}[Non-convex example]\label{eg:Frank2}
Consider the Frank diagram shown in the right of Figure \ref{fig:Frankeg}, whose boundary encloses a non-convex set. Let $\bm{r}_1$ and $\bm{r}_2$ denote the two unit vectors whose angles, say $\theta$ and $2 \pi - \theta$, respectively, associate to the two endpoints of the circular arc. From previous discussions, the anisotropy density function $\gamma$ will prefer directions with angles in $(\theta, 2 \pi - \theta)$. 

With such $\gamma$, consider two spatial points $\bm{x}_1$ and $\bm{x}_2$ at the same height, see Figure \ref{fig:drippingFrank}. Connecting them via a horizontal straight line is energetically expensive since this is associated to a direction with angle zero (where on the boundary of the Frank diagram is closest to the origin). An energetically more favorable connection is a zigzag path from $\bm{x}_1$ and $\bm{x}_2$ whose normal vectors oscillate between $\bm{r}_1$ and $\bm{r}_2$. This is similar to a behavior termed ``dripping effect'' in \cite{AllaireDEFM17} (see also \cite[Fig.~15]{Qian}), which is the tendency for shapes to develop oscillatory boundaries in order to meet the overhang angle constraints.
\begin{figure}[h]
\centering
\includegraphics[width=0.4\textwidth]{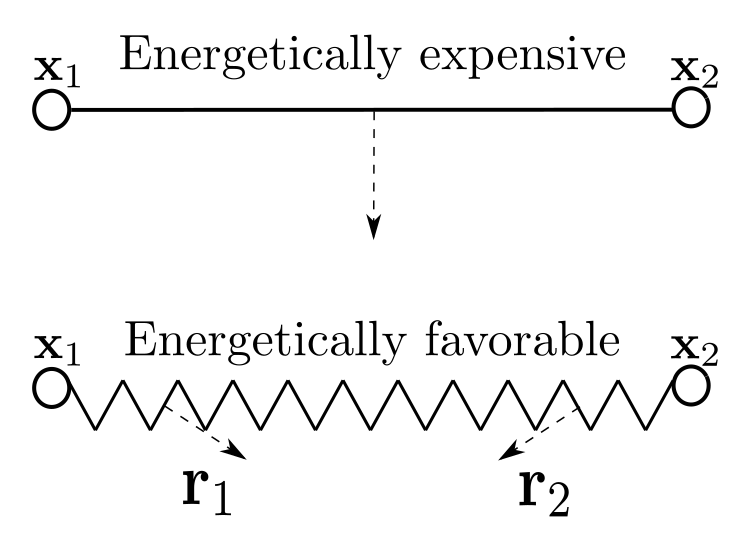}
\caption{Connection between two spatial points with non-convex anisotropic density function $\gamma$ whose Frank diagram looks like the right of Figure \ref{fig:Frankeg}.}
\label{fig:drippingFrank}
\end{figure}
\end{ex}

\subsection{Relevant examples of anisotropic density function}
Motivated by the above discussion, in this section we provide some examples of $\gamma$ that achieve the Frank diagrams shown in Figure \ref{fig:Frankeg}.

\subsubsection{An example of a convex anisotropy}
\an{To fix the ideas, let us begin with the two-dimensional case.}
For a fixed constant $\alpha \in (0,1)$, we consider the function
\begin{align}\label{2D:gam}
\gamma_{\alpha}(\x) = \begin{cases}
 \sqrt{x_1^2+x_2^2} & \text{ if } \x = (x_1,x_2) \in V, \\[10pt]
\displaystyle -\frac{1}{\alpha} x_2 & \text{ if } (x_1,x_2) \in \R^2 \setminus \overline{V} =: V^c,
\end{cases}
\end{align}
where the set $V\an{\subseteq\R^2}$ will be determined in the following.  To achieve the boundary of the Frank diagram $F$ shown in the left of Figure \ref{fig:Frankeg}, we notice that $\pd F \cap V$ is a circular arc of radius $1$ centered at the origin, while $\pd F \cap V^c$ implies $x_2 = - \alpha$ which is a horizontal line segment at height $-\alpha$.  Following the convention in Figure \ref{fig:overhang_def} where angles are measured anticlockwise from the negative $x_2$-axis, we can parameterize the circular arc by $(\cos \phi, \sin \phi)$ for $\phi \in [\theta, 2 \pi -\theta]$ where $\theta := \cos^{-1}(\alpha)$, see Figure \ref{fig:Frankeg}. Hence, the set $V$ in the definition \eqref{2D:gam} can be characterized as
\begin{align*}
V = \Big \{ \cos \phi \leq \alpha \, : \, \phi \in [0,2\pi] \Big \}.
\end{align*}
\begin{remark}\label{rem:iso}
Note that in the limiting case $\alpha=1$, the above set $V$ is the entire plane $\R^2$. Hence we have the isotropic case $\gamma(\x) = |\x|$ for all $\x \in \R^2$, and $\pd F$ is simply the unit circle. 
\end{remark}
For a parametric characterization of the set $V$, let $c := \frac{\alpha}{\sqrt{1-\alpha^2}}$ and consider the two straight lines $\{x_2 = c x_1\}$ and $\{x_2 = -c x_1\}$ dividing $\R^2$ into eight regions (see Figure \ref{fig:gamEg2D}), which we label as Region $1, 2, \dots, 8$ in an anticlockwise direction starting from the positive $x_1$-axis.  Then, the horizontal straight line portion $\pd F \cap V^c$ of the Frank diagram at height $x_2 = - \alpha$ is contained in Regions 6 and 7, whose union is described by the set $\{x_2 < -c |x_1|\}$, while the circular arc portion $\pd F \cap V$ is contained in Regions $1, \dots, 5$ and $8$, whose union is described by the set $\{x_2 \geq -c |x_1|\}$.  Hence, a parameteric characterization of the set $V$ in \eqref{2D:gam} is
\begin{align*}
V = \Big \{ x_2 \geq - \frac{\alpha}{\sqrt{1-\alpha^2}} |x_1| \Big \} = \Big \{ x_2 \geq - \alpha | \x| \Big \}.
\end{align*}
Generalizing to the $d$-dimensional case, we obtain 
\begin{align}\label{D:gam}
 \gamma_{\alpha}(\x) = \begin{cases}
 |\x| & \text{ if } \displaystyle x_d \geq -\alpha |\x| \\[10pt]
\displaystyle -\frac{1}{\alpha} x_d & \text{ if } \displaystyle x_d < - \alpha |\x|
\end{cases} \text{ for } \x = (x_1,\dots , x_d) \in \R^d.
\end{align}
From \eqref{D:gam} we see that $\gamma_{\alpha} \in C^0(\R^d)$ but it is not continuously differentiable at the points $x_d = - \alpha |\x|$. Furthermore, it is clear that $\gamma_{\alpha}$ satisfies the assumptions \eqref{gamma:1}--\eqref{gamma:3} and hence \eqref{gamma:4}.
\begin{figure}[h]
\centering
\includegraphics[width=0.4\textwidth]{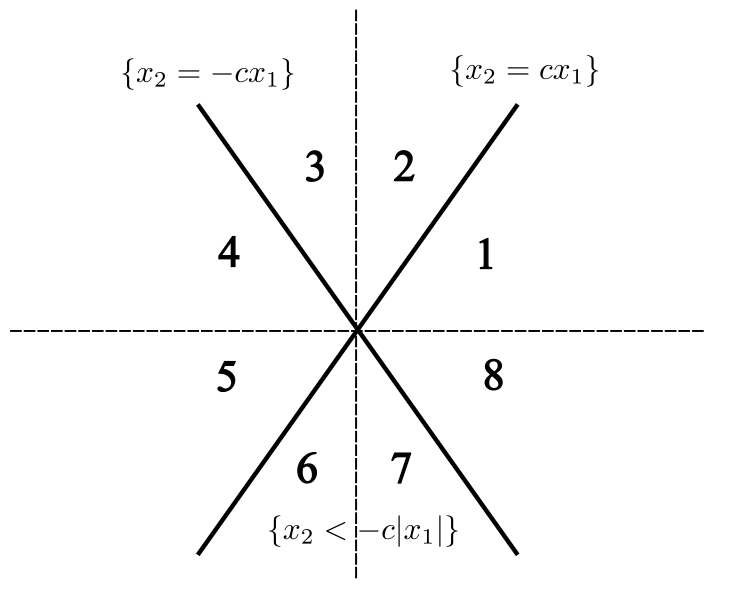}
\caption{Schematics for the parametric characterization.}
\label{fig:gamEg2D}
\end{figure}

\subsubsection{An example of a non-convex anisotropy}
For completeness, we provide an example of an anisotropic function $\gamma$ that yields a non-convex Frank diagram as seen in the right of Figure \ref{fig:Frankeg}.  Referring to the construction of the previous example, we present the two-dimensional case first. Fix $\lambda \in (0,1)$ and let $(0,-\lambda \alpha)^{\top}$ denote the intersection of the two straight line segments (which we call $L_1$ and $L_2$ respectively) just below the origin. Denoting by $\bm{p}_-$ the endpoint of the left line segment $L_1$ and by $\bm{p}_{+}$ the endpoint of the right line segment $L_2$ that connects $(0,-\lambda \alpha)^{\top}$ to the circular arc of radius $1$, a short calculation shows that $\bm{p}_{\pm} = (\pm \sqrt{1-\alpha^2}, -\alpha)^{\top}$.

A choice of tangent vector for $L_1$ is $\bm{\tau} = (-\sqrt{1-\alpha^2}, (\lambda - 1)\alpha)^{\top}$ so that a normal vector for $L_1$ is $\bm{n} = ((1-\lambda)\alpha, -\sqrt{1-\alpha^2})^{\top}$. Consider, for some constant $b>0$ to be identified,
\[
\gamma(\x) = b \big |(1-\lambda)\alpha x_1 - x_2 \sqrt{1-\alpha^2}  \big |
\]
for $\x = (x_1, x_2) \in \{x_1 \leq 0, x_2 \leq \frac{\alpha}{\sqrt{1-\alpha^2}}x_1\}$. This corresponds to Region 6 in the right of Figure \ref{fig:gamEg2D} that contains the line segment $L_1$, which can be parameterized as 
\begin{align*}
L_1 = \big \{(-\zeta \sqrt{1-\alpha^2}, -\lambda \alpha - \zeta(1-\lambda)\alpha)^{\top} \, : \, \zeta \in [0,1] \big \}. 
\end{align*}
Notice that $\gamma$ satisfies \eqref{gamma:1}--\eqref{gamma:2} (due to the modulus), and a short calculation shows that for $\x \in L_1$,
\[
\gamma(\x) = b \big |\lambda \alpha \sqrt{1-\alpha^2} \big | = b \lambda \alpha \sqrt{1-\alpha^2}.
\]
Hence, choosing $b = (\lambda \alpha \sqrt{1-\alpha^2})^{-1} > 0$ yields that $\gamma(\x) = 1$ for $\x \in L_1$. Similarly, a choice of tangent vector for $L_2$ is $\bm{\tau} = (\sqrt{1-\alpha^2}, (\lambda - 1)\alpha)^{\top}$, so that a normal vector for $L_2$ is $\bm{n} = ((\lambda - 1)\alpha, -\sqrt{1-\alpha^2})^{\top}$. We consider
\begin{align*}
\gamma(\x) & = \frac{1}{\lambda \alpha \sqrt{1-\alpha^2}} \big | (\lambda - 1) \alpha x_1 - x_2 \sqrt{1-\alpha^2} \big | \\
& = \frac{1}{\lambda \alpha \sqrt{1-\alpha^2}} \big | (1-\lambda) \alpha x_1 + x_2 \sqrt{1-\alpha^2} \big | 
\end{align*}
for $\x \in \{x_1 \geq 0, x_2 \leq - \frac{\alpha}{\sqrt{1-\alpha^2}} x_1\}$ which corresponds to Region 7 in Figure \ref{fig:gamEg2D} that contains the line segment $L_2$. Then, a short calculation shows that $\gamma(\x) = 1$ for $\x \in L_2$. Thus, an example of an anisotropic function $\gamma$ that give rise to a Frank diagram whose boundary is the \an{right} figure in Figure \ref{fig:Frankeg} is
\begin{align}\label{2D:gam:Ncon}
\gamma_{\alpha, \lambda}(\x) =\begin{cases}
 |\x| & \text{ if } \displaystyle x_2 \geq - \alpha |\x|, \\[10pt]
\displaystyle \Big | \frac{1-\lambda}{\lambda \sqrt{1-\alpha^2}} x_1 - \frac{1}{\lambda \alpha} x_2 \Big |& \text{ if } \displaystyle x_1 \leq 0, \, x_2 <- \alpha |\x|, \\[10pt]
\displaystyle \Big | \frac{1-\lambda}{\lambda \sqrt{1-\alpha^2}} x_1 + \frac{1}{\lambda \alpha} x_2 \Big |& \text{ if } \displaystyle x_1 > 0, \, x_2 < -\alpha |\x|,
\end{cases}
\end{align}
for $\alpha \in (0,1)$, $\lambda \in (0,1)$ and $\x = (x_1, x_2) \in \R^2$. Notice that in the limit $\lambda \to 1$, we recover the convex anisotropic function $\gamma_{\alpha}$ defined in \eqref{D:gam}. 

To generalize to the $d$-dimensional case, we notice that the lines $L_1$ and $L_2$ in the above discussion are now replaced by the lateral surface $S$ of a cone with apex $(0,0,\dots, 0,-\lambda \alpha) \in \R^d$, which can be parameterized as 
\[
S = \left \{\x = (\tilde{\x}, x_d) \in \R^d \, : \, x_d + \lambda \alpha = - \frac{(1-\lambda)\alpha}{\sqrt{R^2-\alpha^2}} |\tilde{\x}|, \, x_d \in [-\alpha,-\lambda \alpha] \right \}.
\]
Then, by similar arguments leading to \eqref{2D:gam:Ncon}, we obtain the function
\begin{align*}
\gamma_{\alpha,\lambda} (\x)= \begin{cases}
\displaystyle |\x| & \text{ if } \displaystyle x_d \geq - \alpha |\x|, \\[10pt]
\displaystyle \Big | \frac{1-\lambda}{\lambda \sqrt{R^2-\alpha^2}} |\tilde{\x}| + \frac{1}{\lambda \alpha} x_d \Big |& \text{ if } \displaystyle x_d <- \alpha |\x|,
\end{cases}
\end{align*}
for $\x = (\tilde{\x}, x_d) \in \R^d$, $\tilde{\x} \in \R^{d-1}$, where we can verify that $\gamma_{\alpha,\lambda}(\x) = 1$ for $\x \in S$. In Figure \ref{fig:frank_nc} we display the Frank diagrams for non-convex anisotropy functions of the form \eqref{2D:gam:Ncon} with $\lambda = 0.5$ and $\alpha \an{= 0.7, 0.5, 0.3}$.

\begin{figure}[h]
\centering
\includegraphics[angle=-0,width=0.3\textwidth]{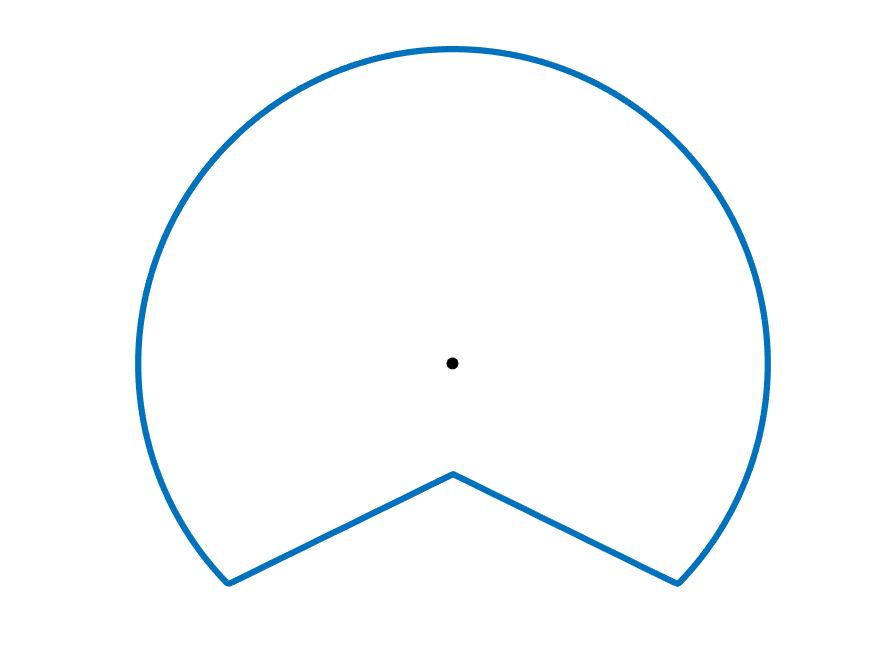} 
\includegraphics[angle=-0,width=0.3\textwidth]{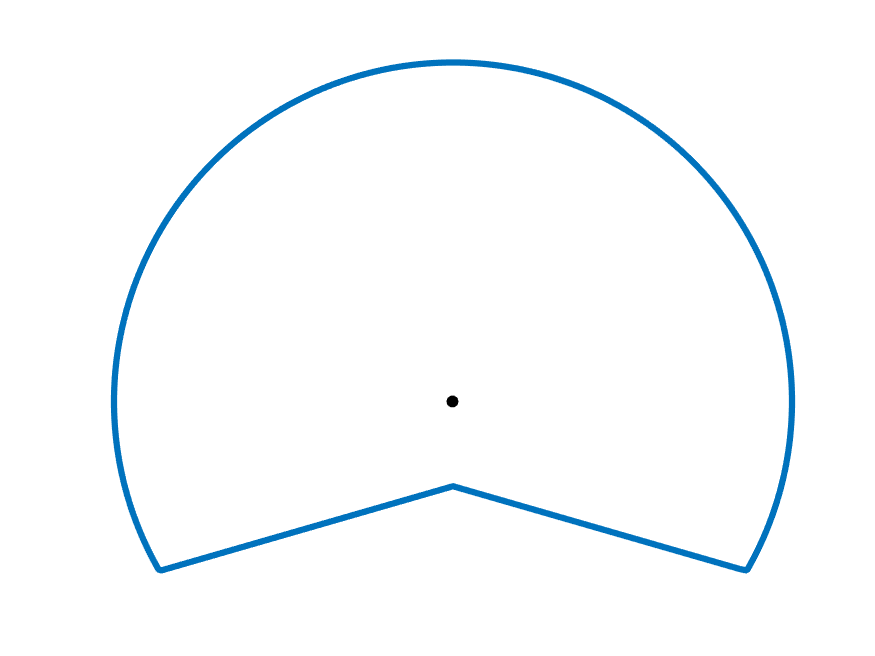}
\includegraphics[angle=-0,width=0.3\textwidth]{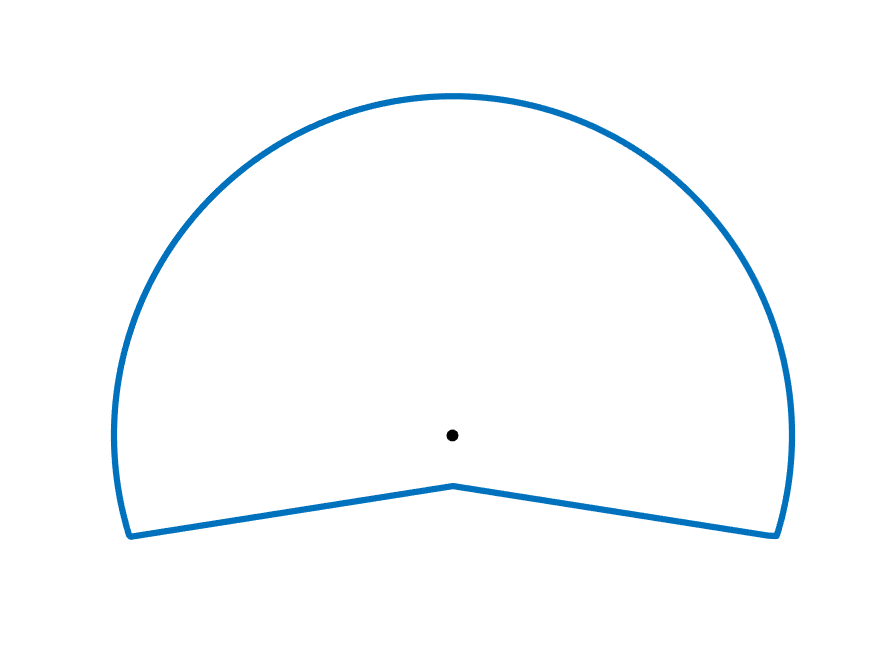}
\caption{Frank diagrams for non-convex anisotropy \eqref{2D:gam:Ncon} with $\lambda=0.5$. From left to right: $\alpha = 0.7$, $0.5$ and $0.3$.
}
\label{fig:frank_nc}
\end{figure}

\subsection{Subdifferential characterization}
The analysis of the structural optimization problem \eqref{opt} follows almost analogously as in \cite{BGFS}. The major difference is the anisotropic Ginzburg--Landau functional in \eqref{obj}. From the examples of $\gamma$ discussed in the previous subsection, our interest lies in anisotropy density functions that are convex and continuous, but not necessarily $C^1(\R^d)$, which necessitates a non-trivial modification to the analysis performed in \cite{BGFS}. In this subsection we focus only on the gradient part of \eqref{obj} and investigate its subdifferential in preparation for the first-order necessary optimality conditions for \eqref{opt} (cf.~Theorem \an{\ref{thm:optcond}}).

We define $A: \R^d \to [0,\infty)$ as
\begin{align}\label{defn:A}
	A(\q) = \frac{1}{2} |\gamma(\q)|^2 \quad  \text{ for all } \q \in \R^d,
\end{align}
and from \eqref{gamma:1}--\eqref{gamma:4}, it readily follows that $A$ is convex, continuous with $A(\0) = 0$, positive for non-zero vectors, positively homogeneous of degree two:
\[
A(\lambda \q) = \lambda^2 A(\q) \quad \text{ for all } \q \in \R^d \setminus \{\0\}, \quad \lambda \geq 0,
\] 
and there exist positive constants $c$ and $C$ such that
\begin{align}\label{prop:A}
A(\q) \geq c| \q|^2, \quad |A(\q) - A(\rr)| \leq C |\q - \rr| (|\q| + |\rr|) \quad \text{ for all } \q, \rr \in \R^d.
\end{align}
Associated to such a function $A$, we consider the integral functional 
\begin{align}\label{def:F}
\F: H^1(\Omega) \to \R, \quad \vp \mapsto \int_\Omega A(\nabla \vp) \dx.
\end{align}
Convexity of $A$ immediately imply that $\F$ is convex, proper and weakly lower semicontinuous in $H^1(\Omega)$, see, e.g., \cite[Thm.~1, \S 8.2.2]{Evans}.  Consequently, its subdifferential $\pd \F : H^1(\Omega) \to 2^{H^1(\Omega)^*}$, defined as
\begin{align*}
\pd \F(\vp) = \Big \{ \xi \in H^1(\Omega)^* \, : \, \F(\phi) - \F(\vp) \geq \inn{\xi,\phi - \vp}_{H^1}\, \; \forall \phi \in H^1(\Omega) \Big \},
\end{align*}
is a maximal monotone operator from $H^1(\Omega)$ to $H^1(\Omega)^*$ (see \cite[Thm.~2.43]{Barbu2}), which is equivalent to the property that, for any $f \in H^1(\Omega)^*$, there exists at least one solution $\vp_0 \in D(\pd \F)$ with $\xi_0 \in \pd \F(\vp_0)$ such that
\begin{align*}
\vp_0 + \xi_0 = f \quad \text{ in } H^1(\Omega)^*.
\end{align*}
Let us now provide a useful lemma that characterizes the subdifferential $\pd \F(\vp)$ in terms of elements of $\pd A(\nabla \vp)$ (see also \cite[p.~146, Problem 2.7]{Barbu2}). Heuristically, elements of $\pd \F(\vp)$ are the negative weak divergences of elements of $\pd A(\nabla \vp)$.

\begin{lem}\label{LEM:A}
Consider the map $\G : H^1(\Omega) \to 2^{H^1(\Omega)^*}$ defined by
\begin{align*}
\vp \mapsto \Big \{ v \in H^1(\Omega)^* \, : \, & \exists\, \bxi \in L^2(\Omega, \R^d) \text{ s.t. } \bxi \in \pd A(\nabla \vp) \text{ a.e. in } \Omega, \\
& \text{ and } \inn{v,p}_{H^1} = (\bxi, \nabla p) \; \forall p \in H^1(\Omega) \Big \}.
\end{align*}
Then, for $\vp \in D(\pd \F)$, it holds that
\begin{align*}
\pd \F(\vp) = \G(\vp).
\end{align*}
\end{lem}

\begin{remark}\label{rmk:Evans}
If $A :\R^d \to [0,\infty)$ is smooth, convex, coercive with bounded second derivatives, i.e., there exists $C> 0$ such that $|A_{p_i,p_j}(\p)| \leq C$ for all $\p \in \R^d$ and $1 \leq i,j \leq d$, then the corresponding characterization of the subdifferential of $\F$ can be found in \cite[Thm.~4, \S 9.6.3]{Evans}. Namely, $\pd \F(\vp)$ is single-valued and $\pd \F(\vp) = \{- \sum_{i=1}^d \pd_{x_i} [A_{p_i}(\nabla \vp)] \}$.
\end{remark}

\begin{proof}[Proof of Lemma \ref{LEM:A}]
To begin, fix $\vp \in H^1(\Omega)$, $v \in \G(\vp)$ and let $\bxi \in L^2(\Omega, \R^d)$ with $\bxi(\x) \in \pd A(\nabla \vp(\x))$ for almost every $\x \in \Omega$ satisfying $\inn{v,p}_{H^1} = (\bxi, \nabla p)$ for all $p \in H^1(\Omega)$. By the definition of the subgradient $\pd A(\nabla \vp(\x))$, we have
\begin{align*}
A(\q) - A(\nabla \vp(\x)) \geq \bxi(\x) \cdot (\q - \nabla \vp(\x)) \quad \forall \q \in \R^d.
\end{align*}
Choosing $\q = \nabla \phi(\x)$ where $\phi \in H^1(\Omega)$ is arbitrary, and integrating over $\Omega$ we infer that 
\begin{align*}
\F(\phi) - \F(\vp) \geq \int_\Omega \bxi \cdot (\nabla \phi - \nabla \vp) \dx = \inn{v, \phi - \vp}_{H^1} \quad \forall \phi \in H^1(\Omega).
\end{align*}
This shows $v \in \pd \F(\vp)$ and hence the inclusion $\G(\vp) \subset \pd \F(\vp)$. 

For the converse inclusion, we show $\G : H^1(\Omega) \to 2^{H^1(\Omega)^*}$ is a maximal monotone operator, and as $\pd \F$ is maximal monotone we obtain by definition $\G(\vp) = \pd \F(\vp)$ for all $\vp \in D(\pd \F)$.  Thus, it suffices to study the solvability of the following problem: For $f \in H^1(\Omega)^*$, find a pair $(\vp, \bxi) \in H^1(\Omega) \times L^2(\Omega,\R^d)$ such that $\bxi \in \pd A(\nabla \vp)$ almost everywhere in $\Omega$ and
\begin{align}\label{aux:prob}
\vp - \div \bxi = f \quad  \text{ in } H^1(\Omega)^*.
\end{align}
This result can be obtained by following a strategy outlined in \cite[Thm.~2.17]{Barbu1}. 
First, from recalling the properties of $A$ in \eqref{prop:A}, we deduce that for $\bxi   = (\xi_1, \dots, \xi_d)^{\top} \in \pd A(\q)$,
\begin{align}\label{aux:xi:prop}
\bxi  \cdot \q \geq A(\q) \geq c|\q|^2, \quad |\xi_i| \leq C \big ( 1 +  |\q| \big ),
\end{align}
where the first inequality is obtained from applying the relation $A(\rr ) - A(\q) \geq \bxi  \cdot (\rr  - \q)$ with the choice $\rr  = \0$, and the second inequality is obtained from the second property of $A$ in \eqref{prop:A} with the choice $\rr  = \q + {\bm e}_i$ where ${\bm e}_i$ is the canonical $i$-th unit vector in $\R^d$.  Taking the maximum over $i \in \{1, \dots, d\}$ and then the supremum over all elements $\bxi  \in \pd A(\q)$, the second inequality in \eqref{aux:xi:prop} implies
\begin{align}\label{aux:xi:prop:2}
\sup \Big \{ |\bxi | \, : \, \bxi  \in \pd A(\q) \Big \} \leq C \big ( 1 +  |\q| \big ).
\end{align}
Next, for $\lambda > 0$, we introduce the Moreau--Yosida approximation $B_\lambda$ of $B := \pd A$ as
\begin{align*}
B_\lambda({\bm s}) = \frac{1}{\lambda} \Big ({\bm s} - ( I + \lambda B)^{-1} {\bm s} \Big ), \quad {\bm s} \in \R^d,
\end{align*}
where $I$ denotes the identity map. The maximal monotonicity of $B$ guarantees that $(I + \lambda B)^{-1}$ is a well-defined operator and it is bounded independently of $\lambda$ (see \cite[p.~524, Thms.~1 and 2]{Evans}). It is well-known that $B_\lambda$ is single-valued, Lipschitz continuous with Lipschitz constant $\frac{1}{\lambda}$, and $B_\lambda({\bm s})$ is an element of $B(( I + \lambda B)^{-1} {\bm s})$. Then, by \eqref{aux:xi:prop} and \eqref{aux:xi:prop:2} (cf.~\cite[p.~84, (2.143)-(2.144)]{Barbu1}), we infer from the identity $\q = (I + \lambda B)^{-1} \q + \lambda B_\lambda(\q)$ that for all $\q \in \R^d$ and $\lambda \in (0,\lambda_0)$, $\lambda_0 := \frac{1}{c}$ with constant $c$ from \eqref{prop:A},
\begin{subequations}\label{aux:Blam}
\begin{alignat}{2}
|B_\lambda(\q)| & \leq C \big ( 1 + |(I + \lambda B)^{-1} \q| \big ) \leq C \big ( 1 + |\q| \big ), \\
\notag B_\lambda(\q) \cdot \q & = B_\lambda(\q) \cdot (I + \lambda B)^{-1} \q + \lambda |B_\lambda(\q)|^2 \geq c |(I + \lambda B)^{-1} \q|^2 + \lambda |B_\lambda(\q)|^2 \\
& \geq \frac{c}{2} |\q|^2 + \lambda (1 - c \lambda) |B_\lambda(\q)|^2 \geq \frac{c}{2} |\q|^2,
\end{alignat}
\end{subequations}
with constants independent of $\lambda$. We now consider the approximation problem: For $f \in H^1(\Omega)^*$, $\lambda\in (0,\lambda_0),$ find $\vp_\lambda \in H^1(\Omega)$ such that 
\begin{align}\label{aux:weak}
\inn{T_\lambda \vp_\lambda,\zeta}_{H^1} := \int_\Omega \vp_\lambda \zeta + B_\lambda(\nabla \vp_\lambda) \cdot \nabla \zeta + \lambda \nabla \vp_\lambda \cdot \nabla \zeta \dx = \inn{f,\zeta}_{H^1} \quad \forall \zeta \in H^1(\Omega).
\end{align}
It is not difficult to verify that $T_\lambda$ is monotone, demicontinuous (i.e., $\vp_n \to \vp$ in $H^1(\Omega)$ implies $T_\lambda \vp_n \rightharpoonup T_\lambda \vp$ in $H^1(\Omega)^*$ as $n\to\infty$), and coercive (i.e., $\inn{T_\lambda \vp, \vp}_{H^1} \geq c_0 \| \vp \|_{H^1}$ for some $c_0 > 0$ independent of $\lambda$), and so by standard results (see, e.g., \cite[p.~37, Cor.~2.3]{Barbu1}) there exists at least one solution $\vp_\lambda \in H^1(\Omega)$ to $T_\lambda \vp_\lambda = f$ in $H^1(\Omega)^*$.

We then establish uniform estimates for $\{\vp_\lambda\}_{\lambda \in (0,\lambda_0)}$ whose weak limit in $H^1(\Omega)$ will be a solution to \eqref{aux:prob}, thereby verifying the maximal monotonicity of $\G$ and hence completing the proof.  Choosing $\zeta = \vp_\lambda$ in \eqref{aux:weak} and recalling the identity $\q = (I+\lambda B)^{-1} \q + \lambda B_\lambda(\q)$, we obtain
\begin{equation}\label{aux:bdd}
\begin{aligned}
& \| \vp_\lambda \|^2 + \lambda \Big ( \| B_\lambda(\nabla \vp_\lambda) \|^2 +  \| \nabla \vp_\lambda \|^2 \Big ) \\
& \quad + \int_\Omega B_\lambda(\nabla \vp_\lambda) \cdot (I + \lambda B)^{-1}\nabla \vp_\lambda \dx  = \inn{f, \vp_\lambda}_{H^1}.
\end{aligned}
\end{equation}
From \eqref{aux:Blam} we immediately infer
\begin{align*}
\| \vp_\lambda \|^2 + \frac{c}{2} \| \nabla \vp_\lambda \|^2 \leq \| f \|_{H^1(\Omega)^*} \| \vp_\lambda \|_{H^1(\Omega)} \leq C \| f \|_{H^1(\Omega)^*}^2 + \min \big (\tfrac{c}{4}, \tfrac{1}{2} \big ) \| \vp_\lambda \|_{H^1(\Omega)}^2,
\end{align*}
which provides a uniform estimate for $\vp_\lambda$ in $H^1(\Omega)$ with respect to $\lambda$.  On the other hand, also from \eqref{aux:Blam}, particularly $B_\lambda(\q) \cdot \q \geq c |(I +\lambda B)^{-1} \q|^2$, we observe that 
\begin{align*}
\| B_\lambda(\nabla \vp_\lambda) \|^2 \leq C\big ( 1 +  \| (I + \lambda B)^{-1} \nabla \vp_\lambda \|^2 \big ) & \leq C, \\
\| (I + \lambda B_\lambda)^{-1} \nabla \vp_\lambda - \nabla \vp_\lambda \|^2 = \lambda^2 \| B_\lambda(\nabla \vp_\lambda) \|^2& \leq C\lambda^2,
\end{align*}
which implies, along a non-relabeled subsequence $\lambda \to 0$,
\begin{subequations}\label{aux:cmpt}
\begin{alignat}{2}
\vp_\lambda & \rightharpoonup \vp  \quad &&\text{ in } H^1(\Omega), \\
(I + \lambda B)^{-1} \nabla \vp_\lambda & \rightharpoonup \nabla \vp \quad && \text{ in } L^2(\Omega, \R^d), \\
B_\lambda(\nabla \vp_\lambda) & \rightharpoonup \bxi \quad && \text{ in } L^2(\Omega, \R^d),
\end{alignat}
\end{subequations}
for some limit functions $\vp$ and $\bxi$ satisfying
\begin{align*}
\int_\Omega \vp \zeta + \bxi \cdot \nabla \zeta \dx = \inn{f,\zeta}_{H^1} \quad \forall \zeta \in H^1(\Omega).
\end{align*}
To show that $\bxi(\x) \in \pd A(\nabla \vp(\x))$ for almost every $\x \in \Omega$, for arbitrary ${\bm \phi} \in D(\pd A)$ and ${\bm \zeta} \in B({\bm \phi}) \subset L^2(\Omega, \R^d)$, we use the inequality
\begin{align*}
\int_\Omega (B_\lambda(\nabla \vp_\lambda) - {\bm \zeta}) \cdot ( (I + \lambda B)^{-1} \nabla \vp_\lambda -  {\bm \phi}) \dx \geq 0
\end{align*}
obtained from the relation $B_\lambda(\nabla \vp_\lambda) \in B(( I + \lambda B)^{-1} \nabla \vp_\lambda)$ and the monotonicity of $B $.  Then, passing to the limit $\lambda \to 0$ with \eqref{aux:cmpt}, as well as
\begin{align*}
\int_\Omega B_\lambda(\nabla \vp_\lambda) \cdot (I + \lambda B)^{-1}\nabla \vp_\lambda \dx & = \inn{f,\vp_\lambda}_{H^1} - \| \vp_\lambda \|^2 - \lambda \| \nabla \vp_\lambda \|^2 - \lambda \| B_\lambda(\nabla \vp_\lambda)\|^2 \\
& \to \inn{f,\vp}_{H^1} - \| \vp \|^2 = \int_\Omega \bxi \cdot \nabla \vp \dx,
\end{align*}
we arrive at
\begin{align*}
\int_\Omega (\bxi - {\bm \zeta}) \cdot (\nabla \vp - {\bm \phi}) \dx \geq 0.
\end{align*}
Picking ${\bm \phi} = (I+B)^{-1}(\bxi + \nabla \vp)$, so that ${\bm \zeta}:= \bxi + \nabla \vp - {\bm \phi} \in B({\bm \phi})$, we see that the above reduces to 
\begin{align*}
\int_\Omega |\nabla \vp - {\bm \phi}|^2 \dx = 0,
\end{align*}
which implies $\nabla \vp = {\bm \phi}$ and ${\bm \zeta}= \bxi \in B(\nabla \vp) = \pd A(\nabla \vp)$ almost everywhere in $\Omega$. This shows that $(\vp, \bxi) \in H^1(\Omega) \times L^2(\Omega, \R^d)$ is a solution to \eqref{aux:prob} concluding the proof.
\end{proof}

\section{Analysis of the structural optimization problem}\label{sec:ana}
By Lemma \ref{lem:state} the solution operator $\So : L^\infty(\Omega) \to H^1_D(\Omega,\R^d)$, $\So(\vp) = \u$ where $\u$ is the unique solution to \eqref{state} corresponding to the design variable $\vp$, is well-defined. This allows us to consider the reduced functional $\J: \V_m \to \R$,
\[
\J (\phi) := J(\phi, \So(\phi))
\]
in the structural optimization problem \eqref{opt}. Invoking \cite[Thm.~1, \S 8.2.2]{Evans}, the convexity of $\gamma$ and hence of $A(\cdot) = \frac{1}{2} |\gamma(\cdot)|^2$ yields that the gradient term in \eqref{obj} is weakly lower semicontinuous in $H^1(\Omega)$. Then, following the proof of \cite[Thm.~4.1]{BGFS} we obtain the existence of an optimal design to \eqref{opt}.

\begin{thm}\label{thm:min}
Suppose that \eqref{gamma:1}--\eqref{gamma:3} hold, and let $(\f, \g) \in L^2(\Omega, \R^d) \times L^2(\Gamma_g, \R^d)$. Then, \eqref{opt} admits a minimiser $\opt \in \V_m$.
\end{thm}
\begin{proof}
Since the proof is now rather standard, we only sketch some of the essential details. Using the bound \eqref{state:bdd} we infer that the reduced functional $\J$ is bounded from below in $\V_m$. This allows us to consider a minimizing sequence $\{\phi_n\}_{n \in \N} \subset \V_m$ such that $\J(\phi_n) \to \inf_{\zeta \in \V_m} \J(\zeta)$ as $n\to\infty$. Then, again by \eqref{state:bdd} and also \eqref{gamma:4} we deduce that $\{\phi_n\}_{n \in \N}$ is uniformly bounded in $H^1(\Omega)\cap L^\infty(\Omega)$, from which we extract a non-relabeled subsequence converging weakly to $\opt$ in $H^1(\Omega)$. As $\V_m$ is a \an{convex and closed set, hence weakly sequentially closed,} we infer also $\opt \in \V_m$ and by passing to the limit infimum of $\J(\phi_n)$, employing the weak lower semicontinuity of the gradient term in \eqref{obj} and also the weak convergence $\So(\phi_n) \to \So(\opt)$ in $H^1(\Omega, \R^d)$ we infer that $\J(\opt) = \inf_{\zeta \in \V_m} \J(\zeta)$, which implies that $\opt$ is a solution to \eqref{opt}.
\end{proof}

To derive the first-order optimality conditions, we first take note of the following result concerning the Fr\'echet differentiability of the solution operator $\So$, which can be inferred from \cite[Thm.~3.3]{BGFS}:
\begin{lem}\label{lem:diff}
Let $(\f, \g) \in L^2(\Omega, \R^d) \times L^2(\Gamma_g, \R^d) $ and $\opt \in L^\infty(\Omega)$. Then $\So$ is Fr\'echet differentiable at $\opt$ as a mapping from $L^\infty(\Omega)$ to $\HD$. Moreover,
for every $\zeta \in L^\infty(\Omega)$, $\So'(\opt) \in \mathcal{L}(L^\infty(\Omega); \HD)$ and 
\begin{align*}
\So'(\opt) [\zeta] = \w ,
\end{align*}
where $\w \in \HD$ is the unique solution to the linearized system
\begin{align}\label{lin:weak}
	\int_{\Omega} \C(\opt) \E(\w) : \E(\be) \dx & = - \int_\Omega \C'(\opt)\zeta \E(\uopt) : \E(\be) \dx
	+ \int_\Omega \hh'(\opt) \zeta \f \cdot \be \dx
\end{align}
for all $\be \in \HD,$ with $\uopt = \So(\opt)$.
\end{lem}
Then, we introduce the adjoint system for the adjoint variable $\p$ associated to $\opt$, whose structure is similar to the state equation \eqref{state} for $\u$:
\begin{subequations}\label{adjoint}
\begin{alignat}{2}
-\div ( \C(\opt) \E(\p)) &= \beta \hh(\opt) \f \quad && \text{ in } \Omega, \\
\p &= \0 \quad && \text{ on } \Gamma_D, \\
(\C(\opt) \E(\p)) \bm{n} &= \beta \g \quad && \text{ on } \Gamma_g, \\
(\C(\opt) \E(\p)) \bm{n} &= \0 \quad && \text{ on } \Gamma_0.
\end{alignat}
\end{subequations}
Via a similar argument (see also \cite[Thm.~4.3]{BGFS}), the well-posedness of the adjoint system is straightforward:
\begin{lem}\label{lem:adj}
For any $\opt \in L^\infty(\Omega)$ and $(\f, \g) \in L^2(\Omega, \R^d) \times L^2(\Gamma_g, \R^d)$, there exists a unique solution $\p \in H^1_D(\Omega, \R^d)$ to \eqref{adjoint} satisfying
\begin{align}\label{adj:weak}
\int_\Omega \C(\opt) \E(\p) : \E(\bv) \dx = \beta \int_\Omega \hh(\opt) \f \cdot \bv \dx + \beta \int_{\Gamma_g} \g \cdot \bv \dH
\end{align}
for all $\bv \in \HD$. In fact, if $\uopt = \So(\opt) \in H^1_D(\Omega, \R^d)$ is the unique solution to \an{\eqref{state:weak}}, then $\p = \beta \uopt$.
\end{lem}

Our main result is the following first-order optimality conditions for the structural optimization problem with anisotropy \eqref{opt}:
\begin{thm}\label{thm:optcond}
Suppose \eqref{gamma:1}--\eqref{gamma:3} hold. Let $(\f, \g) \in L^2(\Omega, \R^d) \times L^2(\Gamma_g,\R^d)$ and $\opt \in \V_m$ be an optimal design variable with associated state $\uopt = \So(\opt)$.  Then, there exists $\bxi \in \pd A(\nabla \opt)$ almost everywhere in $\Omega$ such that 
\begin{equation}\label{optcond}
\begin{aligned}
&\widehat \alpha \int_\Omega \eps \bxi \cdot \nabla (\phi-\opt) + \frac 1 \eps \Psi'_0(\opt)(\phi-\opt) \dx \\
& \quad + \beta \int_\Omega  2\hh'(\opt)(\phi-\opt)\f \cdot \uopt  -\C'(\opt)(\phi-\opt) \E(\uopt) : \E(\uopt) \dx \geq 0
\end{aligned}
\end{equation}
for all $\phi \in \V_m$. 
\end{thm}

\begin{proof}
Recalling the definition of the functional $\F$ in \eqref{def:F}, we introduce
\begin{align}\label{K:defn}
\K(\phi) := \int_\Omega  \frac{\widehat \alpha}{\eps} \Psi_0(\phi) + \beta \hh(\phi) \f \cdot \So(\phi) \dx + \beta \int_{\Gamma_g} \g \cdot \So(\phi) \dH,
\end{align}
so that the reduced functional $\J$ can be expressed as $\J(\phi) = \an{\widehat\alpha }\eps \F(\phi) + \K(\phi)$. Using the differentiability of $\So$ we infer that $\K$ is also Fr\'echet differentiable with derivative at $\opt \in \V_m$ in direction $\zeta \in \V = \{ f \in H^1(\Omega) \, : \, f(\an{\x}) \in [-1,1] \text{ a.e.~in } \Omega\}$ given as
\begin{align}\label{opt:ineq0}
\K'(\opt)[\zeta] = \int_\Omega \frac{\widehat \alpha}{\eps} \Psi_0'(\opt)\zeta + \beta \hh'(\opt) \zeta \f \cdot \uopt + \beta \hh(\opt) \f \cdot \w \dx + \beta \int_{\Gamma_g} \g \cdot \w \dH,
\end{align}
where $\w = \So'(\opt)[\zeta]$ is unique solution to the linearized system \eqref{lin:weak} and $\uopt = \So(\opt)$. We can simplify this using the adjoint system by testing \eqref{adj:weak} with $\bv = \w$ and testing \eqref{lin:weak} with $\be = \p$ to obtain the identity
\[
\beta \int_\Omega \hh(\opt) \f \cdot \w \dx + \beta \int_{\Gamma_g} \g \cdot \w \dH = \int_\Omega \hh'(\opt) \zeta \f \cdot \p - \C'(\opt) \zeta \E(\uopt): \E(\p) \dx,
\]
so that, together with the relation $\p = \beta \uopt$, \eqref{opt:ineq0} becomes
\begin{align}\label{opt:ineq1}
\K'(\opt)[\zeta] = \int_\Omega \frac{\widehat \alpha}{\eps}\Psi_0'(\opt) \zeta + 2 \beta \hh'(\opt) \zeta \f \cdot \uopt - \beta \C'(\opt) \zeta \E(\uopt): \E(\uopt) \dx.
\end{align}
On the other hand, the convexity of $\F$ and the optimality of $\opt \in \V_m$ imply that
\begin{align*}
0 & \leq \J(\opt + t(\phi - \opt)) - \J(\opt) \\
& = \widehat \alpha \eps \F((1-t)\opt  + t\phi) - \widehat \alpha \eps \F(\opt) + \K(\opt + t(\phi - \opt)) - \K(\opt) \\
& \leq t \widehat \alpha \eps (\F(\phi ) - \F(\opt)) + \K(\opt + t(\phi - \opt)) - \K(\opt)
\end{align*}
holds for all $t \in [0,1]$ and arbitrary $\phi \in \V_m$. Dividing by $t$ and passing to the limit $t \to 0$ yields
\begin{align}\label{opt:ineq2}
0 \leq \widehat \alpha \eps \F(\phi) - \widehat \alpha \eps \F(\opt) + \K'(\opt)[\phi - \opt] \quad \forall \phi \in \V_m.
\end{align}
The above inequality allows us to interpret $\opt \in \V_m$ as a solution to the convex minimization problem
\begin{align}\label{opt:altP}
\min_{\zeta \in H^1(\Omega)} \Big ( \widehat \alpha \eps \F(\zeta) + \K'(\opt)[\zeta] + \I_{\V_m}(\zeta) \Big ) \quad \text{ where } \I_{\V_m}(\zeta) = \begin{cases} 0 & \text{ if } \zeta \in \V_m, \\
+\infty & \text{ otherwise},
\end{cases}
\end{align}
denotes the indicator function of the set $\V_m$. Using the well-known sum rule for subdifferentials of convex functions, see, e.g., \cite[Cor.~2.63]{Barbu2}, the inequality \eqref{opt:ineq2} can be interpreted as 
\[
0 \in \pd \big ( \widehat \alpha \eps \F + \K'(\opt) + \I_{\V} \big )(\opt) = \widehat \alpha \eps \pd \F(\opt) + \{\K'(\opt)\} + \pd \I_{\V_m}(\opt),
\]
where $\pd$ denotes the subdifferential mapping in $H^1(\Omega)$. This implies the existence of elements $v \in \pd \F(\opt)$ and $\psi \in \pd \I_{\V_m}(\opt)$ such that 
\begin{align}\label{opt:subdiff}
0 = \widehat \alpha \eps v + \K'(\opt) + \psi.
\end{align}
For arbitrary $\phi \in \V_m$, by definition of $\pd \I_{\V_m}(\opt)$ we have
\[
\I_{\V_m}(\phi) - \I_{\V_m}(\opt) = 0 \geq \inn{\psi,\phi - \opt}_{H^1} \quad \implies \quad \inn{\widehat \alpha \eps v, \phi - \opt}_{H^1} +\K'(\opt)[\phi - \opt] \geq 0.
\]
Then, using Lemma \ref{LEM:A}, there exists $\bxi \in L^2(\Omega, \R^d)$ with $\bxi \in \pd A(\nabla \opt)$ almost everywhere in $\Omega$ and 
\[
(\widehat \alpha \eps \bxi, \nabla (\phi - \opt)) + \K'(\opt)[\phi - \opt] \geq 0 \quad \forall \phi \in \V_m.
\]
The optimality condition \eqref{optcond} is then a consequence of the above and \eqref{opt:ineq1} with $\zeta = \phi - \opt$.
\end{proof}

\begin{remark}
An alternate formulation of the optimality condition \eqref{optcond} is as follows: There exist $\bxi \in \pd A(\nabla \opt)$, $\theta \in \pd \I_{[-1,1]}(\opt)$, and a Lagrange multiplier $\mu \in \R$ for the mass constraint such that
\begin{equation}\label{optcond:2}
\begin{aligned}
&\widehat \alpha \int_\Omega \eps \bxi \cdot \nabla \zeta + \frac{1}{\eps} \Psi'_0(\opt) \zeta + \theta \zeta \dx + \mu \int_\Omega \zeta \dx \\
& \quad + \beta \int_\Omega  2\hh'(\opt)\zeta \f \cdot \uopt  -\C'(\opt) \zeta \E(\uopt) : \E(\uopt) \dx = 0
\end{aligned}
\end{equation}
for all $\zeta \in H^1(\Omega)$. In the above, the subdifferential of the indicator function $\I_{[-1,1]}$ has the explicit characterization
\[
\pd \I_{[-1,1]}(s) = \begin{cases}
[0,\infty) & \text{ if } s = 1, \\
0 & \text{ if } |s| < 1, \\
(-\infty,0] & \text{ if } s = -1.
\end{cases}
\]
Indeed, instead of \eqref{opt:altP} we can interpret $\opt \in \V_m$ as a solution to the minimization problem
\[
\min_{ \zeta \in H^1(\Omega)} \Big ( \widehat \alpha \eps \F(\zeta) + \K'(\opt)[\zeta] + \I_{[-1,1]}(\zeta) + \I_m(\zeta) \Big ), \quad \I_m(\zeta) = \begin{cases} 0 & \text{ if } \int_\Omega \zeta \dx = m |\Omega|, \\
+\infty & \text{ otherwise},
\end{cases}
\]
which then yields the existence of elements $\theta \in \pd \I_{[-1,1]}(\opt)$ and $\eta \in \pd \I_{m}(\opt)$ such that
\[
0 = \widehat \alpha \eps v + \K'(\opt) + \theta + \eta.
\]
For arbitrary $\zeta \in H^1(\Omega)$, we test the above equality with $\zeta - (\zeta)_\Omega$, where $(\zeta)_\Omega = \frac{1}{|\Omega|} \int_\Omega \zeta \dx$, and on noting that
\[
0 = \an{\I_{m}}(\opt \pm (\zeta - (\zeta)_\Omega)) - \an{\I_{m}}(\opt) \geq \inn{\eta, \pm (\zeta - (\zeta)_\Omega)}_{H^1}  \quad  \implies \quad \inn{\eta, \zeta - (\zeta)_\Omega}_{H^1} = 0,
\]
this yields \eqref{optcond:2} with Lagrange multiplier $\mu := - |\Omega|^{-1}  (\K'(\opt)[1] + \theta)$.\end{remark}

\section{Sharp interface asymptotics}\label{sec:SI}
Our interest is to study the behavior of solutions under the sharp interface limit $\eps \to 0$, which connects our phase field approach with the shape optimization approach of \cite{AllaireDEFM17}.

\subsection{$\Gamma$-convergence of the anisotropic Ginzburg--Landau functional}
We begin with the $\Gamma$-convergence of the extended anisotropic Ginzburg--Landau functional \eqref{E:gam:eps} to the extended anisotropic perimeter functional \eqref{E:gam:0}, which is formulated as follows:
\begin{lem}\label{lem:Gamma}
Let $\Omega \subset \R^d$ be an open bounded domain with Lipschitz boundary. Let $\gamma : \R^d \to \R$ satisfy \eqref{gamma:1}--\eqref{gamma:3}, $\Psi: \R \to \R_{\geq 0}$ is a double well potential with minima at $\pm 1$ and define $c_{\Psi} = \int_{-1}^1 \sqrt{2 \Psi(s)} ds$. Then, for any $\eps > 0$
\begin{enumerate}
\item[$(i)$] If $\{u_\eps\}_{\eps>0} \subset \V_m$ is a sequence such that $\liminf_{\eps \to 0} \E_{\gamma,\eps}(u_\eps) < \infty$ and $u_\eps \to u_0$ strongly in $L^1(\Omega)$, then $u_0 \in \BV_m(\Omega, \{-1,1\})$ with $\E_{\gamma,0}(u_0) \leq \liminf_{\eps \to 0} \E_{\gamma,\eps}(u_\eps)$.
\item[$(ii)$] Let $u_0 \in \BV(\Omega, \{-1,1\})$. Then, there exists a sequence $\{u_\eps\}_{\eps > 0}$ of Lipschitz continuous functions on $\Omega$ such that $u_\eps(\x) \in [-1,1]$ a.e.~in $\Omega$, $u_\eps \to u_0$ strongly in $L^1(\Omega)$, $\int_\Omega u_\eps(\x) \dx = \int_\Omega u_0(\x) \dx$ for all $\eps > 0$, and $\limsup_{\eps \to 0} \E_{\gamma,\eps}(u_\eps) \leq \E_{\gamma,0}(u_0)$.
\item[$(iii)$] Let $\{u_\eps\}_{\eps > 0} \subset \V_m$ be a sequence satisfying $\sup_{\eps > 0} \E_{\gamma,\eps}(u_\eps) < \infty$. Then, there exists a non-relabeled subsequence $\eps \to 0$ and a limit function $u$ such that $u_\eps \to u$ strongly in $L^1(\Omega)$ with $\E_{\gamma,0}(u)< \infty$.
\end{enumerate}
\end{lem}
The first and second assertions are known as the liminf and limsup inequalities, respectively, while the third assertion is the compactness property. In the following, we outline how to adapt the proofs of \cite[Thms.~3.1 and 3.4] {Bellettini} for the $\Gamma$-convergence result in the multi-phase case. Our present setting corresponds to the case $N = 2$ in their notation.

\begin{proof}
For arbitrary $u \in \V_m$, we introduce the associated vector $\bm{w} = (w_1, w_2)$ where $w_1 := \frac{1}{2}(1-u)$ and $w_2 := \frac{1}{2}(1+u)$.  Then, $\bm{w}$ take values in the Gibbs simplex $\Sigma = \{ (w_1, w_2) \in \R^2 \, : \, w_1, w_2 \geq 0, \, w_1 + w_2 = 1\}$.  Setting $\bm{\alpha}^1 := (0,1)^{\top}$ and $\bm{\alpha}^2:= (1,0)^{\top}$, we consider a multiple-well potential $W: \Sigma \to [0,\infty)$ satisfying $W^{-1}(0) = \{\bm{\alpha}^1, \bm{\alpha}^2\}$ and 
\begin{align}\label{Gam:W}
W\an{\Big(}\frac{1}{2}(1-u), \frac{1}{2}(1+u)\an{\Big)} = \Psi_0(u),
\end{align}
 where the latter relation connects $W$ to $\Psi_0$ defined in \eqref{obs}.  Denoting by $M^{2 \times d}$ the set of 2-by-$d$ matrices, we consider a function $f: \Sigma \times M^{2 \times d} \to [0,\infty)$ defined as $f(\bm{w}, \bm{X}) = 2|\gamma(w_1 \bm{X}_2 - w_2 \bm{X}_1) |^2$, where $\gamma$ satisfies \eqref{gamma:1}--\eqref{gamma:3}, and $\bm{X}_i$ denotes the \an{$i$-th} row of $\bm{X}$. Taking $\bm{X} = \nabla \bm{w}$ for $\bm{w} \in \Sigma$ yields that $\bm{X}_1 + \bm{X}_2 = \nabla w_1 + \nabla w_2 = 0$ and 
\begin{align}\label{Gam:f}
f(\bm{w}, \nabla \bm{w}) =2 |\gamma((1-w_2) \nabla w_2 + w_2 \nabla w_2)|^2 = \frac{1}{2} |\gamma(\nabla u)|^2
\end{align}
by the relation $\nabla w_2 = \frac{1}{2} \nabla u$ and the one-homogeneity of $\gamma$. Assumptions \eqref{gamma:1}--\eqref{gamma:3} on $\gamma$ ensures the function $f$ defined above fulfills the corresponding assumptions in \cite[p.~80]{Bellettini}, and thus by \cite[Thm.~3.1]{Bellettini},  the extended functional 
\begin{align*}
G_\eps(\bm{w}) = \begin{cases}
\displaystyle \int_\Omega \eps f(\bm{w}, \nabla \bm{w}) + \frac{1}{\eps} W(\bm{w}) \dx & \text{ if } \bm{w} \in H^1(\Omega; \Sigma), \\[10pt]
+\infty & \text{ elsewhere in } L^1(\Omega; \Sigma)
\end{cases}
\end{align*}
$\Gamma$-converges in $L^1(\Omega; \Sigma)$ to a limit functional 
\begin{align*}
\mathcal{G}(\bm{w}) = \begin{cases}
\displaystyle \int_{S_{\bm{w}}} \vp( \bm{\alpha}^1, \bm{\alpha}^2, \bnu) \, \mathrm{d} \Haus & \text{ if } \bm{w} \in \BV(\Omega; \{ \bm{\alpha}^1, \bm{\alpha}^2\}), \\[10pt]
+\infty & \text{ otherwise},
\end{cases}
\end{align*}
where the definitions of the set $S_{\bm{w}}$, normal $\bnu$ and function $\vp$ can be found in \cite[p.~78, p.~81]{Bellettini}. It is clear that $\E_{\gamma,\eps}(u) = G_\eps((\frac{1}{2}(1-u),\frac{1}{2}(1+u)))$ for $u \in \V_m$, and our aim is to show $\E_{\gamma,0}(u) = \mathcal{G}((\frac{1}{2}(1-u),\frac{1}{2}(1+u)))$ for $u \in \BV(\Omega, \{-1,1\})$.  For a fixed vector $\bnu \in \mathbb{S}^{d-1}$, let $\mathcal{C}_\lambda$ denote the $(d-1)$-dimensional cube of side $\lambda$ lying in the orthogonal complement $\bnu^{\perp}$ centered at the origin.  Following \cite{Bellettini}, a function $\bm{w}: \R^d \to \Sigma$ is $\mathcal{C}_\lambda$-periodic if $\bm{w}(\x + \lambda \bm{e}_i) = \bm{w}(\x)$ for every $\bm{x}$ and every $i = 1, \dots, d-1$, with $\bm{e}_1, \dots, \bm{e}_{d-1}$ as directions of the sides of $\mathcal{C}_\lambda$.  Setting $Q_{\lambda,\infty}^{\bnu} = \{ \bm{z} + t \bm{\nu} \, : \, \bm{z} \in \mathcal{C}_\lambda, \, t \in \R\}$ we denote by $\mathcal{X}(Q_{\lambda,\infty}^{\bnu})$ the class of all functions $\bm{w} \in W^{1,2}_{\mathrm{loc}}(\R^d; \Sigma)$ which are $\mathcal{C}_\lambda$-periodic and satisfy $\lim_{\inn{\x, \bnu} \to - \infty} \bm{w}(\x) = \bm{\alpha}^1$ and $\lim_{\inn{\x, \bnu} \to + \infty} \bm{w}(\x) = \bm{\alpha}^2$. Then, by \cite[(6)]{Bellettini}, the integrand in the $\Gamma$-limit functional $\mathcal{G}$ has the representation formula
\[
\varphi(\bm{\alpha}^1, \bm{\alpha}^2, \bnu) = \lim_{\lambda \to +\infty} \inf \Big \{ \lambda^{1-d} G_1(\bm{w}, Q_{\lambda,\infty}^{\bnu}) \, : \, \bm{w} \in \mathcal{X}(Q_{\lambda,\infty}^{\bnu}) \Big \},
\]
where $G_1(\bm{w},A) = \int_A f(\bm{w}, \nabla \bm{w}) + W(\bm{w}) \dx$. To simplify the above expression, for a fixed $\bnu \in \mathbb{S}^{d-1}$, consider a function $u \in \V_m$ of the form $u(\x) = \psi(\x \cdot \bnu /\gamma(\bnu))$ for a monotone function $\an{\psi} : \R \to [-1,1]$ such that $\lim_{s \to \pm \infty} \psi(s) = \pm 1$ and satisfies $\psi''(s) = \Psi_0'(\psi(s))$. For the choice $\Psi_0(s) = \frac{1}{2}(1-s^2)$ we have the explicit solution (known also as the optimal profile)
\begin{align}\label{Obs:prof}
\psi(s) = \begin{cases}
1 & \text{ if } s \geq \frac{\pi}{2}, \\
\sin(s) & \text{ if } |s| \leq \frac{\pi}{2}, \\
-1 & \text{ if } s \leq - \frac{\pi}{2}.
\end{cases}
\end{align}
Furthermore, multiplying the equality $\psi''(s) = \Psi'_0(\psi(s))$ by $\psi'(s)$ and integrating yields the so-called equipartition of energy $\frac{1}{2}|\psi'(s)|^2 = \Psi_0(\psi(s))$ for all $s \in \R$.  Then, it is clear that $\bm{w} =  (\frac{1}{2}(1-u), \frac{1}{2}(1+u)) \in \mathcal{X}(Q_{\lambda,\infty}^{\bnu})$.  By \eqref{gamma:1}, \eqref{Gam:W} and \eqref{Gam:f}, as well as Fubini's theorem, for $u(\x) = \an{\psi}(\x \cdot \bnu / \gamma(\bnu))$ we have $\nabla u(\x) = \an{\psi}'(\x \cdot \bnu/ \gamma(\bnu)) \bnu/ \gamma(\bnu)$, and
\begin{align*}
G_1(\bm{w},Q_{\lambda,\infty}^{\bnu}) & =  \int_{\mathcal{C}_{\lambda}} \int_{\R} \frac{1}{2}|\gamma(\nabla u(\bm{z} + t \bnu))|^2 + \Psi_0(u(\bm{z} + t \bnu)) \dt \dbz  \\
& = \int_{\mathcal{C}_\lambda} \int_{\R} \frac{1}{2} |\psi'(t/\gamma(\bnu))|^2 + \Psi_0(\psi(t/\gamma(\bnu))) \dt \dbz \\
&  = |\mathcal{C}_\lambda| \gamma(\bnu) \int_{-1}^1 \frac{1}{2}|\psi'(s)|^2 + \Psi_0(\psi(s)) \, \mathrm{d}s = \lambda^{d-1} \gamma(\bnu) \int_{-1}^1 \sqrt{2 \Psi_0(r))} \dr
\end{align*}
after a change of variables $r = \psi(s)$ and using the equipartition of energy $\frac{1}{2}|\psi'(s)|^2 = \Psi_0(\psi(s))$.  Hence, we obtain the identification $\varphi(\bm{\alpha}^1, \bm{\alpha}^2, \bnu) = c_{\Psi} \gamma(\bnu)$ with constant $c_{\Psi} = \int_{-1}^1 \sqrt{2 \Psi_0(r)} \dr = \pi$. Then, for $\bm{w} = (w_1, w_2) \in \BV(\Omega, \{ \bm{\alpha}^1, \bm{\alpha}^2 \})$, we define $u \in \BV(\Omega, \{-1,1\})$ via the relation $u = w_2 - w_1$, and this allows us to identify $S_{\bm{w}} = \pd^* \{u = 1\}$ and consequently $\mathcal{G}(\bm{w}) = \E_{\gamma,0}(u)$. Then, the liminf, limsup and compactness properties follow directly from \cite[Thms.~3.1 and 3.4]{Bellettini}.
\end{proof}

\subsection{Formally matched asymptotic analysis}\label{sec:formal}
In this section we consider an anisotropy function $\gamma \in C^2(\R^d)\cap W^{1,\infty}(\R^d)$ satisfying \eqref{gamma:1}--\eqref{gamma:3} in the structural optimization problem \eqref{opt}, as well as a more regular body force $\f \in H^1(\Omega, \R^d)$.  The differentiability of $\gamma$ implies the characterization of the subdifferential $\pd A(\nabla \vp)$ as the singleton set $\{ [\gamma \D \gamma](\nabla \vp) \}$, where
\[
\D\gamma(\q) = (\pd_{q_1} \gamma, \pd_{q_2} \gamma, \dots, \pd_{q_d} \gamma)(\q) \in \R^d, \quad  \q = (q_1, \dots, q_d) \in \R^d.
\]
Then, the corresponding optimality condition for a minimizer $(\opt_\eps, \uopt_\eps)$ to \eqref{opt} becomes
\begin{equation}\label{opt:alt}
\begin{aligned}
& \widehat \alpha \int_\Omega \eps [\gamma \D \gamma](\nabla \opt_\eps) \cdot \nabla ( \phi - \opt_\eps) + \frac{1}{\eps} \Psi_0'(\opt_\eps) (\phi - \opt_\eps) \dx \\
& \quad + \beta \int_\Omega [2 \hh'(\opt_\eps) \f \cdot \uopt_\eps - \C'(\opt_\eps) \E(\uopt_\eps) : \E(\uopt_\eps) ] (\phi - \opt_\eps) \dx \geq 0 \quad \forall \phi \in \an{\V_m},
\end{aligned}
\end{equation}
and its strong formulation reads as (see \eqref{optcond:2})
\begin{subequations}\label{opt:Strong}
\begin{alignat}{2}
\notag \beta \Big ( 2 \hh'(\opt_\eps) \f \cdot \uopt_\eps - \C'(\opt_\eps) \E(\uopt_\eps) : \E(\uopt_\eps) \Big ) + \mu_\eps & \\
\quad + \, \widehat \alpha \Big ( \frac{1}{\eps} \Psi_0'(\opt_\eps) - \eps \div \big( [\gamma \D \gamma](\nabla \opt_\eps)  \big) \Big )  & = 0  \quad \text{ in } \Omega \cap \{|\opt_\eps|<1\}, \label{opt1} \\
[\gamma \D \gamma](\nabla \opt_\eps) \cdot \bm{n} & = 0 \quad \text{ on } \pd \Omega, \label{opt2}
\end{alignat}
\end{subequations}
with Lagrange multiplier $\mu_\eps$ for the mass constraint.  Our aim is to perform a formally matched asymptotic analysis as $\eps \to 0$, similar as in \cite[Sec.~5]{BlankGSS12}, in order to infer the sharp interface limit of \eqref{opt:Strong}. The method proceeds as follows, see \cite{Cahn}: formally we assume that the domain $\Omega$ admits a decomposition $\Omega = \Omega_\eps^+ \cup \Omega_\eps^- \cup \Omega_\eps^{\rm{I}}$, where $\Omega_\eps^{\rm{I}}$ is an annular domain and $\opt_\eps \in \V_m$ satisfies
\begin{align}\label{SI:ass}
\opt_\eps = \pm 1 \text{ in } \Omega_\eps^{\pm}, \quad |\opt_\eps| < 1 \text{ in } \Omega_\eps^{\rm{I}}.
\end{align}
The zero level set $\Lambda_\eps := \{ \x \in \Omega \, : \, \opt_\eps(\x) = 0\}$ is assumed to converge to a smooth hypersurface $\Lambda$ as $\eps \to 0$, and that $(\opt_\eps, \uopt_\eps)$ admit an {\it outer expansion} in regions in $\Omega_\eps^{\pm}$ as well as an {\it inner expansion} in regions in $\Omega_\eps^{\rm{I}}$. We substitute these expansions (in powers of $\eps$) in \eqref{state} and \eqref{opt:Strong}, collecting terms of the same order of $\eps$, and with the help of suitable {\it matching conditions} connecting these two expansions, we deduce an equation posed on $\Lambda$. For an introduction and more detailed discussion of this methodology we refer to \cite{Fife,GStin}.

To start, let us collect some useful relations.  As $\gamma$ is positively homogeneous of degree one, taking the relation $\gamma(t \p) = t \gamma(\p)$ for $t> 0$ and $\p \in \R^d \setminus \{\0\}$ and differentiating with respect to $t$, and also with respect to $\p$ leads to 
\begin{align}\label{Dgam}
\D \gamma(t \p) \cdot \p = \gamma(\p), \quad \D \gamma(t \p) = \D \gamma(\p),
\end{align}
with the latter relation showing that $\D \gamma$ is positively homogeneous of degree zero.  Then, it is easy to see $\gamma \D \gamma$ is positively homogeneous of degree one. Next, from \eqref{defn:A}, we see that $A$ is positively homogeneous of degree two. Taking the relation $A(t \p) = t^2 A(\p)$ and differentiating with respect to $t$ twice leads to 
\begin{align}\label{gam:Hess}
\D^2 A(t \p) \p \cdot \p = 2 A(\p) \quad \implies \quad [\D(\gamma \D\gamma)(\p)] \p \cdot \p = |\gamma(\p)|^2,
\end{align}
where $\D(\gamma \D \gamma) = \gamma \D^2 \gamma + \D \gamma \otimes \D \gamma$ with $\D^2 \gamma$ denoting the Hessian matrix of $\gamma$. Lastly, differentiating the relation $[\gamma \D \gamma](t \p) = t [\gamma \D \gamma](\p)$ with respect to $t$ and setting $t = 1$ yields
\begin{align}\label{gam:Hess2}
[\D (\gamma \D \gamma)(\p)] \p = ( \gamma \D \gamma)(\p).
\end{align}

\subsubsection{Outer expansions} 
For points $\x$ in $\Omega_\eps^{\pm}$, we assume an outer expansion of the form
\[
\uopt_\eps(\x) = \sum_{k=0}^\infty \eps^k \u_k(\x), \quad \opt_\eps(\x) = \sum_{k=0}^\infty \eps^k \vp_k(\x), 
\]
where all functions are sufficiently smooth and the summations converge. From \eqref{SI:ass} we deduce that 
\[
\opt_0(\x) \in \{-1,1\}, \quad \opt_i (\x)= 0 \text{ for } i \geq 1.
\]
Setting $\Omega_+ = \{\opt_0 = 1\}$ and $\Omega_{-} = \{\opt_0 = -1\}$, it holds that $(\pd \Omega_+ \cap \pd \Omega_-)\cap \Omega = \Lambda$. Then, substituting the outer expansion into \eqref{state}, we obtain to order $\mathcal{O}(1)$ the following system of equations
\begin{subequations}\label{outer:sys}
\begin{alignat}{2}
\label{o1} - \div ( \C(\pm 1) \E(\u_0)) &= \hh(\pm 1) \f && \quad \text{ in } \Omega_{\pm}, \\
\label{o2} \u_0 &= \0 && \quad \text{ on } \Gamma_D \cap \pd \Omega_{\pm}, \\
\label{o3} (\C(\pm 1) \E(\u_0)) \bm{n} &= \g && \quad \text{ on } \Gamma_g \cap \pd \Omega_{\pm}, \\
\label{o4} (\C(\pm 1) \E(\u_0)) \bm{n} &= \0 && \quad \text{ on } \Gamma_0 \cap \pd \Omega_{\pm}.
\end{alignat}
\end{subequations}
Moreover, from the definition of the Lagrange multiplier $\mu_\eps$ we see that 
\begin{equation}\label{L:exp}
\begin{aligned}
& \mu_\eps = \frac{1}{\eps} \mu_{-1} + \mu_0 + \mathcal{O}(\eps), \quad \text{ with } \quad \mu_{-1} := \frac{1}{|\Omega|} \int_\Omega \widehat{\alpha} \Psi_0'(\opt_0) \dx = 0, \\
& \mu_0 = \frac{1}{|\Omega|} \int_\Omega 2 \beta \hh'(\opt_0) \f \cdot \uopt_0- \beta \C'(\opt_0) \E(\uopt_0): \E(\uopt_0) \dx\an{.}
\end{aligned}
\end{equation}
It remains to derive the boundary conditions for \eqref{outer:sys} holding on $\Lambda$, which can be achieved with the inner expansions.
\subsubsection{Inner expansions} 
We assume the  outer boundary $\Gamma_\eps^+$ and inner boundary $\Gamma_\eps^-$ of the annular region $\Omega_\eps^{\rm{I}}$ can be parameterized over the smooth hypersurface $\Lambda$. Let $\bnu$ denote the unit normal of $\Lambda$ pointing from $\Omega_-$ to $\Omega_+$, and we choose a spatial parameter domain $U \subset \R^{d-1}$ with a local parameterization $\rr: U \to \R^d$ of $\Lambda$.  Let $d$ denote the signed distance function of $\Lambda$ with $d(\x) > 0$ if $\x \in \Omega_+$, and denote by $z = \frac{d}{\eps}$ the rescaled signed distance.  Then, by the smoothness of $\Lambda$, there exists $\delta_0 > 0$ such that for all $\x \in \{|d(\x)| < \delta_0\}$, we have the representation
\[
\x = G^\eps(\sss,z) = \rr(\sss) + \eps z \bnu(\sss),
\]
where $\sss = (s_1, \dots, s_{d-1}) \in U$. In particular, $\rr(\sss) \in \Lambda$ is the projection of $\x$ to $\Lambda$ along the normal direction. This representation allows us to infer the following expansion for gradients and divergences \cite{GStin}:
\begin{align}
\label{eq:nabladiv}
\nabla_{\x} b = \frac{1}{\eps} \pd_z \hat b \bnu + \nabla_\Lambda \hat b + \mathcal{O}(\eps), \quad \div_{\x} \bm{j} = \frac{1}{\eps} \pd_z \widehat{\bm{j}} \cdot \bnu + \div_\Lambda \widehat{\bm{j}} + \mathcal{O}(\eps)
\end{align}
for scalar functions $b(\x) = \hat{b}(\sss(\x), z(\x))$ and vector functions $\bm{j}(\x) = \widehat{\bm{j}}(\sss(\x), z(\x))$.  In the above $\nabla_\Lambda$ is the surface gradient on $\Lambda$ and $\div_\Lambda$ is the surface divergence. Analogously, for a vector function $\bm{j}$, we find that 
\[
\nabla_{\x} \bm{j} = \frac{1}{\eps} \pd_z \widehat{\bm{j}} \otimes \bnu + \nabla_\Lambda \widehat{\bm{j}} + \mathcal{O}(\eps).
\]
For points close by $\Lambda$ in $\Omega_\eps^{\rm{I}}$, we assume an expansion of the form
\[
\uopt_\eps(\x) = \sum_{k=0}^\infty \eps^k \bm{U}_k(\sss,z), \quad \opt_\eps(\x) = \sum_{k=0}^\infty \eps^k \Phi_k(\sss,z).
\]
\subsubsection{Matching conditions} In a tubular neighborhood of $\Lambda$, we assume the outer expansions and the inner expansions  hold simultaneously. Since $\Gamma_\eps^\pm$ are assumed to be graphs over $\Lambda$, we introduce the functions $Y_\eps^{\pm}(\sss)$ such that $\Gamma_{\eps}^{\pm} = \{ z = Y_{\eps}^{\pm}(\sss) \, : \, \sss \in U \}$.  Furthermore, we assume an expansion of the form
\[
Y_{\eps}^{\pm} (\sss) = Y_0^{\pm}(\sss) + \eps Y_1^{\pm}(\sss) + \mathcal{O}(\eps^2)
\]
is valid. As $\Gamma_\eps^{\pm}$ converge to $\Lambda$,  in the computations below we will deduce the values for $Y_0^{\pm}(\sss)$.  Then, by comparing these two expansions in this intermediate region we infer matching conditions relating the outer expansions to the inner expansions via boundary conditions for the outer expansions. For a scalar function $b({\x})$ admitting an outer expansion $\sum_{k=0}^\infty \eps^k b_k({\x})$ and an inner expansion $\sum_{k=0}^\infty \eps^k B_k({\sss},z)$, it holds that (see \cite{Cahn} and \cite[Appendix D]{GStin})
\begin{align*}
B_0(\sss,z) & \to \begin{cases}
\lim_{\delta \searrow 0} b_0(\x + \delta \bnu(\x)) =: b_0^+(\x) & \text{ for } z \to Y_0^+, \\
\lim_{\delta \searrow 0} b_0(\x - \delta \bnu(\x)) =: b_0^-(\x) & \text{ for } z \to Y_0^-,
\end{cases} \\
 \pd_z B_0(\sss,z) & \to 0 \text{ as } z \to Y_0^{\pm}, \\
\pd_z B_1(\sss,z) & \to \begin{cases}
\lim_{\delta \searrow 0} (\nabla b_0)(\x + \delta \bnu(\x)) \cdot \bnu(\x) =: \nabla b_0^+ \cdot \bnu & \text{ for } z \to Y_0^+, \\
\lim_{\delta \searrow 0} (\nabla b_0)(\x - \delta \bnu(\x)) \cdot \bnu(\x) =: \nabla b_0^- \cdot \bnu& \text{ for } z \to Y_0^-,
\end{cases}
\end{align*}
for $\x \in \Lambda$. Consequently, we denote the jump of a quantity $b$ across $\Lambda$ as 
\[
[b]_{-}^{+} := \lim_{\delta \searrow 0} b(\x + \delta \bnu(\x)) - \lim_{\delta \searrow 0} b(\x - \delta \bnu(\x) \an{)}  \quad \text{ for }\x \in \Lambda.
\]
\subsubsection{Analysis of the inner expansions}
We introduce the notation $\sym{\bm{B}} = \frac{1}{2}(\bm{B} + \bm{B}^{\top})$ for a second order tensor $\bm{B}$.  Plugging in the inner expansions to the state equation \eqref{state}, to leading order $\mathcal{O}(\frac{1}{\eps^2})$ we find that
\[
\0 = \bnu \pd_z \Big ( \C(\Phi_0) \sym{\pd_z \bm{U}_0 \otimes \bnu } \Big ).
\]
Taking the product with $\bm{U}_0$ and by the symmetry of $\C$ we have the relation
\[
\bnu  \pd_z \Big ( \C(\Phi_0) \sym{\pd_z \bm{U}_0 \otimes \bnu } \Big ) \cdot \bm{U}_0 = \pd_z \Big ( \C(\Phi_0) \sym{\pd_z \bm{U}_0 \otimes \bnu} \Big ): \bm{U}_0 \otimes \bnu.
\]
Integrating over $z$ and by parts, using $\pd_z \bnu = \0$ and by the matching condition $\lim_{z \to Y_0^{\pm}} \pd_z \bm{U}_0 = \0$ leads to
\[
0 = \int_{Y_0^-}^{Y_0^{+}} \C(\Phi_0) \sym{\pd_z \bm{U}_0 \otimes \bnu} : \sym{\pd_z \bm{U}_0 \otimes \bnu} \, dz.
\]
Coercivity of $\C$ yields that $\sym{\pd_z \bm{U}_0 \otimes \bnu} = \0$ and hence $\pd_z \bm{U}_0(\sss,z) = \0$. This implies $\bm{U}_0$ is constant in $z$ and by the matching conditions we obtain
\begin{align}\label{inn:u}
\u_0^+ = \bm{U}_0 = \u_0^- \quad \implies \quad [\u_0]_{-}^{+} = \0.
\end{align}
Then, using $\pd_z \bm{U}_0 = \0$, to the next order $\mathcal{O}(\frac{1}{\eps})$ we obtain from the state equation \eqref{state}
\begin{align}\label{inn:state}
\0 = \pd_z \Big ( \C(\Phi_0) \sym{\pd_z \bm{U}_1 \otimes \bnu + \nabla_\Lambda \bm{U}_0 } \bnu \Big ).
\end{align}
Integrating over $z$ and using the matching condition
\[
\pd_z \bm{U}_1 \otimes \bnu + \nabla_\Lambda \bm{U}_0  \to \begin{cases}
\nabla \u_0^+ & \text{ for } z \to Y_0^+, \\
\nabla \u_0^- & \text{ for } z \to Y_0^-,
\end{cases}
\]
we obtain
\begin{align}\label{inn:u2}
\C(1) \E(\u_0^+) \bnu - \C(-1) \E(\u_0^-) \bnu = [\C \E(\u_0) ]_{-}^{+} \bnu = \0.
\end{align}
From the optimality condition \eqref{opt1}, to leading order $\mathcal{O}(\frac{1}{\eps^2})$ we obtain the trivial equality $0 = 0$ due to $\pd_z \bm{U}_0 = \0$. Recalling \eqref{L:exp} for $\mu_{-1}$, to the next order $\mathcal{O}(\frac{1}{\eps})$ of \eqref{opt1} we have
\begin{align}\label{inn:ode}
\Psi_0'(\Phi_0) - |\gamma(\bnu)|^2 \pd_{zz} \Phi_0 = 0,
\end{align}
where we again use $\pd_z \bm{U}_0 = \0$ to see that no terms from the elastic part contribute at order $\mathcal{O}(\frac{1}{\eps})$, the expression \eqref{eq:nabladiv} for the divergence in the inner variables, the relation $\D \gamma(\bnu) \cdot \bnu = \gamma(\bnu)$ from \eqref{Dgam}, the fact that $\gamma \D \gamma$ is positively homogeneous of degree one, so that 
\[
(\gamma \D \gamma) \big ( \tfrac{1}{\eps} \pd_z \Phi_0 \bnu \big ) = \tfrac{1}{\eps}\pd_z \Phi_0  (\gamma \D \gamma) (\bnu ),
\]
as well as the fact that $\bnu$ is independent of $z$. We now construct a solution to \eqref{inn:ode} fulfilling the conditions $\Phi_0(\sss,z) \to \pm 1$ as $z \to Y_0^{\pm}(\sss)$ as per the matching conditions, and also satisfies $\pd_z \Phi_0(\sss,z) > 0$. Let $\psi(z)$ be the monotone function defined in \eqref{Obs:prof} which satisfies the boundary value problem
\[
\psi''(z) = \Psi_0'(\psi(z)), \quad \lim_{z \to \pm \frac{\pi}{2}} \psi(z) = \pm 1,
\]
with $\psi(0) = 0$ and $\psi'(z) >0$ for all $z \in \R$. We define $\Phi_0(\sss,z) = \psi \Big ( \frac{z}{\gamma(\bnu(\sss))} \Big )$, and a short calculation shows that 
\begin{align*}
\pd_{zz} \Phi_0(\sss,z) = \frac{1}{|\gamma(\bnu(\sss))|^2} \psi'' \Big ( \frac{z}{\gamma(\bnu(\sss))} \Big ) =  \frac{1}{|\gamma(\bnu(\sss))|^2} \Psi_0'( \Phi_0(\sss,z)).
\end{align*}
The property $\psi(\pm \frac \pi2) =\pm1$ and the matching conditions for $\Phi_0$ provide the identifications
\[
	Y_0^+(\sss) = \frac{\pi}{2 }\gamma(\bnu(\sss)),
	\quad 
	Y_0^-(\sss)=  -\frac{\pi}{2} \gamma(\bnu(\sss)).
\]
Moreover, since $\psi$ is monotone and $\gamma(\bnu)$ is positive, we can infer that $\pd_z \Phi_0(\sss,z) > 0$.  Now, multiplying \eqref{inn:ode} with $\pd_z \Phi_0$ and integrat\an{ing} over $z$ yields the so-called equipartition of energy
\[
\Psi_0( \Phi_0(\sss,z)) = \frac{|\gamma(\bnu(\sss))|^2}{2} |\pd_z \Phi_0(\sss,z)|^2,
\]
after using $\Psi_0(-1) = 0$ and $\lim_{z \to Y_0^-} \pd_z \Phi_0(\sss,z) = 0$. Performing another integration over $z$ leads to the relation
\begin{align}\label{value}
\int_{Y_0^-}^{Y_0^+} |\pd_z \Phi_0|^2 \dz = \frac{2}{|\gamma(\bnu(\sss))|^2} \int_{Y_0^-}^{Y_0^+} \Psi_0(\Phi_0(\sss,z)) \dz = \int_{-1}^1 \frac{\sqrt{2 \Psi_0(r)}}{\gamma(\bnu(\sss))}  \dr = \frac{c_{\Psi}}{\gamma(\bnu(\sss))}.
\end{align}
To the next order $\mathcal{O}(1)$, we find that 
\begin{equation}\label{inn:solv}
\begin{aligned}
0 & = 2 \beta \hh'(\Phi_0) \f \cdot \bm{U}_0 + \mu_0 +\widehat \alpha \Psi_0''(\Phi_0) \Phi_1 - \an{\widehat\alpha }\div_\Lambda ([\gamma \D \gamma](\bnu) \pd_z \Phi_0)  \\
& \quad - \an{\widehat\alpha} \pd_z (\bnu \cdot \D(\gamma \D \gamma)(\pd_z \Phi_0 \bnu) (\pd_z \Phi_1 \bnu + \nabla_\Lambda \Phi_0)) \\
& \quad - \beta \C'(\Phi_0) \sym{\pd_z \bm{U}_1 \otimes \bnu + \nabla_\Lambda \bm{U}_0}: \sym{\pd_z \bm{U}_1 \otimes \bnu + \nabla_\Lambda \bm{U}_0}.
\end{aligned}
\end{equation}
The aim is to multiply the above with $\pd_z \Phi_0$ and integrate over $z$, but first, let us derive some useful relations. Using the symmetry $\C_{ijkl} = \C_{jikl}$ of the elasticity tensor and $\pd_z \bm{U}_0 = \0$, we deduce that, with the shorthand $\E_U = \sym{\pd_z \bm{U}_1 \otimes \bnu + \nabla_\Lambda \bm{U}_0}$,
\begin{align*}
\C(\Phi_0) \E_U : \pd_z \E_U = \C(\Phi_0) \E_U : \sym{\pd_{zz} \bm{U}_1 \otimes \bnu} = \C(\Phi_0) \E_U \bnu \cdot \pd_{zz} \bm{U}_1.
\end{align*}
Hence, together with $\pd_z(\C(\Phi_0) \E_U \bnu) = \0$ from \eqref{inn:state} and the matching conditions we obtain the relation
\begin{align*}
\int_{Y_0^-}^{Y_0^+} \pd_z \C(\Phi_0) \E_U : \E_U \dz & =  \int_{Y_0^-}^{Y_0^+} \pd_z \C(\Phi_0) \E_U : \E_U + 2 \C(\Phi_0) \E_U : \pd_z \E_U \dz \\
& \quad - \int_{Y_0^-}^{Y_0^+} 2 \C(\Phi_0) \E_U \bnu \cdot \pd_{zz} \bm{U}_1 + 2 \pd_z( \C(\Phi_0) \E_U \bnu) \cdot \pd_z \bm{U}_1 \dz \\
& = \int_{Y_0^-}^{Y_0^+} \pd_z \Big ( \C(\Phi_0) \E_{U} : \E_{U} \Big ) - 2\pd_z \Big ( \C(\Phi_0) \E_U \bnu \cdot \pd_z \bm{U}_1 \Big ) \dz\\
& = [\C \E( \u_0) : \E( \u_0) - 2\C(\Phi_0) \E( \u_0) \bnu \cdot (\nabla \u_0) \bnu]_{-}^{+}.
\end{align*}
By \eqref{gam:Hess}, \eqref{gam:Hess2}, \eqref{inn:ode}, \eqref{value} and the matching conditions of $\pd_z \Phi_0$, we have
\begin{align*}
& \int_{Y_0^-}^{Y_0^+} \Psi_0''(\Phi_0) \Phi_1 \pd_z \Phi_0 \dz = - \int_{Y_0^-}^{Y_0^+} \Psi_0'(\Phi_0) \pd_z \Phi_1 \dz = -  |\gamma(\bnu)|^2 \int_{Y_0^-}^{Y_0^+} \pd_{zz} \Phi_0 \pd_z \Phi_1 \dz, \\
& \int_{Y_0^-}^{Y_0^+} \bnu \cdot [\D(\gamma \D \gamma)(\bnu)] \bnu (\pd_{zz} \Phi_1 \pd_z \Phi_0) \dz = - |\gamma(\bnu)|^2 \int_{Y_0^-}^{Y_0^+} \pd_{z} \Phi_1 \pd_{zz} \Phi_0 \dz,\\
& \int_{Y_0^-}^{Y_0^+} \div_\Lambda ([\gamma \D \gamma](\bnu) \pd_z \Phi_0) \pd_z \Phi_0 \dz = \frac{c_\Psi \div_\Lambda ( [\gamma \D \gamma](\bnu))}{\gamma(\bnu)} + [\gamma \D \gamma](\bnu) \cdot \nabla_\Lambda \Big ( \frac{c_{\Psi}}{2 \gamma(\bnu)} \Big ), \\
& \int_{Y_0^-}^{Y_0^+} (\bnu \cdot [\D( \gamma \D \gamma)(\bnu)] \nabla_\Lambda \pd_z \Phi_0) \pd_z \Phi_0 \dz = [\gamma \D \gamma](\bnu) \cdot \nabla_{\Lambda} \Big ( \frac{c_{\Psi}}{2 \gamma(\bnu)} \Big ).
\end{align*}
Hence, 
\begin{align*}
&\int_{Y_0^-}^{Y_0^+} \div_\Lambda ([\gamma \D \gamma](\bnu) \pd_z \Phi_0) \pd_z \Phi_0 \dz + \int_{Y_0^-}^{Y_0^+} (\bnu \cdot [\D( \gamma \D \gamma)(\bnu)] \nabla_\Lambda \pd_z \Phi_0) \pd_z \Phi_0 \dz \\
& \quad =  \frac{c_\Psi \div_\Lambda ( [\gamma \D \gamma](\bnu))}{\gamma(\bnu)} + [\gamma \D \gamma](\bnu) \cdot \nabla_\Lambda \Big ( \frac{c_{\Psi}}{ \gamma(\bnu)} \Big ) = c_{\Psi} \div_{\Lambda} \big ( \D \gamma(\bnu) \big ),
\end{align*}
and upon multiplying \eqref{inn:solv} with $\pd_z \Phi_0$ and integrate over $z$, employ the matching conditions leads to the following solvability condition for $\Phi_1$:
\begin{equation}\label{solvability}
\begin{aligned}
0 & = 2 \beta \f \cdot \u_0 [\hh]_{-}^+ + 2 \mu_0 - \widehat \alpha c_{\Psi} \div_\Lambda (\D \gamma(\bnu)) \\
& \quad - \beta [\C \E(\u_0) : \E(\u_0) - 2\C \E(\u_0) \bnu \cdot (\nabla \u_0) \bnu]_{-}^{+} \quad \text{ on } \Lambda.
\end{aligned}
\end{equation} 
Note that the assumed higher regularity $\f \in H^1(\Omega, \R^d)$ allows us to define $\f$ on $\Lambda$. Thus, the sharp interface limit $\eps \to 0$ consists of the system \eqref{outer:sys} posed in $\Omega_{\pm}$, along with the boundary conditions \eqref{inn:u}, \eqref{inn:u2} and \eqref{solvability} on $\Lambda$.

\begin{remark}[Isotropic case]
In the isotropic case $\gamma(\bnu) = |\bnu|$, for unit normals $\bnu$ of $\Lambda$, we have that $\D \gamma(\bnu) = \frac{\bnu}{|\bnu|} = \bnu$ and $\div_\Lambda(\D \gamma(\bnu)) = - \kappa_\Lambda$, where $\kappa_{\Lambda}$ is the mean curvature of $\Lambda$. Then, choosing $\hh(\vp) = 1$ so that $[\hh]_{-}^{+} = 0$, we observe that \eqref{solvability} simplifies to 
\[
0 = 2\mu_0 +  \widehat \alpha c_{\Psi} \kappa_{\Lambda} - \beta [ \C\E(\u_0) : \E(\u_0) - 2 \C\E(\u_0) \bnu \cdot (\nabla \u_0) \bnu)]_{-}^{+} \text{ on } \Lambda,
\]
which coincides with the formula obtained from \cite[(24)]{BGHR} with adjoint $\q_0 = \beta \u_0$, zero eigenstrain $\overline{\E} = \0$, Lagrange multiplier $\lambda_0 = \mu_0$ for the mass constraint, and $h_\Omega(\x,\u) = \beta \f \cdot \u$ that satisfies $[h_\Omega(\x,\u_0)]_-^+ = 0$ thanks to \eqref{inn:u}. 
\end{remark}

\subsection{Relation to shape derivatives derived in \cite{AllaireDEFM17}}
To relate our work with the setting of \cite{AllaireDEFM17}, we first set $\f = \0$ in \eqref{outer:sys}, and neglect all the equations posed in $\Omega_-$.  
Then, the state equation \cite[(2.1)]{AllaireDEFM17} can be obtained by writing $\C(+1)$ as $A$, $\Omega_+$ as $\Omega$ in \eqref{outer:sys}.  Note that in this setting the free boundary $\Lambda$ is identified with the traction-free part $\Gamma_0 \cap \pd \Omega_+$ of $\pd \Omega_+$.  Furthermore, the normal vector $\bm{n}$ in \cite{AllaireDEFM17} on $\Lambda$ is equal to $-\bnu$ in our notation. 

Writing $\varphi(\bm{n})$ as $\gamma(-\bnu)$, the anisotropic perimeter functional $P_g$ defined in \cite[(3.1)]{AllaireDEFM17} can be expressed as in \eqref{anis:peri}, and in terms of our notation its shape derivative is given as (compare \cite[Prop.~3.1]{AllaireDEFM17})
\[
P_g'(\Omega_{+})[\bm{\theta}] = \int_{\Lambda} -\kappa_\Lambda \gamma(\bnu) \bm{\theta} \cdot \bnu + \D \gamma(\bnu) \cdot \nabla_\Lambda (\bm{\theta} \cdot \bnu) \dH,
\]
for admissible vector fields $\bm{\theta}$ of sufficient smoothness, where we used that $\D \gamma$ is positively homogeneous of degree zero (cf.~\eqref{Dgam}) and the identity $\kappa_{\Lambda} = - \div_\Lambda \bnu = \div_{\Lambda} \bm{n}$. Here we point out a typo in the manuscript, where $\nabla_{\pd \Omega}(\varphi(\bm{n}))$ written there should in fact be $\D \varphi(\bm{n})$ (compare \cite[Prop.~8]{Dapogny}).

Now, consider the shape optimization problem of minimizing $L(\Omega_{+}) := \beta J(\Omega_+) + \an{\widehat\alpha }c_{\Psi} P_g(\Omega_{+})$ with the mean compliance 
\[
J(\Omega_+) := \int_{\Omega_{+}} \C(1) \E(\u) : \E(\u) \dx.
\]
Then, with sufficiently smooth solutions, its shape derivative in our notation is (see \cite[p.~299]{AllaireDEFM17})
\begin{align*}
L'(\Omega_{+})[\bm{\theta}] & = \int_{\Lambda} [- \beta \C(1) \E(\u) : \E(\u) - \widehat \alpha c_{\Psi} \kappa_\Lambda \gamma(\bnu)] \bm{\theta} \cdot \bnu + \widehat \alpha c_{\Psi} \D \gamma(\bnu) \cdot \nabla_\Lambda(\bm{\theta} \cdot \bnu) \dH \\
& = \int_{\Lambda} [- \beta \C(1) \E(\u) : \E(\u) - \widehat \alpha c_{\Psi} \div_{\Lambda} (\D \gamma(\bnu)) ] \bm{\theta} \cdot \bnu  \dH,
\end{align*}
where in the above we have applied the integration by parts formula
\[
\int_\Lambda \nabla_{\Lambda} (\bm{\theta} \cdot \bnu) \cdot \D \gamma(\bnu) \dH = \int_{\Lambda} [\kappa_{\Lambda} \D \gamma(\bnu) \cdot \bnu - \div_{\Lambda}(\D \gamma(\bnu))] (\bm{\theta} \cdot \bnu) \dH,
\]
as well as the relation $\D \gamma(\bnu) \cdot \bnu = \gamma(\bnu)$ from \eqref{Dgam} to cancel the terms involving $\kappa_{\Lambda}$. Hence, the strong formulation of $L'(\Omega_{+})[\bm{\theta}] = 0$ is
\[
0 = \widehat \alpha c_{\Psi} \div_{\Lambda}(\D \gamma(\bnu)) + \beta \C(1) \E(\u) : \E(\u) \quad \text{ on } \Lambda = \Gamma_0 \cap \pd \Omega_{+}.
\]
This coincides with the identity obtained if we neglect the Lagrange multiplier, set $\C(-1) = \0$ and $\f = \0$ in \eqref{solvability}, and also use the relation $\C(+1) \E(\u_0^+) \bnu = \0$ from \eqref{inn:u2}.

\section{Numerical approximation}\label{sec:num}
\subsection{BGN anisotropies}
For the numerical approximation of an optimal design variable $\opt \in \V_m$ to the structural optimization problem \eqref{opt}, we first consider a discretization of \an{the variational inequality} \eqref{opt:alt} with differentiable $\gamma$. In general, the term $\gamma \D \gamma$ is highly nonlinear, and to facilitate the subsequent discussion regarding its discretization we first consider a matrix-type reformulation of the anisotropy density function introduced by the works of the first and third authors, see, e.g., \cite{BGNgeo,eck,vch} for more details.

Suppose for some $L \in \N$, the anisotropy density function $\gamma$ can be expressed as
\begin{align}\label{BGN:ani}
\gamma(\q) = \sum_{l=1}^L \gamma_l(\q), \text{ where } \gamma_l(\q) := \sqrt{\q \cdot \Gl \q} \text{ for } \q \in \R^d,
\end{align}
with symmetric positive definite matrices $\Gl \in \R^{d \times d}$ for $l = 1, \dots, L$. Then, a short calculation shows that 
\[
\D\gamma(\q) = \sum_{l=1}^L \frac{\Gl \q}{\gamma_l(\q)}  \quad \forall \q \in \R^d \setminus \{\0\},
\]
and recalling that $A(\cdot) := \frac{1}{2} |\gamma(\cdot)|^2$, we infer also 
\[
\D A(\q) = [\gamma \D \gamma](\q) = \gamma(\q) \sum_{l=1}^L \frac{\Gl \q}{\gamma_l(\q)} \quad \forall \q \in \R^d \setminus \{\0\}.
\]
For $\p$ close to $\q$, we now perform a linearization $\D A(\q) \approx \BB(\p) \q$, where
\begin{align*}
\BB(\p) := \begin{cases}
\displaystyle \gamma(\p) \sum_{l=1}^L \frac{\Gl}{\gamma_l(\p)} & \text{ if } \p \neq \0, \\[10pt]
\displaystyle L \sum_{l=1}^L \Gl & \text{ if } \p = \0,
\end{cases}
\end{align*}
so that 
\[
\BB(\q) \q = \D A(\q) \text{ for all } \q \in \R^d.
\]
Furthermore, as $\{\Gl\}_{l=1}^L$ are symmetric positive definite, the matrix $\BB(\p)$ is also symmetric positive definite for all $\p \in \R^d$. We term anisotropies $\gamma$ that are of the form \eqref{BGN:ani} as BGN anisotropies. Notice that for the definition of the linearized matrix $\BB$ it is sufficient to have anisotropy densities $\gamma \in C^0(\R^d)$.  However, due to our examples \eqref{D:gam} and \eqref{2D:gam:Ncon} it will turn out that the corresponding $\Gl$ are no longer constant matrices.  Thus, inspired by the works \cite{BGNgeo,eck,vch} we extend their approach to the case where $\Gl$ are piecewise constant and dependent on $\q$, as shown in the following examples.  In our numerical approximation to the optimality condition detailed in the next section we simply replace $\D A(\q) = [\gamma \D \gamma](\q)$ by $\BB(\q)\q$, where $\BB(\q)$ now involves piecewise constant matrices $\Gl(\q)$.

\begin{ex}
The convex anisotropy \eqref{D:gam} can be expressed as an extended BGN anisotropy
\begin{align*}
\gamma_{\alpha}(\q) = \sqrt{\q \cdot \GG(\q) \q},
\end{align*}
where for $\q = (q_1, \dots, q_d) \in \R^d$,
\begin{align*}
\GG(\q) = 
\begin{cases}
\displaystyle {\bf I}  & \displaystyle \text{ if } q_d \geq - \alpha |\q|,\\[10pt]
\displaystyle \mathrm{diag} \big ( 0, 0, \dots, 0, \alpha^{-2} \big ) & \displaystyle \text{ if } q_d < -\alpha |\q|.
\end{cases}
\end{align*}
In turn, this provides us with the linearization matrix $\BB(\q) = \GG(\q)$.
\end{ex}

\begin{ex}
The nonconvex anisotropy \eqref{2D:gam:Ncon} can be written as an extended BGN anisotropy 
\begin{align*}
\gamma_{\alpha,\lambda}(\q) = \sqrt{\q \cdot \GG(\q) \q},
\end{align*}
where for $\q \in \R^2$,
\begin{align*}
\GG(\q) = \begin{cases}
\displaystyle {\bf I}  & \displaystyle \text{ if } q_2 \geq - \alpha |\q| ,\\[10pt]
\displaystyle \begin{pmatrix}
\big ( \frac{1-\lambda}{\lambda} \big)^2 \frac{1}{1-\alpha^2} & - \frac{1-\lambda}{\lambda^2 \alpha} \frac{1}{\sqrt{1-\alpha^2}} \\
- \frac{1-\lambda}{\lambda^2 \alpha} \frac{1}{\sqrt{1-\alpha^2}} & \frac{1}{\lambda^2 \alpha^2} 
\end{pmatrix} & \displaystyle \text{ if } q_1 \leq 0, \, q_2 < - \alpha |\q|, \\[15pt]
\displaystyle \begin{pmatrix}
\big ( \frac{1-\lambda}{\lambda} \big)^2 \frac{1}{1-\alpha^2} & \frac{1-\lambda}{\lambda^2 \alpha} \frac{1}{\sqrt{1-\alpha^2}} \\
\frac{1-\lambda}{\lambda^2 \alpha} \frac{1}{\sqrt{1-\alpha^2}} & \frac{1}{\lambda^2 \alpha^2} 
\end{pmatrix} & \displaystyle \text{ if } q_1 > 0, \, q_2 < - \alpha |\q|.
\end{cases}
\end{align*}
Then, the linearized matrix is given by $\BB(\q) = \GG(\q)$.
\end{ex}
For numerical simulations we also consider regularizations of the convex anisotropy $\gamma_{\alpha}$ in \eqref{D:gam} of the form
\begin{align}\label{eq:HG3dnew}
\widehat{\gamma}_{\alpha,\delta}(\q) := \gamma_{\alpha}(\q) + \delta |\q|
\end{align}
for $\delta > 0$. This can also be expressed as an extended BGN anisotropy with $L = 2$, where we take $\gamma_1(\q) = \gamma_\alpha(\q)$, $\gamma_2(\q) = \delta |\q|$, 
\[
\GG_1(\q) = 
\begin{cases}
\displaystyle {\bf I}  & \displaystyle \text{ if } q_d \geq - \alpha |\q|,\\[10pt]
\displaystyle \mathrm{diag} \big ( 0, 0, \dots, 0, \alpha^{-2} \big ) & \displaystyle \text{ if } q_d < -\alpha |\q|,
\end{cases} \quad \text{ and } \quad \GG_2(\q) = \delta^2 {\bf I}\an{.}
\]
The linearization matrix $\BB(\q)$ has the form
\[
\BB(\q) =
\begin{cases}
 \widehat{\gamma}_{\alpha,\delta}(\q) \Big ( [\gamma_\alpha(\q)]^{-1} \GG_1(\q) + |\q|^{-1} \delta {\bf I} \Big ) & \text{ if } \q \neq \0, \\
2 (1+ \delta^2) {\bf I} & \text{ if } \q = \0.
\end{cases}
\]
In Figure~\ref{fig:frank3_01} we visualize the Frank diagram and Wulff shape of the anisotropy
\eqref{eq:HG3dnew} in two spatial dimensions with $\delta=0.1$ and $\alpha = 0.9$, $0.7$, $0.5$, $0.3$.
\begin{figure}
\center
\includegraphics[angle=-0,width=0.3\textwidth]{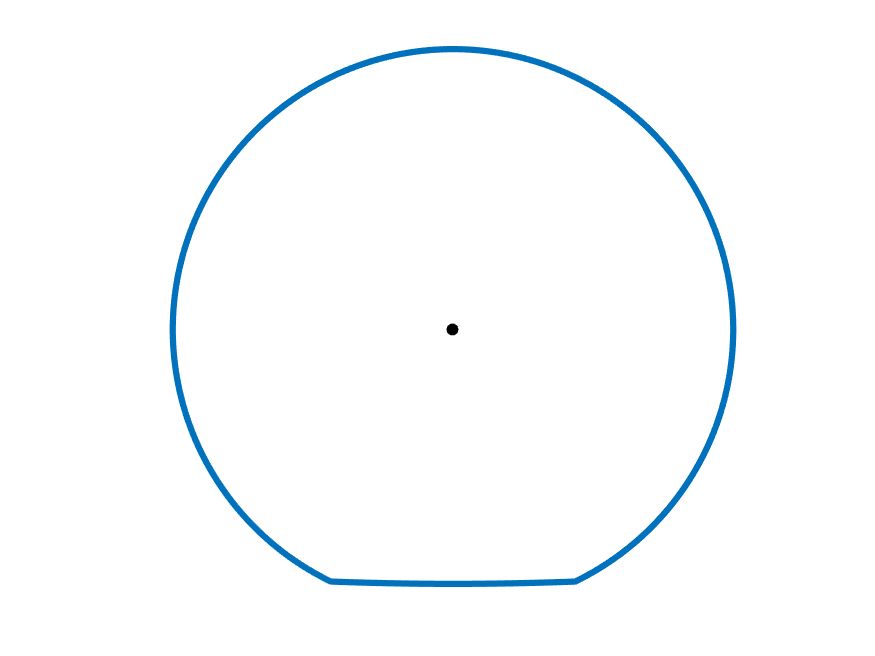} 
\includegraphics[angle=-0,width=0.3\textwidth]{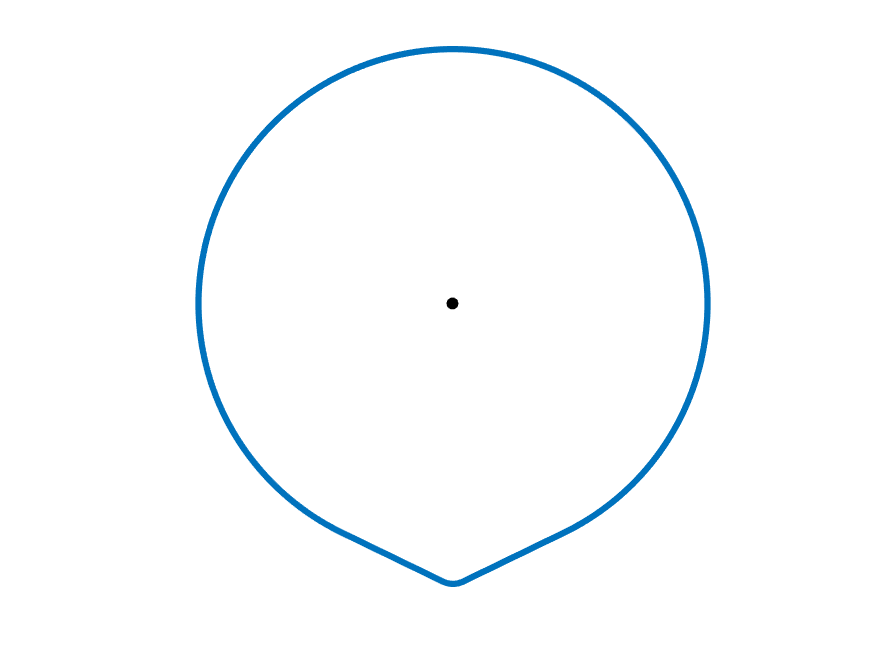}\\
\includegraphics[angle=-0,width=0.3\textwidth]{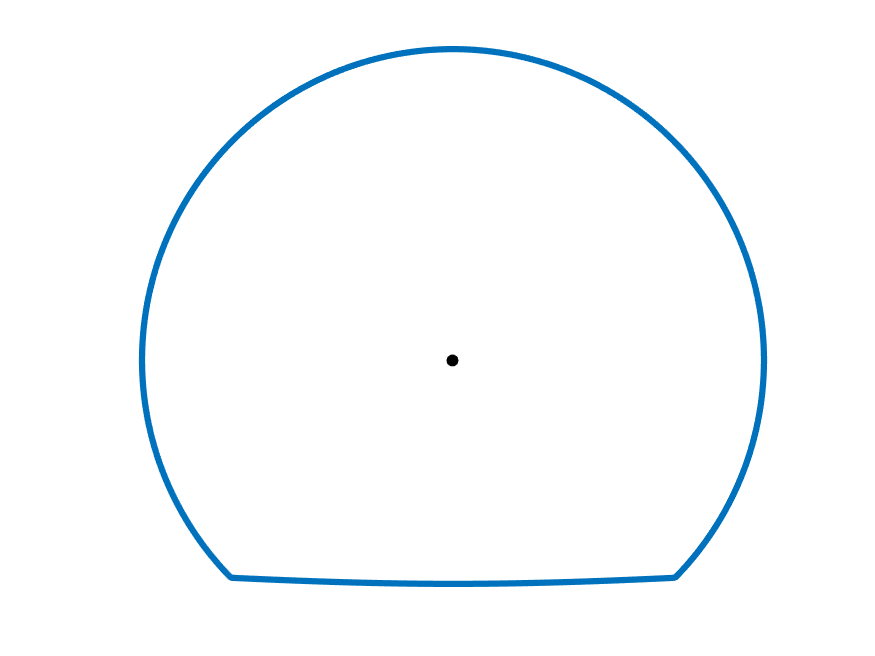}
\includegraphics[angle=-0,width=0.3\textwidth]{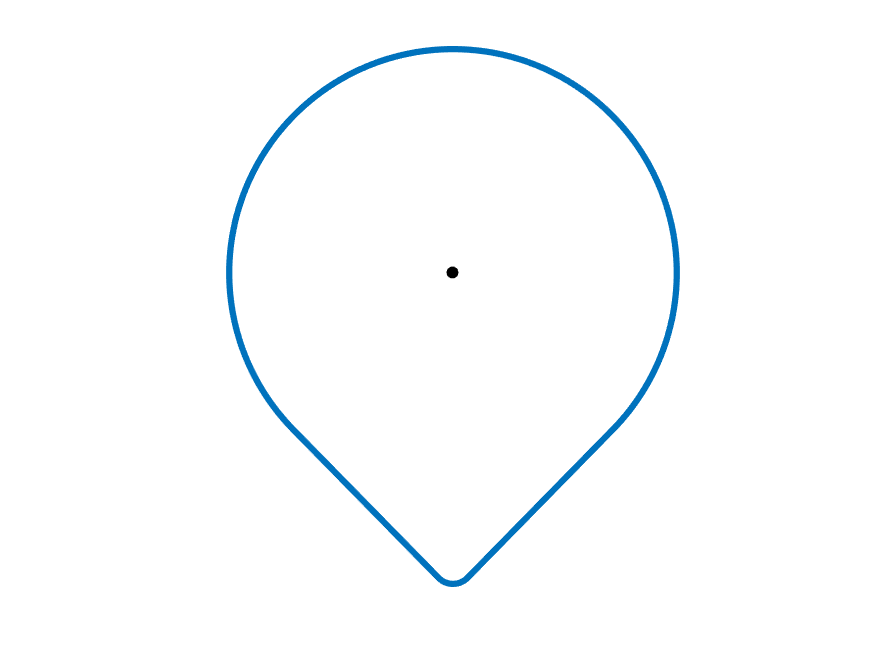}\\
\includegraphics[angle=-0,width=0.3\textwidth]{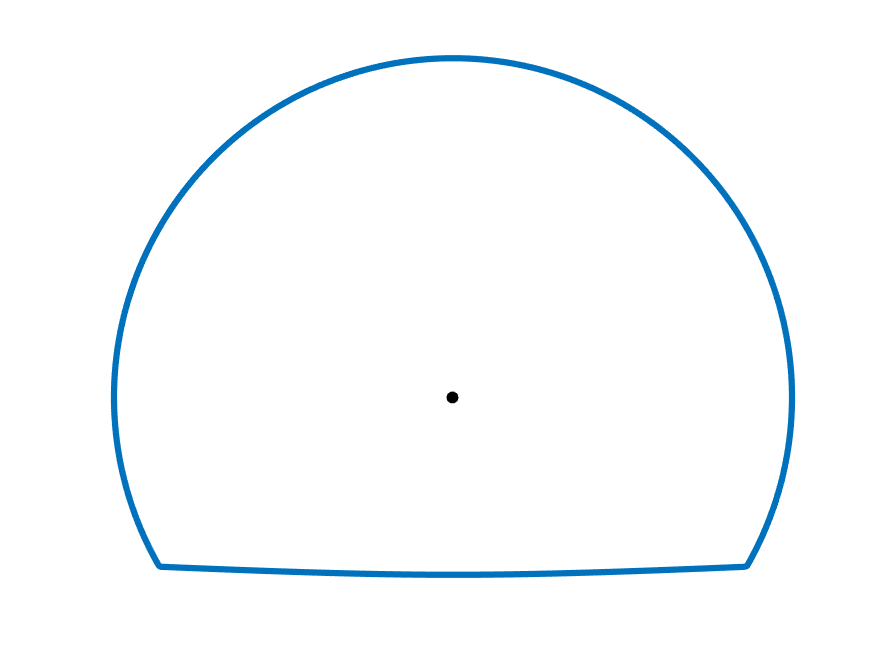}
\includegraphics[angle=-0,width=0.3\textwidth]{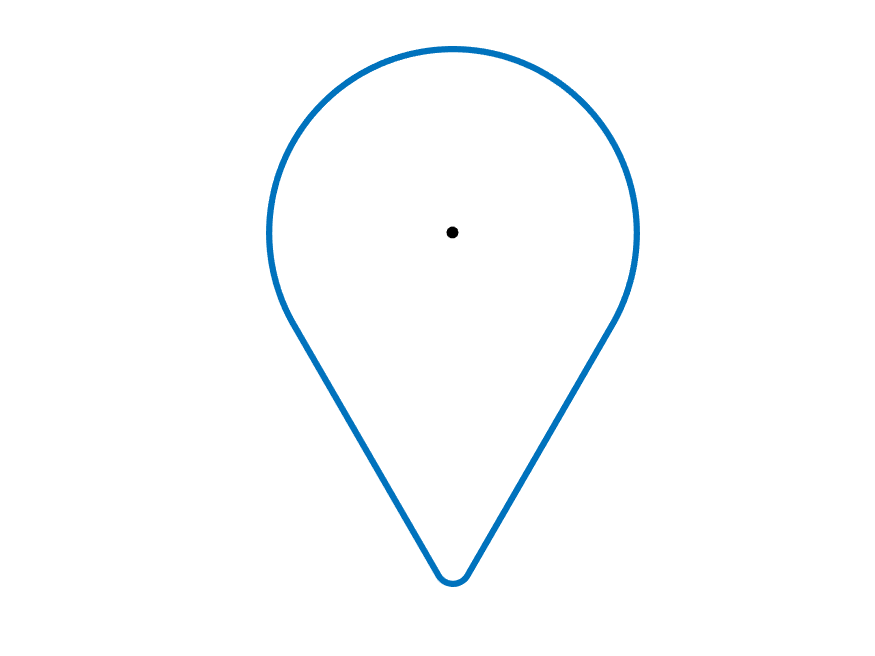}\\
\includegraphics[angle=-0,width=0.3\textwidth]{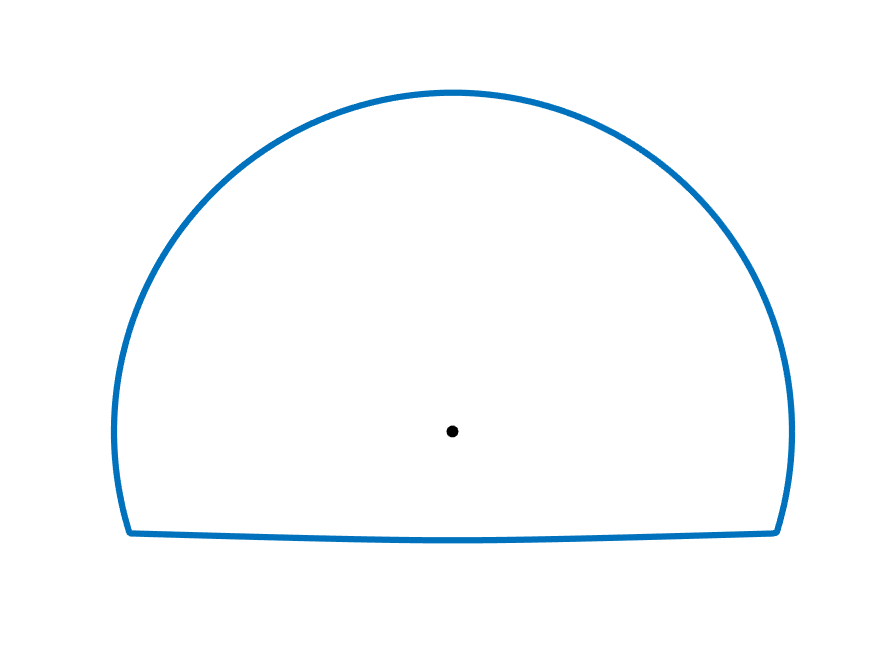}
\includegraphics[angle=-0,width=0.3\textwidth]{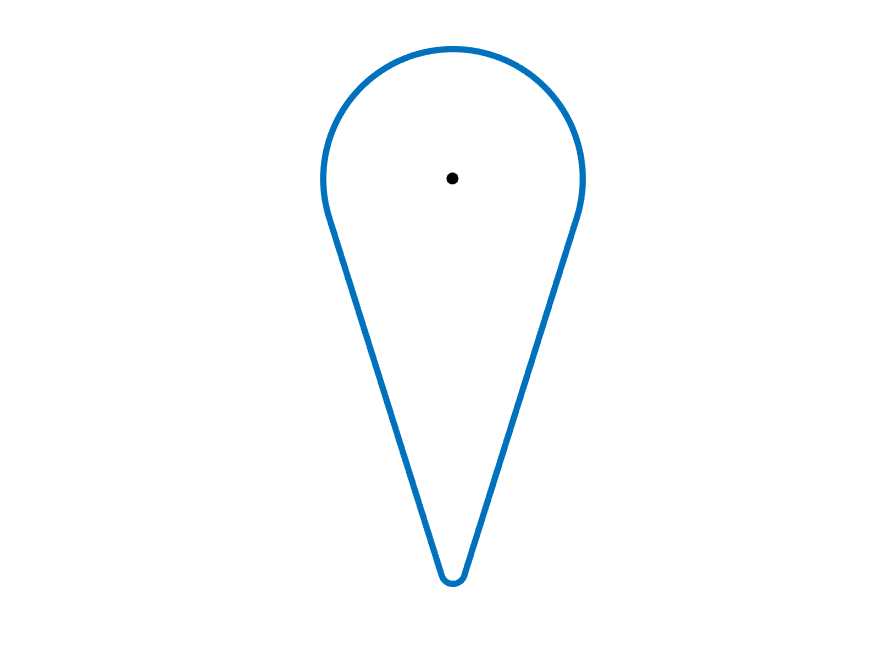}
\caption{The Frank diagram (left column) and Wulff shape (right column) for the regularized anisotropy \eqref{eq:HG3dnew} with $\delta = 0.1$. From top to bottom: $\alpha = 0.9$, $0.7$, $0.5$, $0.3$. Notice the bottom of each Frank diagram is not exactly a straight line, but slightly curved.
}
\label{fig:frank3_01}
\end{figure}%

\subsection{Finite element discretization}

Let $\mathcal{T}_{h}$ be a regular triangulation of $\Omega$ into disjoint open simplices. Associated with $\mathcal{T}_h$ are the piecewise linear finite element spaces
\begin{align*}
S^{h} = \left \{ \chi \in C^{0}(\overline\Omega) : \,  \chi_{\vert_{o}} \in P_{1}(o) \, \forall o \in \mathcal{T}_{h} \right \}
\quad\text{and}\quad \vecS^{h} =S^h \times \cdots \times S^h =  [S^{h}]^d,
\end{align*}
where we denote by $P_{1}(o)$ the set of all affine linear functions on $o$. In addition we define the convex subsets
\begin{equation} \label{eq:Khm}
\V^h = \left\{ \chi \in S^h : |\chi| \leq 1  \text{ in } \overline\Omega \right\}
\quad\text{and}\quad
\V^h_m = \left\{ \chi \in \an{\V}^h : (\chi, 1) = m (1,1) \right\},
\end{equation}
\an{where $(\cdot,\cdot)$ denotes the $L^{2}$--inner product on $\Omega$,}
as well as 
\begin{equation} \label{eq:ShD}
\vecS^h_D = \left\{ 
\bm{\eta} \in \vecS^h : \bm{\eta} = {\bf 0} \text{ on } \Gamma_D
\right\}.
\end{equation}
We now introduce a finite element approximation of the structural optimization problem \eqref{opt} and the optimality conditions described above. Given $\varphi_h^{n-1} \in \V^h_m$, find $(\u_h^n, \varphi_h^n) \in \vecS^h_D \times \V^h_m$ such that
\begin{subequations} \label{eq:FEA}
\begin{align}
&
\langle \mathcal{E}(\u_h^n),
\mathcal{E}(\veceta)\rangle_{ \C(\varphi_h^{n-1})}
= \left ( \hh(\varphi_h^{n-1}) \f, \veceta \right)^h + \int_{\Gamma_g} \g \cdot \veceta \qquad \forall \veceta \in \vecS^h_D\,,
\label{eq:FEAa}\\
& \left( \tfrac{\eps}{\tau} (\varphi_h^n - \varphi_h^{n-1}) 
- \tfrac{\widehat \alpha}{\eps} \varphi_h^n, \chi - \varphi_h^n \right)^h
+ \widehat \alpha \eps ( \BB(\varphi_h^{n-1}) \nabla \varphi_h^n, 
\nabla(\chi - \varphi_h^n)) \nonumber \\ & \qquad
\geq  \langle \E(\u_h^n),
\E(\u_h^n) (\chi - \varphi_h^n) 
\rangle_{\C'(\vp_h^{n-1})}
\qquad \forall \chi \in \V^h_m\,. \label{eq:FEAb}
\end{align}
\end{subequations}
Here $\tau$ denotes the time step size, 
$(\cdot,\cdot)^{h}$ is the usual mass lumped $L^{2}$--inner product on
$\Omega$, and $\inn{\bm{A},\bm{B}}_{\C} = \int_\Omega \C  \bm{A} : \bm{B} \dx$ for any fourth order tensor $\C$ and any matrices $\bm{A}$ and $\bm{B}$. 
We implemented the scheme \eqref{eq:FEA} with the help of the finite element 
toolbox ALBERTA, see \cite{Alberta}. 
To increase computational efficiency, we employ adaptive meshes, which have a 
finer mesh size $h_{f}$ within the diffuse interfacial regions and a coarser 
mesh size $h_{c}$ away from them, see \cite{voids3d,voids}
for a more detailed description. 
Clearly, the system \eqref{eq:FEAa} decouples, 
and so we first solve the linear system \eqref{eq:FEAa} in order to obtain
$\u_h^n$, and then solve the variational inequality \eqref{eq:FEAb} 
for $\varphi_{h}^{n}$.
We employ the package LDL, see \cite{Davis05}, together with the sparse 
matrix ordering AMD, see \cite{AmestoyDD04}, in order to solve \eqref{eq:FEAa}.
For the variational inequality \eqref{eq:FEAb} we employ a secant method
as described in \cite{BloweyE93} to satisfy the mass constraint
$(\varphi_{h}^{n}, 1) = m$, and use a nonlinear multigrid method
similar to \cite{Kornhuber96} for solving the variational inequalities 
over $\an{\V}^h$ that arise as the inner problems for the secant iterations. The second method always converged in at most five steps. Finally, to increase the efficiency of the numerical computations in this 
paper, at times we exploit the symmetry of the problem and perform the 
computations in question only on half of the desired domain $\Omega$. In those
cases we prescribe ``free-slip'' boundary conditions for the displacement field
$\u_h^n$ on the symmetry plane $\Gamma_{D_i}$, that is, we replace
\eqref{eq:ShD} with
\begin{equation} \label{eq:ShDi}
\vecS^h_D = \left\{ 
\veceta = (\eta_1,\ldots,\eta_d) \in \vecS^h 
: \veceta = {\bf 0} \text{ on } \Gamma_D 
\ \text{ and }\ 
\eta_i = 0 \text{ on } \Gamma_{D_i}, i = 1,\ldots,d
\right\}.
\end{equation}
All computations performed in this work are for spatial dimension $d = 2$. \an{For} the remainder of the paper we consider the quadratic interpolation function \eqref{eq:gphic} in the elasticity tensor $\C(\vp)$, forcing $\bm{f} = \0$, objective functional weightings $\beta = 1$ and $\widehat{\alpha} = 1$ unless further specified.

\subsection{Numerical simulations}
\subsubsection{Optimality condition without elasticity}
We first investigate the setting with $\g = \0$, so that \eqref{eq:FEAa} yields $\u_h^n = \0$ and \eqref{eq:FEAb} reduces to an Allen--Cahn variational inequality on $\V_m^h$. From the discussion in Section \ref{sec:aniso}, this is a phase field approximation of the volume preserving anisotropic mean curvature flow (due to the mass constraint), and it is expected that the long time behavior of solutions to display the Wulff shape \eqref{Wulff} corresponding to the anisotropy $\gamma$, see Figure~\ref{fig:frank3_01}. In Figure~\ref{fig:2dcircle_ani} we display snapshots of the solution at times $t = 0$, $0.001$, $0.005$ and $0.03$ with anisotropy \eqref{eq:HG3dnew} for parameter values $\alpha = 0.5$, $\delta = 0.1$ and $\eps = 1/(16\pi)$.  We notice from the bottom plot of the discrete Ginzburg--Landau energy $\mathcal{E}_\gamma^h(\varphi_h^n) = (\eps A(\nabla \varphi_h^n) + \tfrac{1}{\eps}\Psi(\varphi_h^n),1)^h$ that it is decreasing over time, and on the top right we observe the expected Wulff shape attained near equilibrium.
\begin{figure}[h]
\center
\includegraphics[angle=-0,width=0.2\textwidth]{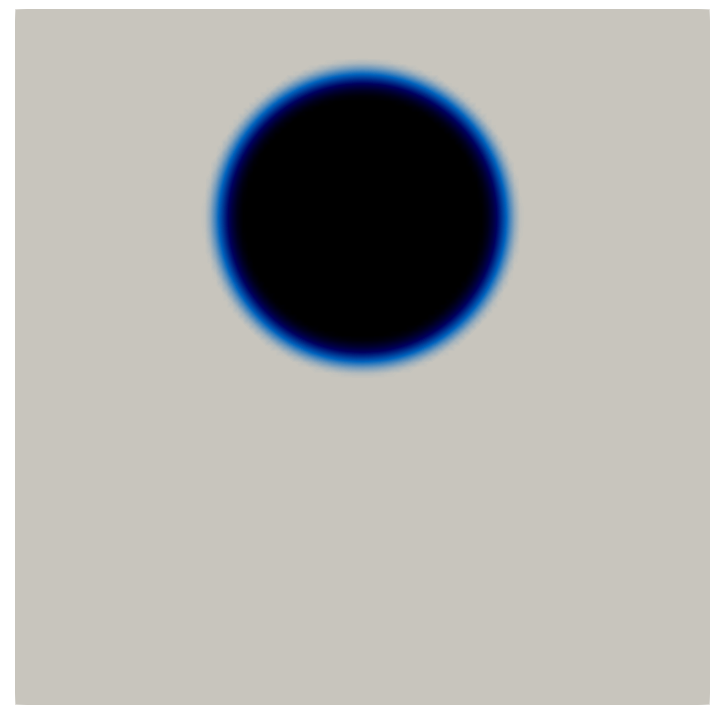}
\includegraphics[angle=-0,width=0.2\textwidth]{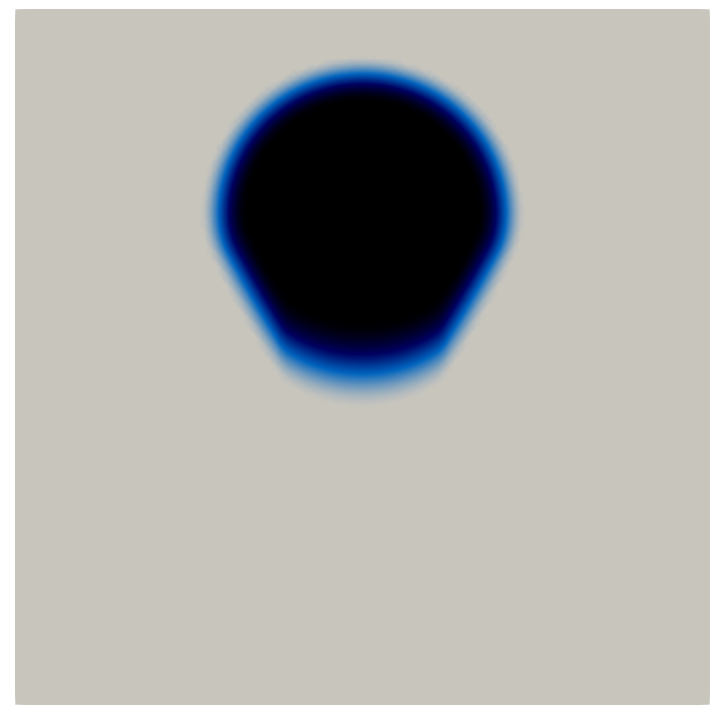}
\includegraphics[angle=-0,width=0.2\textwidth]{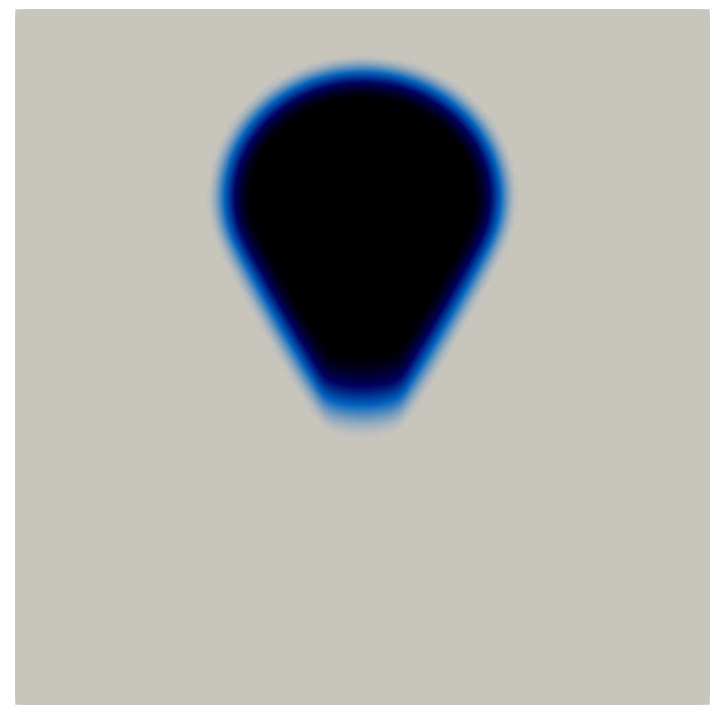}
\includegraphics[angle=-0,width=0.2\textwidth]{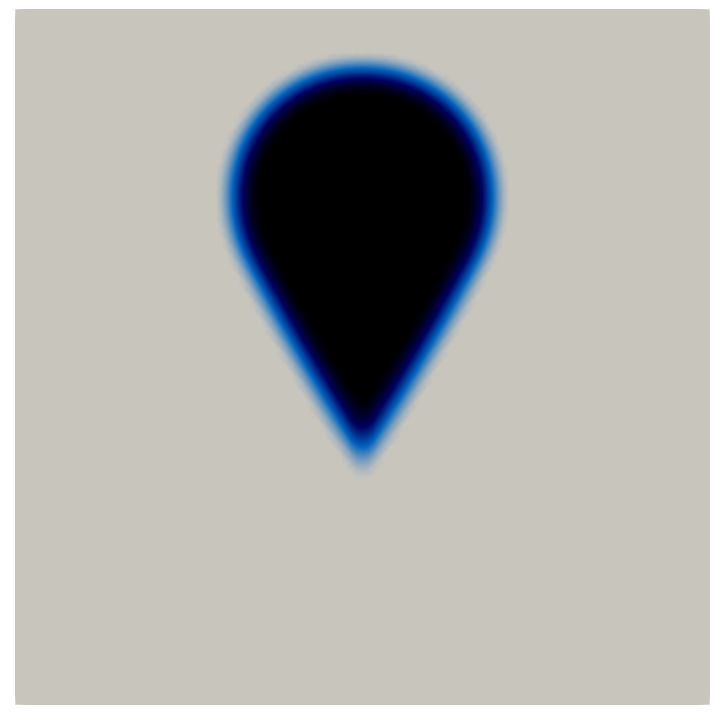}
\includegraphics[angle=-90,width=0.45\textwidth]{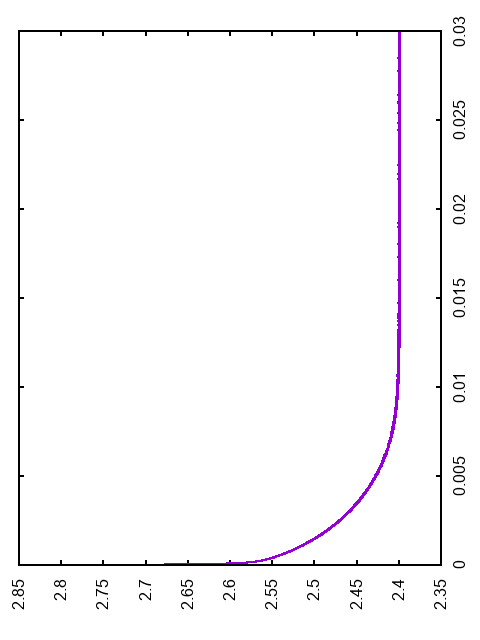}
\caption{(Top) Snapshots of the solution with anisotropy \eqref{eq:HG3dnew} for $\alpha = 0.5$, $\delta=0.1$ and $\eps = 1/(16\pi)$ at times $t = 0$, $0.001$, $0.005$, $0.03$. (Below) A plot of discrete Ginzburg--Landau energy $\mathcal{E}_\gamma^h(\varphi_h^n) = (\eps A(\nabla \varphi_h^n) + \tfrac{1}{\eps}\Psi(\varphi_h^n),1)^h$ over time.
}
\label{fig:2dcircle_ani}
\end{figure}

\subsubsection{Dripping effect of a straight interface}
Next, we investigate the role of the convexity/non-convexity of $\gamma$ in the generation of the dripping effect \cite{AllaireDEFM17,Amir,Qian}. We solve \eqref{eq:FEA} with $\g = \0$ and initial condition taken as some large perturbation of a straight interface. Taking \eqref{eq:HG3dnew} as the anisotropy with $\alpha = 0.5$, $\delta = 0.1$ and $\eps = 1/(32 \pi)$, in Figure~\ref{fig:2dwigglebig32pi_ani} we display the snapshots of the solution at times $t = 0$, $0.001$, $0.005$ and $0.02$, where it is clear that the oscillations are dampened over time.  Thus, with a convex anisotropy the dripping effect is suppressed. In contrast, taking $\gamma$ as the non-convex anisotropy \eqref{2D:gam:Ncon} with $\lambda = 0.5$ now yields Figure~\ref{fig:2dwigglebig32pi_newnc1_ani}, where the corresponding snapshots of the solution at times $t = 0$, $0.001$, $0.005$, $0.02$ are displayed. Here we clearly observe the \an{behavior} described in Example \ref{eg:Frank2}. 

\begin{figure}[h]
\centering
\includegraphics[angle=-0,width=0.2\textwidth]{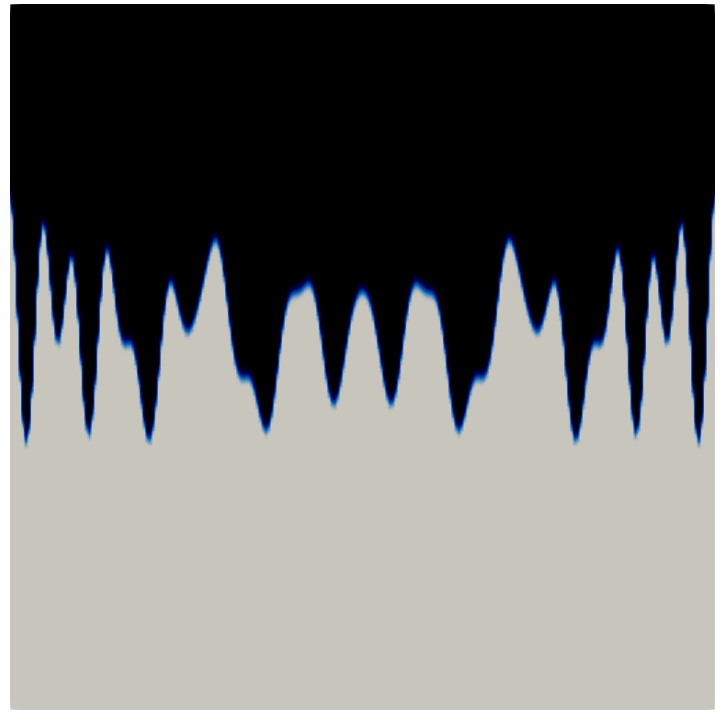}
\includegraphics[angle=-0,width=0.2\textwidth]{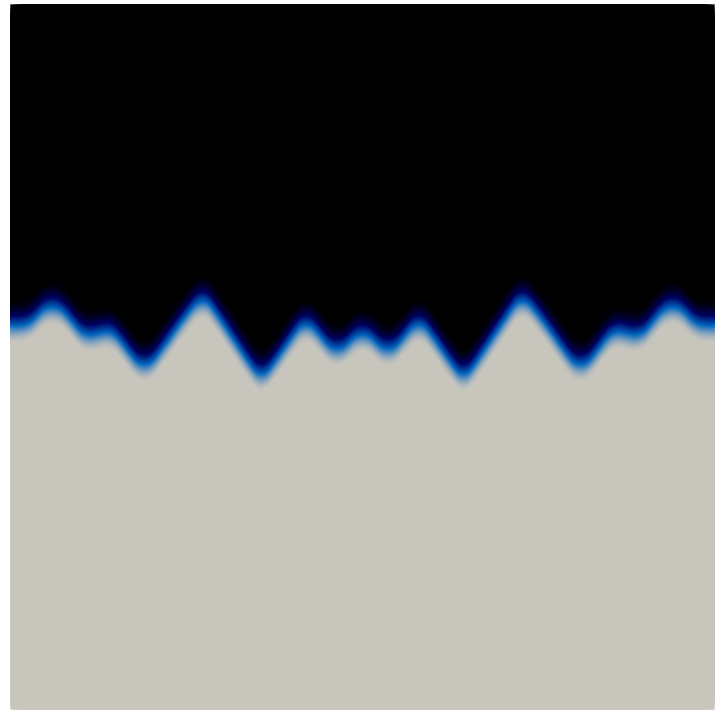}
\includegraphics[angle=-0,width=0.2\textwidth]{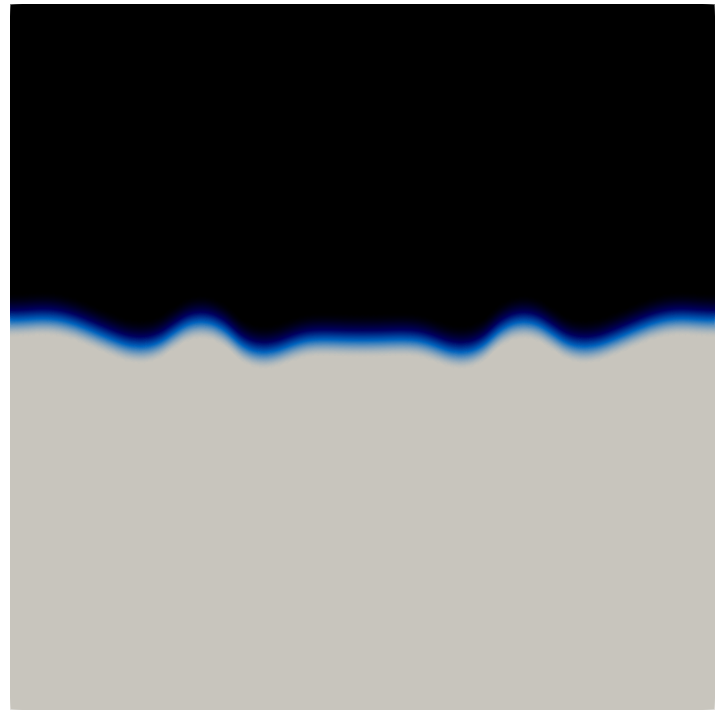}
\includegraphics[angle=-0,width=0.2\textwidth]{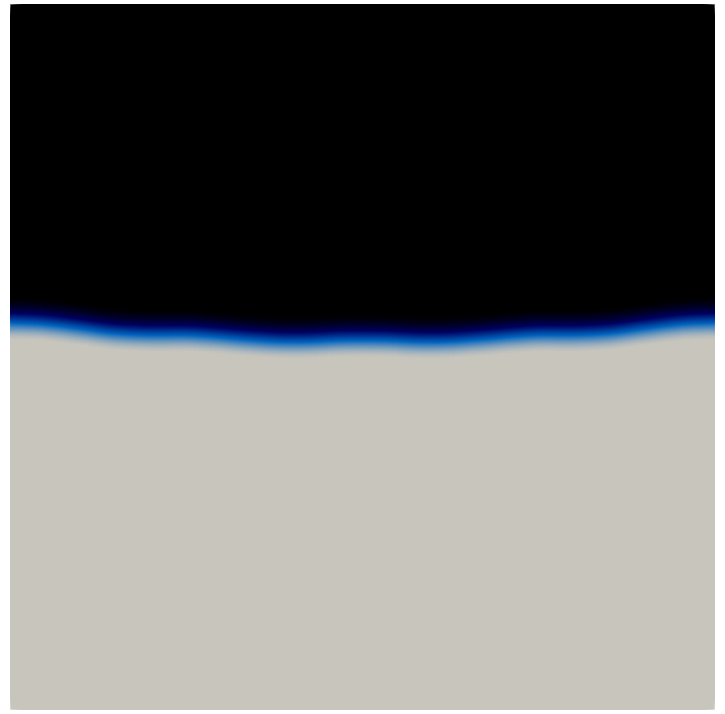}
\caption{Snapshots of the solution with convex anisotropy \eqref{eq:HG3dnew} for $\alpha = 0.5$, $\delta=0.1$ and $\eps = 1/(32\pi)$ at times $t = 0$, $0.001$, $0.005$, $0.02$.}
\label{fig:2dwigglebig32pi_ani}
\end{figure}
\begin{figure}[h]
\centering
\includegraphics[angle=-0,width=0.2\textwidth]{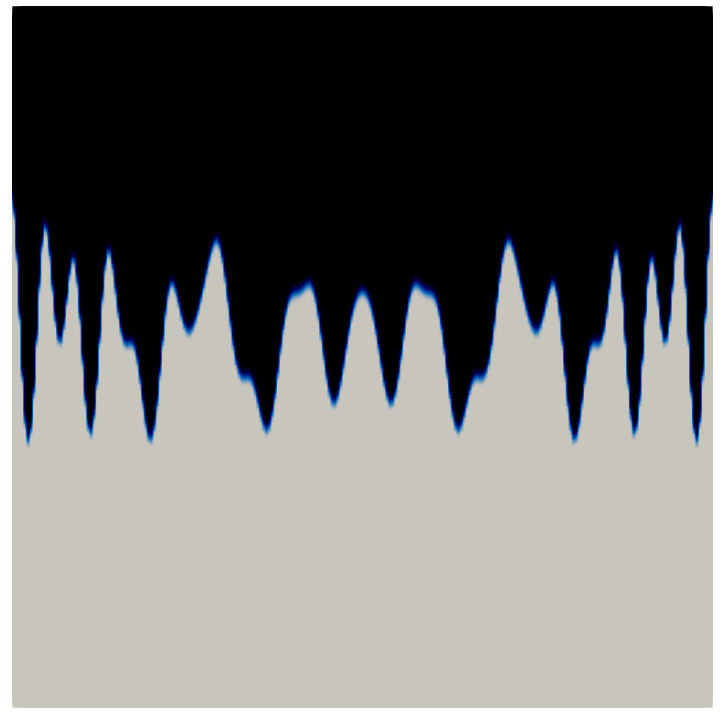}
\includegraphics[angle=-0,width=0.2\textwidth]{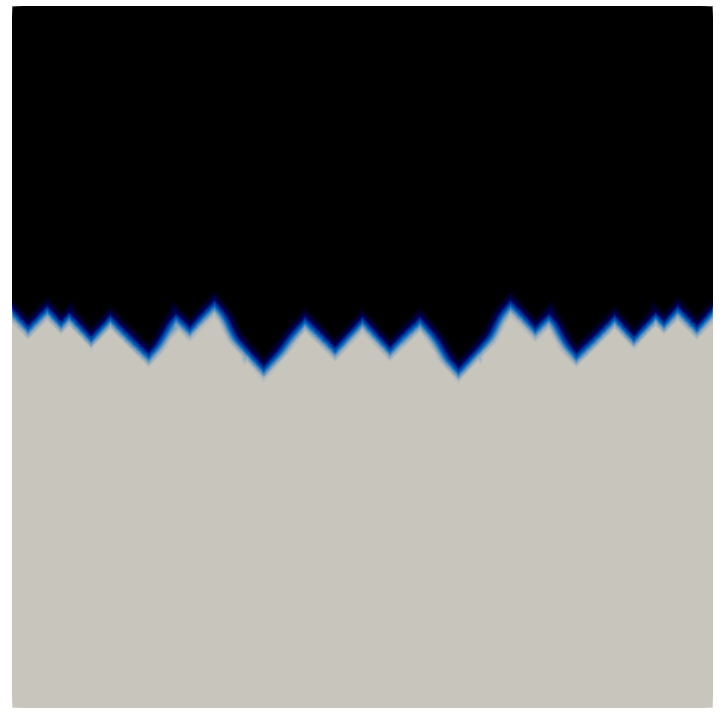}
\includegraphics[angle=-0,width=0.2\textwidth]{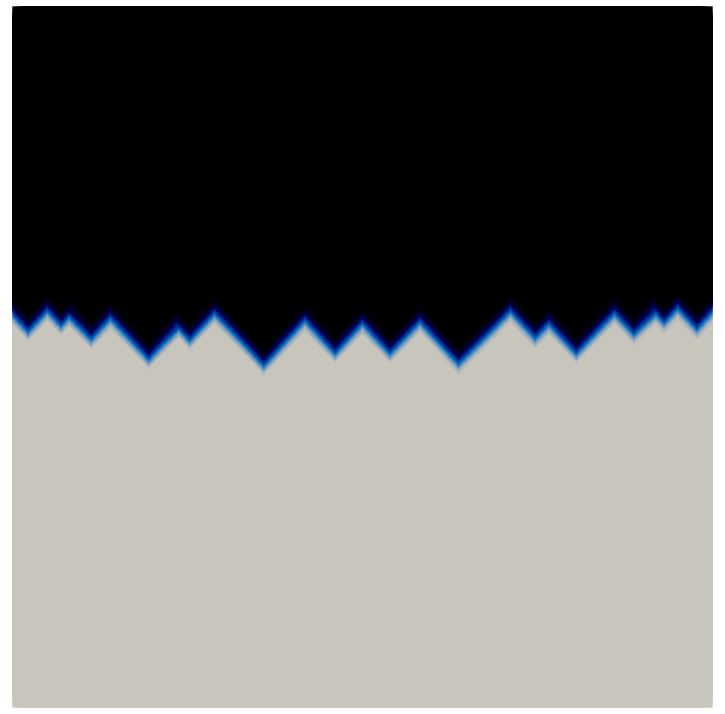}
\includegraphics[angle=-0,width=0.2\textwidth]{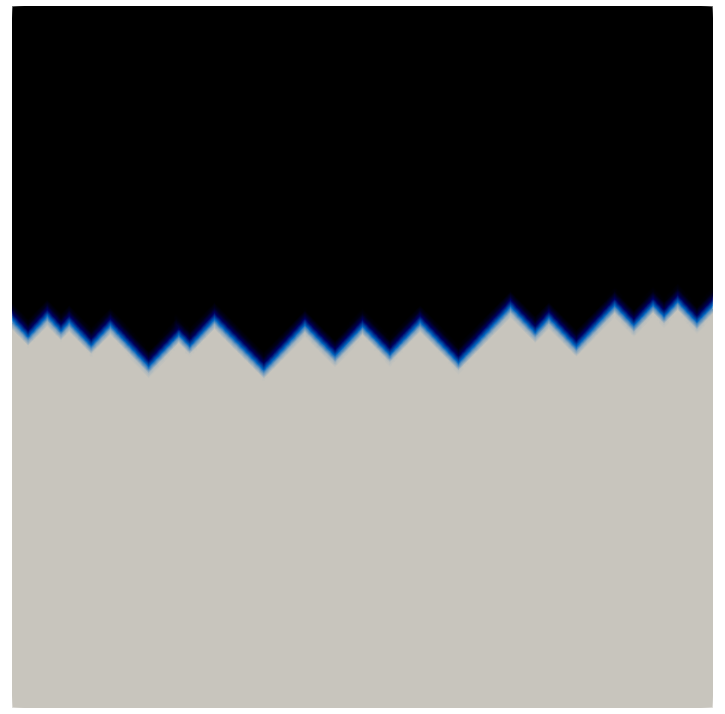}
\caption{Snapshots of the solution with non-convex anisotropy \eqref{2D:gam:Ncon} for $\alpha = 0.5$, $\delta=0.1$, $\lambda = 0.5$ and $\eps = 1/(32\pi)$ at times $t = 0$, $0.001$, $0.005$, $0.02$.
}
\label{fig:2dwigglebig32pi_newnc1_ani}
\end{figure}

\subsubsection{\an{Cantilever beam computation}}
Inspired by \cite[Figure~8]{AllaireDEFM17}, for the state system \eqref{state}, we consider the setting
\begin{equation} \label{eq:2dAllaire}
\begin{aligned}
\Omega & = (-\tfrac12,\tfrac12) \times (-\tfrac12,\tfrac12),\quad
\Gamma_D = (-\tfrac14,\tfrac14) \times \{-\tfrac12\},\\ 
\Gamma_g & = (-0.02, 0.02) \times \{\tfrac12\},\quad
\bm{g} = \begin{pmatrix} g \\ 0 \end{pmatrix},\end{aligned}
\end{equation}
with $\Gamma_0 = \pd \Omega \setminus (\Gamma_D \cup \Gamma_g)$ \an{and $g$ a positive constant}. Note that the domain $\Omega$ we use here is larger than the setting in \cite{AllaireDEFM17}. This is done to to avoid the influence of the domain boundary $\pd \Omega$ on the growth of the interfacial layer $\{|\varphi_h| < 1\}$. As in \cite{AllaireDEFM17} we consider an elastic material with a normalized Young's modulus $E = 1$ and a Poisson ratio $\nu = 0.33$.  The corresponding Lam\'e constants in the elasticity tensor \eqref{Lame} can be calculated through the well-known relations
\begin{equation} \label{eq:YoungPoisson}
\mu = \frac{E}{2(1+\nu)}
\quad\text{and}\quad
\lambda = \frac{E\nu}{(1+\nu)(1-2\nu)}.
\end{equation}
For $g = 5$ and $\widehat \alpha = 0.5$, we solve the full optimality condition \eqref{eq:FEA} with the convex regularized anisotropy \eqref{eq:HG3dnew} and compare the steady states for $\delta = 0.1$ and a range of $\alpha \an{=1, 0.7, 0.5, 0.2}$. Note that smaller values of $\alpha$ indicate a stronger anisotropy.  The results are displayed in Figure~\ref{fig:2dAFig8_32pi_ani_many}, where we taken random initial data $\varphi_h^0$ with mass $m = \frac{1}{|\Omega|} \int_\Omega \varphi_h^0 \dx = 0.7$. Note that $\alpha = 1$ corresponds to the isotropic case, see Remark \ref{rem:iso}, and as $\alpha$ decreases, we observe the interior connecting bridges in the cantilever beam become steeper in slope, up until $\alpha = 0.2$ where a design without interior structure is favored.
\begin{figure}[h]
\center
\includegraphics[angle=-0,width=0.2\textwidth]{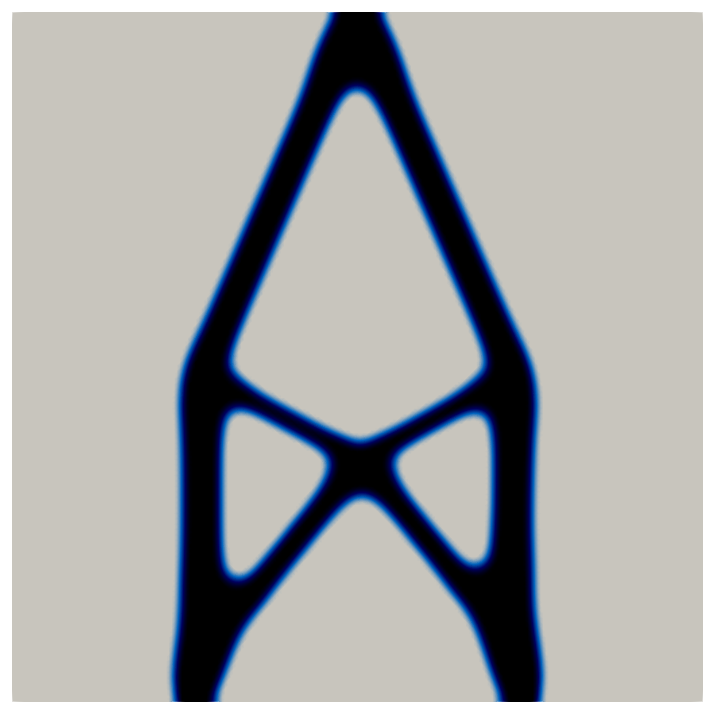}
\includegraphics[angle=-0,width=0.2\textwidth]{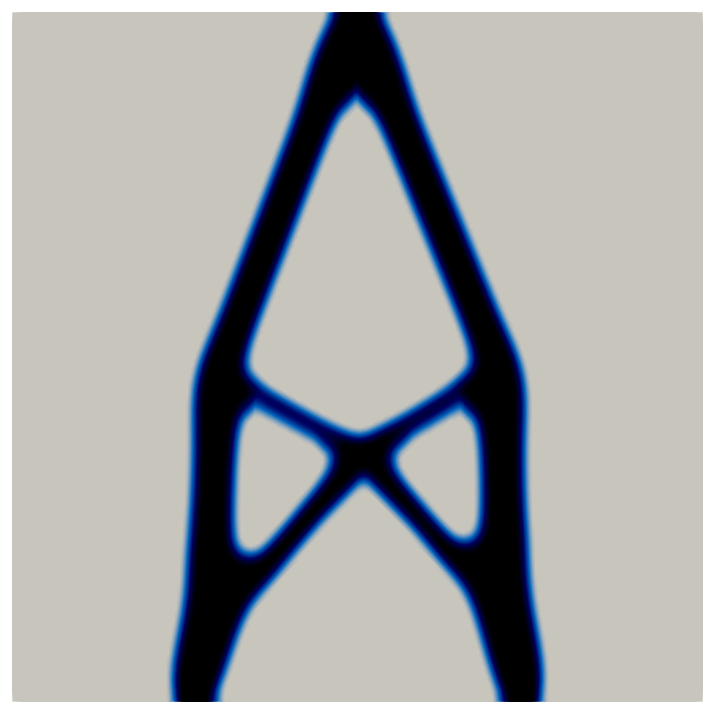}
\includegraphics[angle=-0,width=0.2\textwidth]{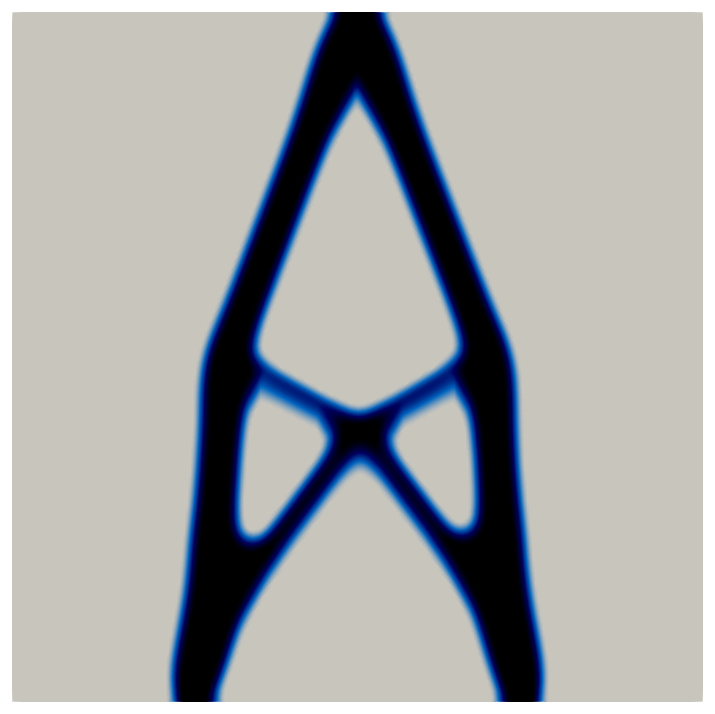}
\includegraphics[angle=-0,width=0.2\textwidth]{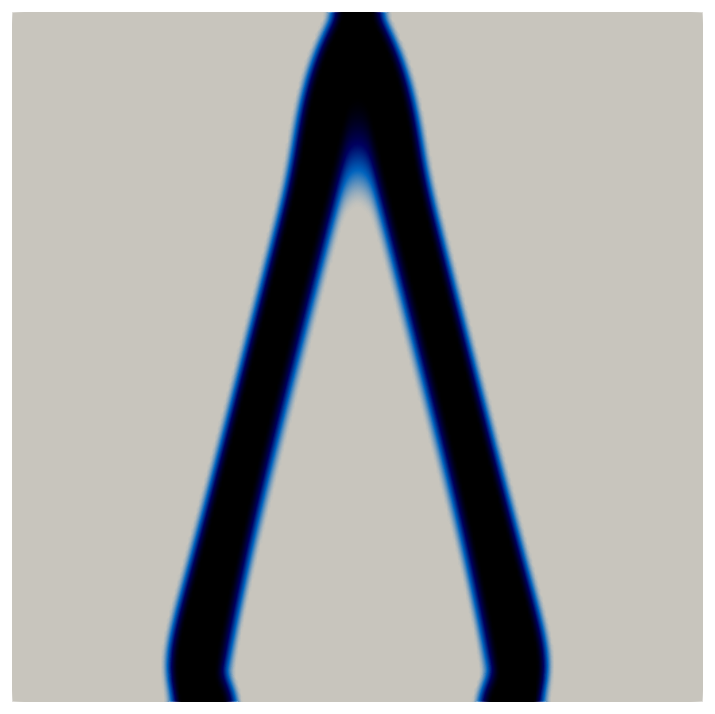}
\caption{Cantilever beam computation for $\eps=1/(32\pi)$ and $\bm{g} = (5,0)^\top$ on $[-0.02,0.02] \times \{0.5\}$. The solutions at time $t=0.04$, starting from random initial data with mass $m = 0.7$.  Anisotropy \eqref{eq:HG3dnew} used with $\delta=0.1$. From left to right: $\alpha = 1$ (isotropic case), $0.7$, $0.5$, $0.2$.
}
\label{fig:2dAFig8_32pi_ani_many}
\end{figure}

\begin{figure}[h]
\center
\includegraphics[angle=-0,width=0.2\textwidth]{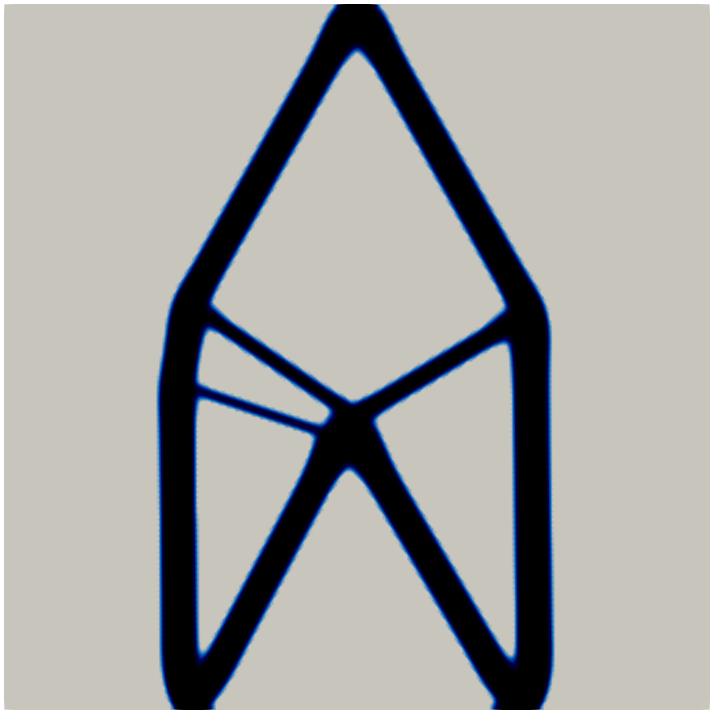}
\includegraphics[angle=-0,width=0.2\textwidth]{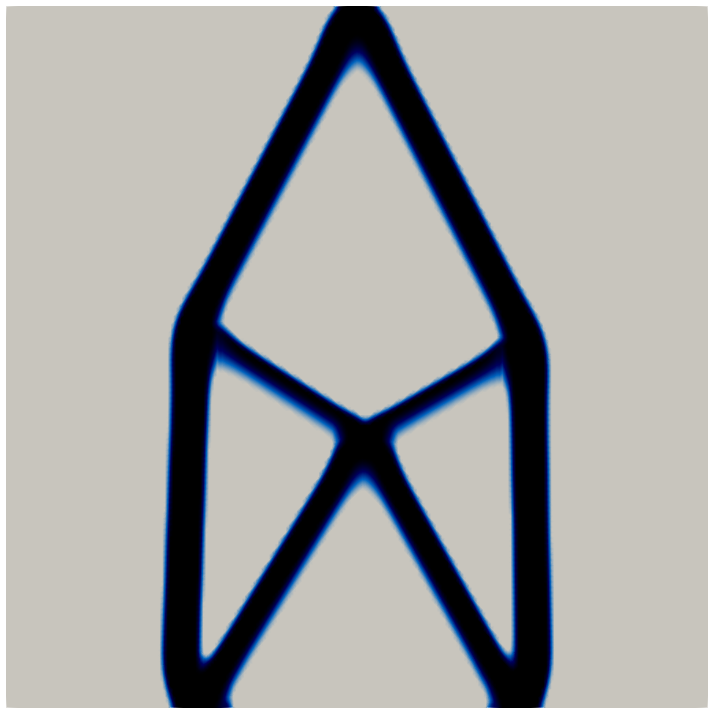}
\includegraphics[angle=-0,width=0.2\textwidth]{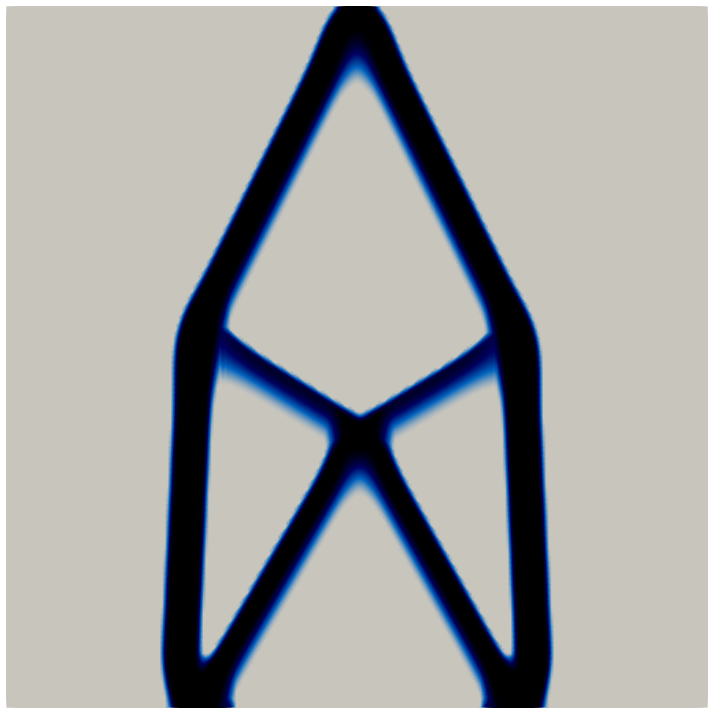}
\includegraphics[angle=-0,width=0.2\textwidth]{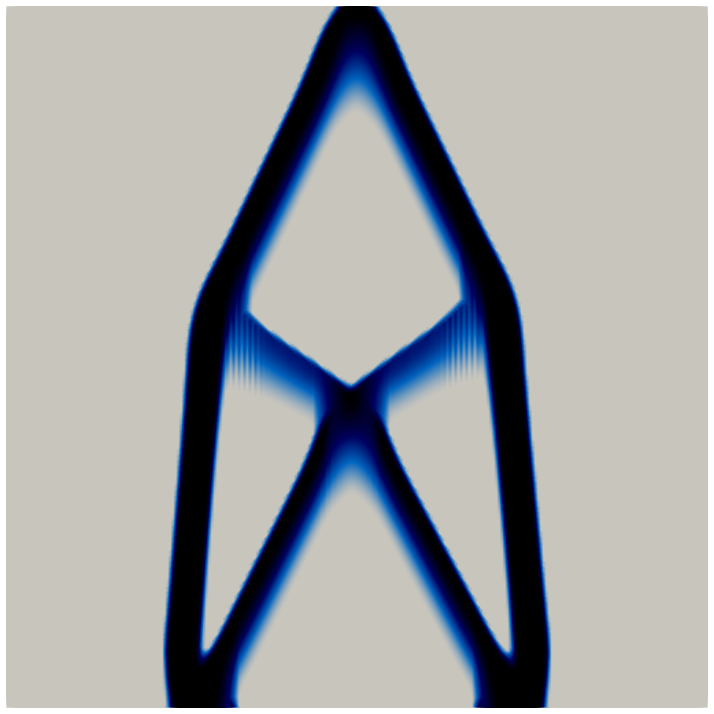}\\
\includegraphics[angle=-0,width=0.2\textwidth]{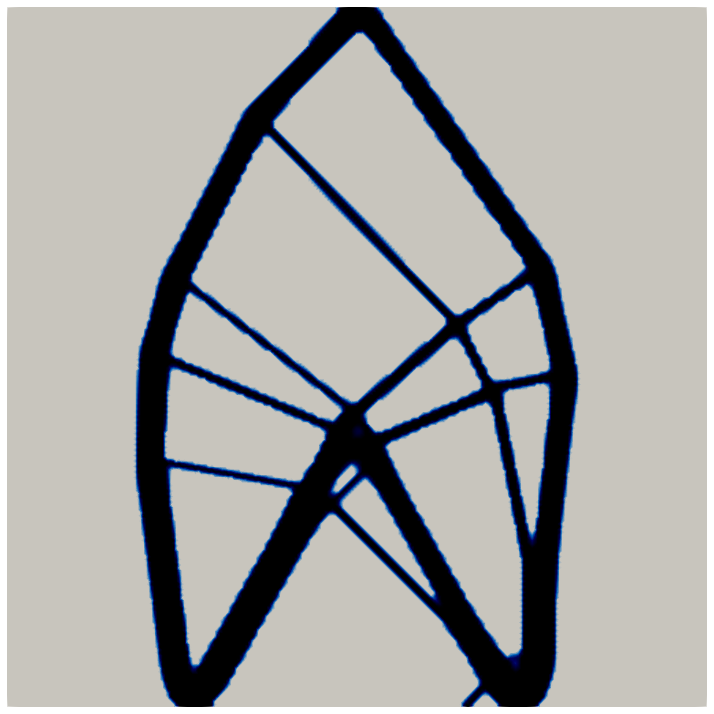}
\includegraphics[angle=-0,width=0.2\textwidth]{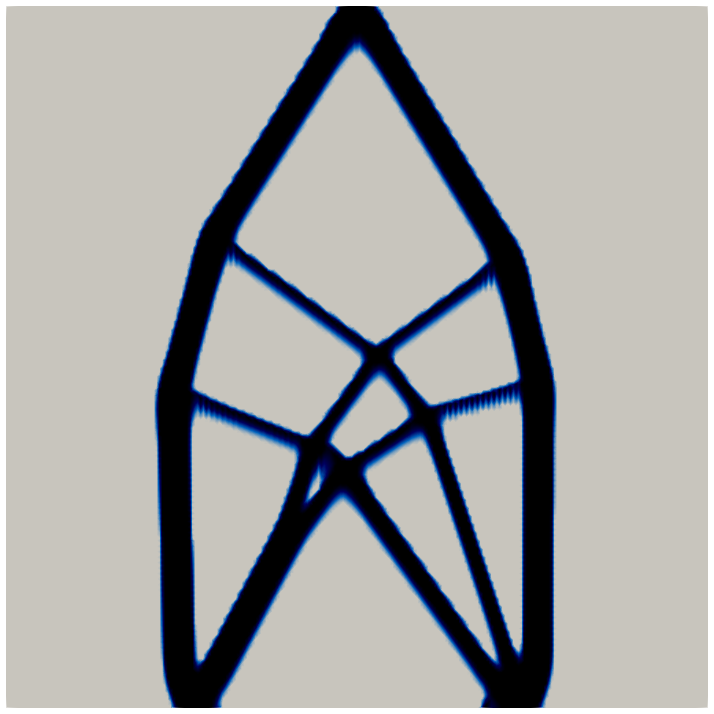}
\includegraphics[angle=-0,width=0.2\textwidth]{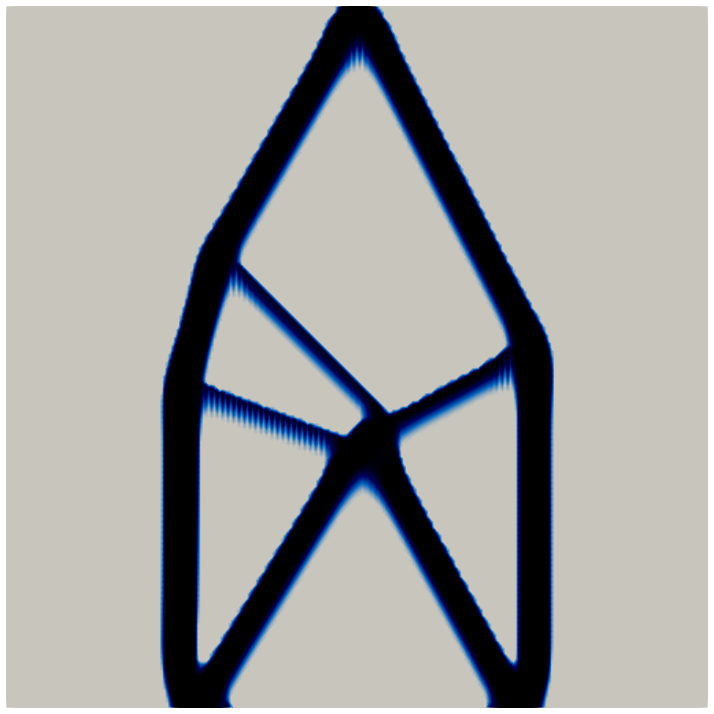}
\includegraphics[angle=-0,width=0.2\textwidth]{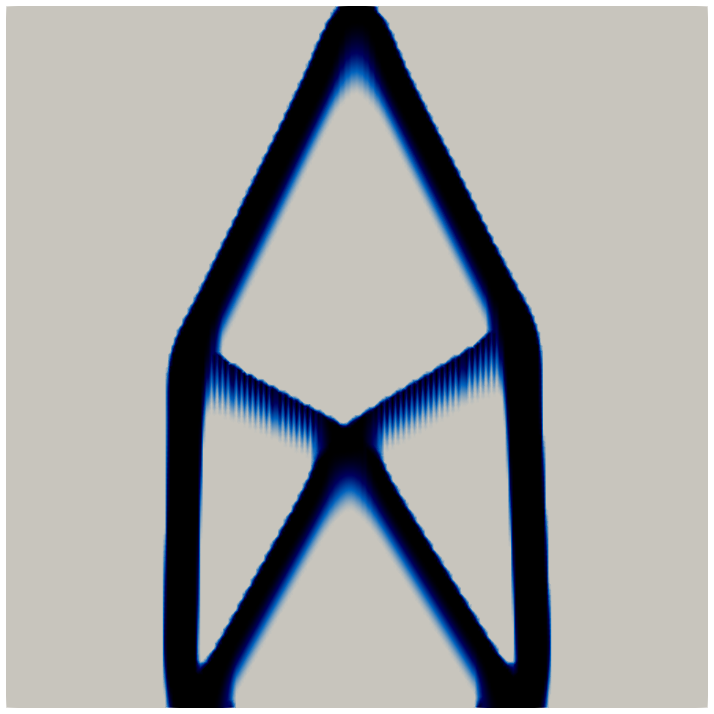}
\caption{Cantilever beam computation for $\eps=1/(32\pi)$ and $\bm{g} = (g,0)^\top$ on $[-0.02,0.02] \times \{0.5\}$. The solutions at time $t=0.04$, starting from random initial data with mass $m = 0.7$. Anisotropy \eqref{eq:HG3dnew} used with $\delta=0.1$.
(Top) $g = 30$ with $\alpha = 1$, $0.3$, $0.2$, $0.1$. (Bottom) $g = 50$ with $\alpha = 1$, $0.3$, $0.2$, $0.1$.  Left-most column corresponds to the isotropic case.
}
\label{fig:2dAFig8_32pi_ani_forces2}
\end{figure}

\begin{figure}[h]
\center
\includegraphics[angle=-0,width=0.2\textwidth]{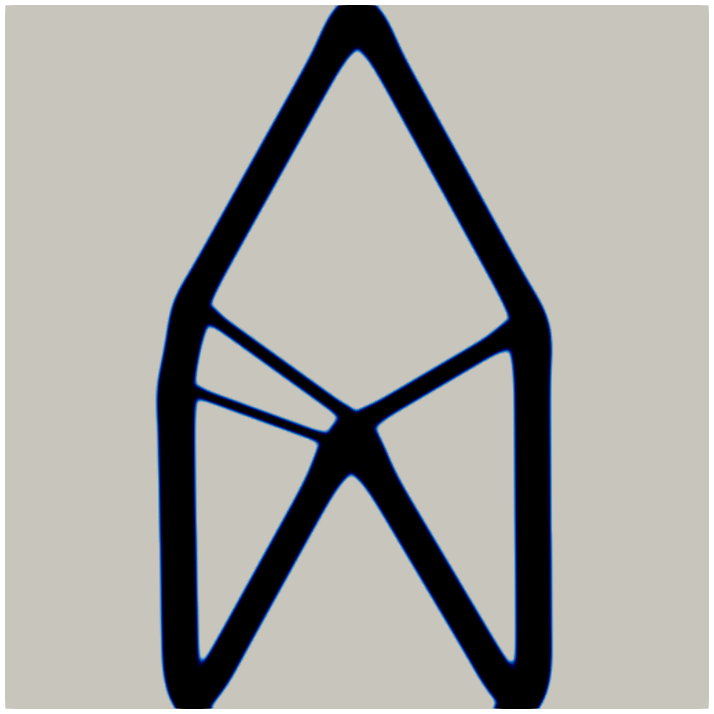}
\includegraphics[angle=-0,width=0.2\textwidth]{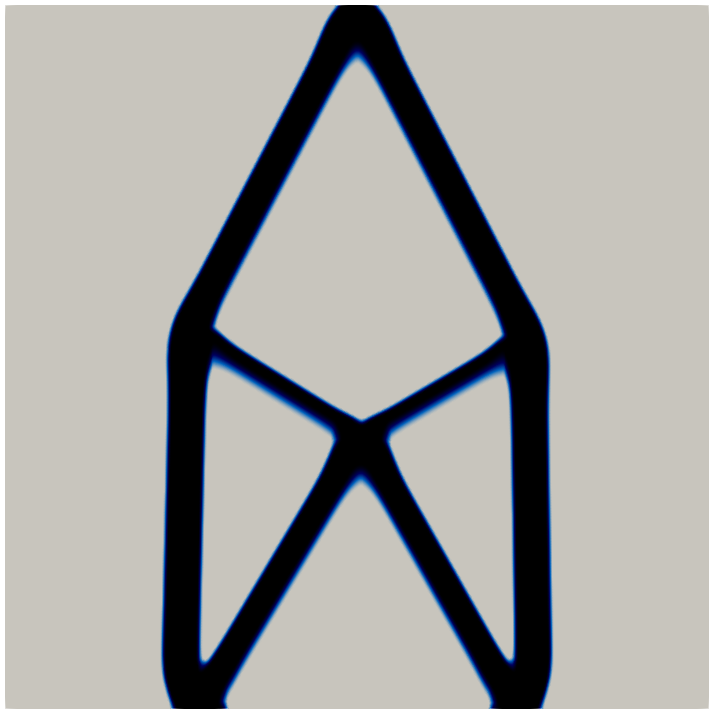}
\includegraphics[angle=-0,width=0.2\textwidth]{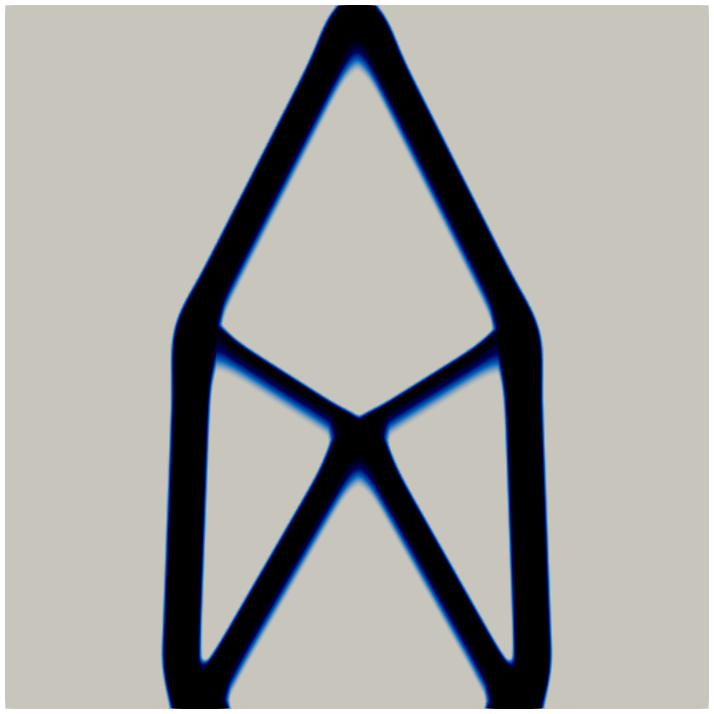}
\includegraphics[angle=-0,width=0.2\textwidth]{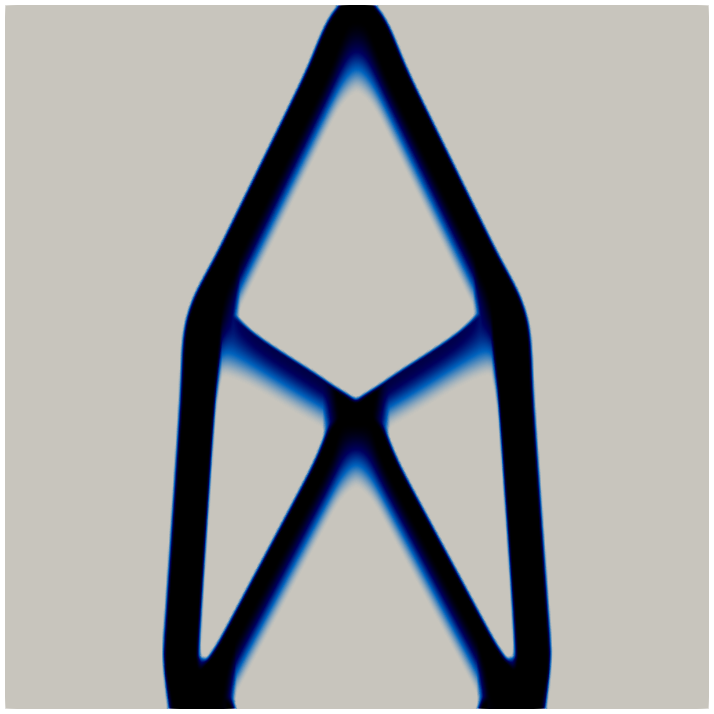}
\includegraphics[angle=-0,width=0.2\textwidth]{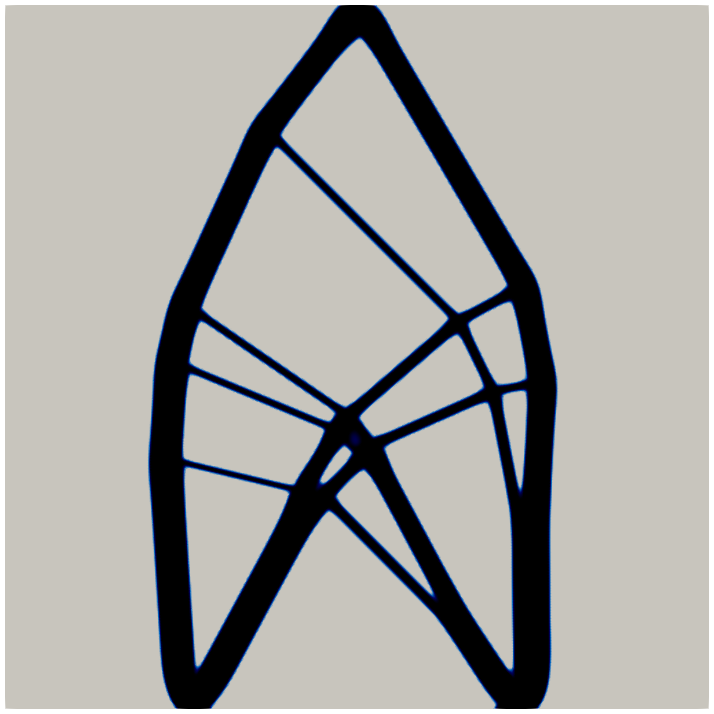}
\includegraphics[angle=-0,width=0.2\textwidth]{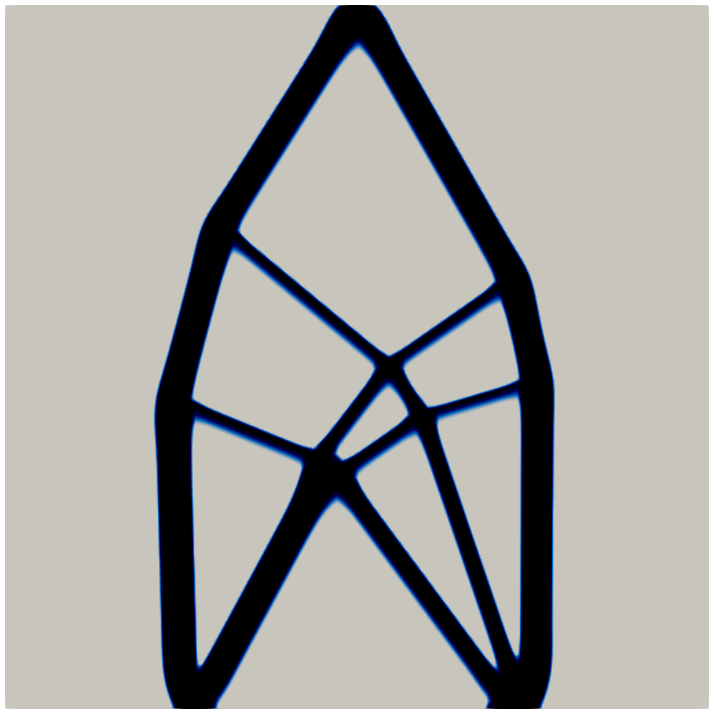}
\includegraphics[angle=-0,width=0.2\textwidth]{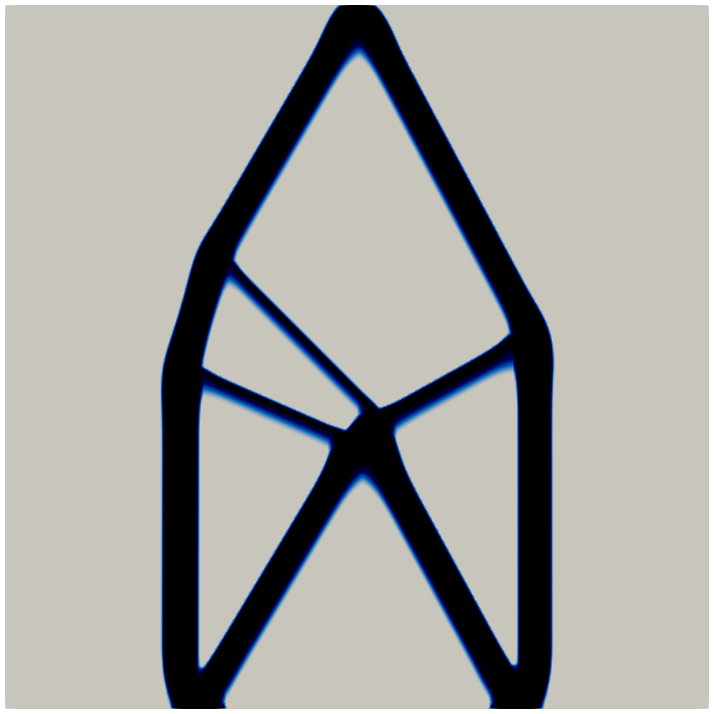}
\includegraphics[angle=-0,width=0.2\textwidth]{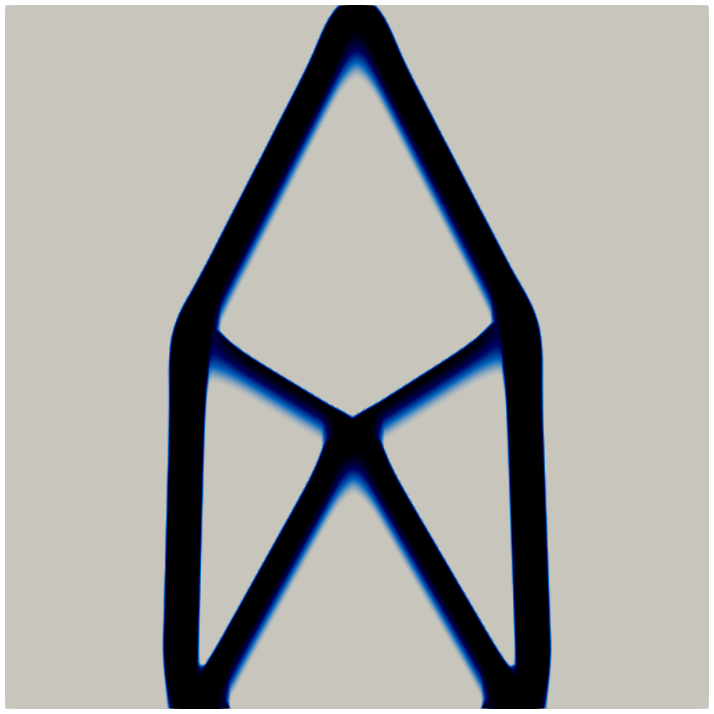}
\caption{Cantilever beam computation for $\eps=1/(64\pi)$ and with $\bm{g} = (g,0)^\top$ 
on $[-0.02,0.02] \times \{0.5\}$. The solutions at time $t=0.01$, starting from the final states displayed in
Figure~\ref{fig:2dAFig8_32pi_ani_forces2}.
Anisotropy \eqref{eq:HG3dnew} used with $\delta=0.1$. (Top) $g = 30$ with \an{$\alpha = 1$, $0.3$, $0.2$, $0.1$.} (Bottom) $g = 50$, \an{$\alpha = 1$, $0.3$, $0.2$, $0.1$.} Left-most column corresponds to the isotropic case.
}
\label{fig:2dAFig8_64pi_ani_forces2}
\end{figure}

In Figure~\ref{fig:2dAFig8_32pi_ani_forces2} we provide a more detailed comparison of the cantilever beam designs for the isotropic case $\alpha = 1$ and the strongly anisotropic cases $\alpha \an{=0.3, 0.2, 0.1}$ with loading magnitudes $g = 50$ and $g = 30$. In both situations, the anisotropic designs exhibit less interior structures and connecting bridges with steeper slopes than the isotropic designs, which is to be expected from the construction of $\gamma_{\alpha}$. It is worth mentioning that \an{some} (but not all) of the interfacial regions (\an{colored} deep blue) are thicker for smaller values of $\alpha$.  This is attributed to the fact that, from the expression of $\Phi_0$ obtained in the formally matched asymptotic analysis in Section \ref{sec:formal}, the interfacial thickness at a spatial point $\bm{x} = (\bm{s}, z)$ in the transformed coordinate system is proportional to $\gamma(\bm{\nu}(\bm{s}))$ with unit normal $\bm{\nu}$. Recalling the left of Figure~\ref{fig:Frankeg}, the normal vectors with directions associated to the line segment $L$ attain higher values of $\gamma$, and thus a thicker interfacial region, compared to directions associated with the circular arc $C$. Smaller values of $\alpha$ amplify the thickness as the line segment $L$ is closer to the origin, leading to higher values of $\gamma$.

In Figure~\ref{fig:2dAFig8_64pi_ani_forces2} we display refined designs taking as initial conditions the final states in Figure~\ref{fig:2dAFig8_32pi_ani_forces2} as well as a smaller value of $\eps =1/(64\pi)$. The overall designs are similar to Figure~\ref{fig:2dAFig8_32pi_ani_forces2} with some minor changes, particularly for the isotropic case $\alpha = 1$ with $g = 50$. We also note that the thicknesses of the interfacial layers in Figure~\ref{fig:2dAFig8_64pi_ani_forces2} are smaller than those in Figure~\ref{fig:2dAFig8_32pi_ani_forces2}, particularly for $\alpha = 0.1$, which is due to the smaller values of $\eps$ used.

\subsubsection{\an{Bridge construction}}

Our last numerical simulation is inspired by \cite[Figure~4]{WangZ04}, where we use the half-domain setup
\begin{align*} 
\Omega & = (0,1) \times (-\tfrac12,\tfrac12),\quad
\Gamma_{D_1} = \{0\} \times (-\tfrac12,\tfrac12),\quad
\Gamma_{D} = (\tfrac78,1) \times \{-\tfrac12\},\nonumber\\
\Gamma_g & = ((0, 0.02) \cup (\tfrac12-0.02,\tfrac12+0.02))\times\{-\tfrac12\},
\quad
\bm{g}(x_1,x_2) = \begin{cases}
-\tbinom0{3000} & x_1 \leq 0.02\,, \\
-\tbinom0{1500} & x_1 > 0.02\,.
\end{cases}
\end{align*}
For an elastic material with $E = 1200$ and $\nu=0.3$, we investigate the optimal designs with the convex regularized anisotropy \eqref{eq:HG3dnew}. In Figure~\ref{fig:2dbridge2q_32pi_anis} we display results for $\alpha = 1$, $0.7$, $0.5$ starting from random initial data with mass $m = 0.04$. Although the overall shapes are rather similar in all three cases, we observe the anisotropic cases (middle and right figures) have sharper corners on the underside of the bridge, whereas the underside of the isotropic case (left figure) is smoother.
\begin{figure}
\center
\includegraphics[angle=-0,width=0.3\textwidth]{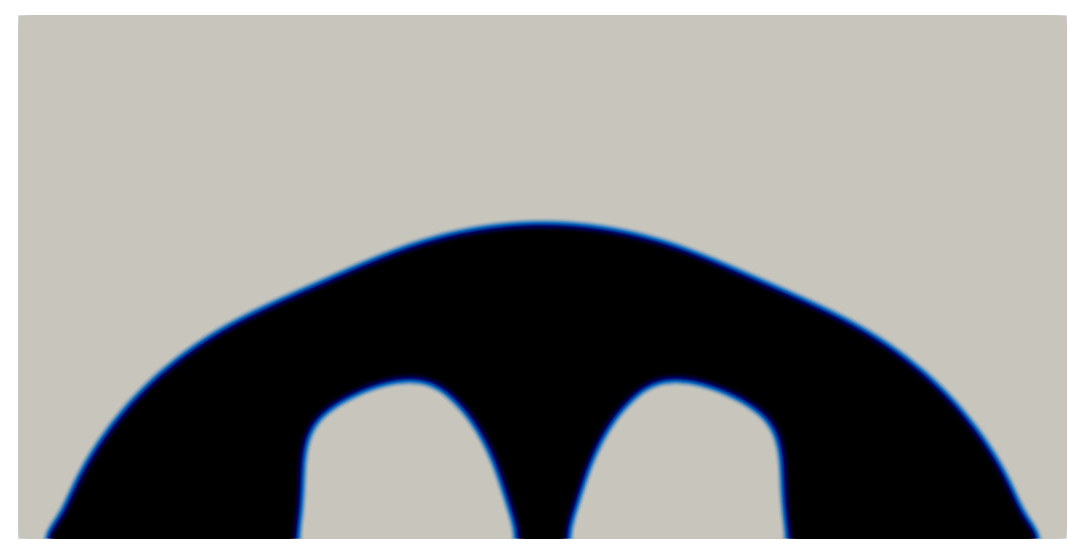}
\includegraphics[angle=-0,width=0.3\textwidth]{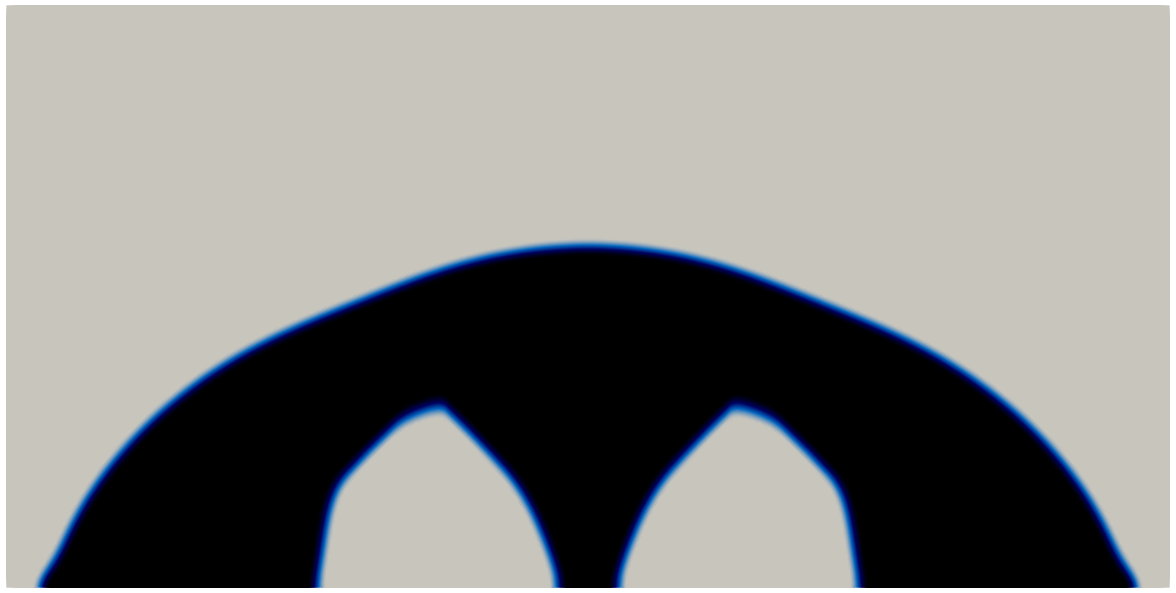}
\includegraphics[angle=-0,width=0.3\textwidth]{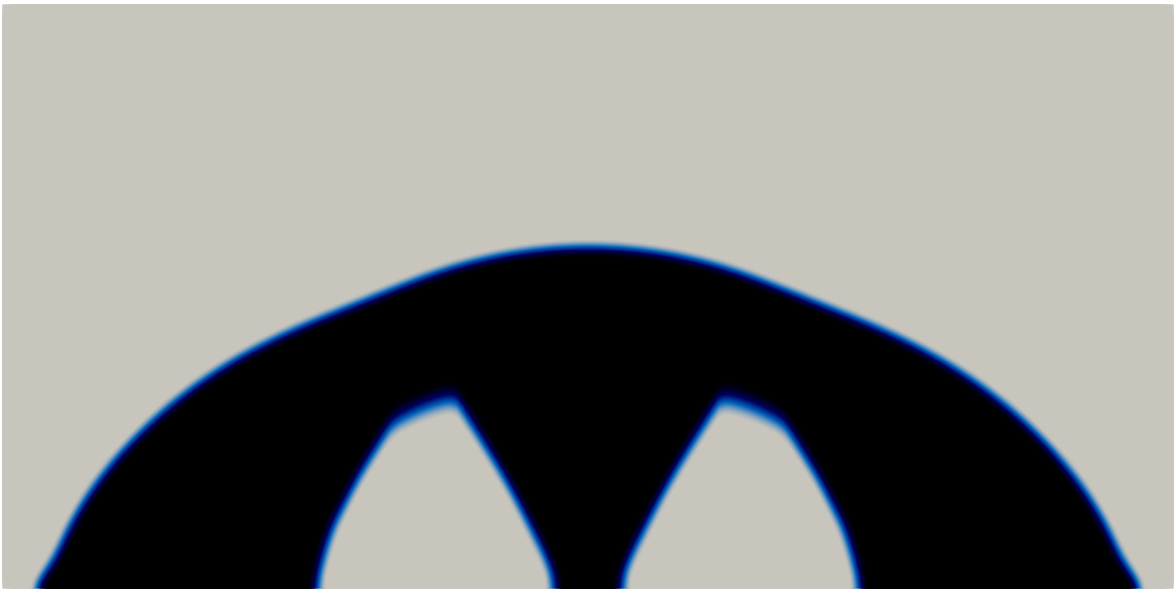}
\caption{Bridge construction half-domain computation for $\eps = 1/(32\pi)$. Anisotropy \eqref{eq:HG3dnew} used with $\delta=0.1$ and
$\alpha = 1$ (isotropic case), $0.7$, $0.5$.
The solutions at time $t=0.1$ starting from random initial data with mass $m = 0.4$.
}
\label{fig:2dbridge2q_32pi_anis}
\end{figure}

\section*{Acknowledgments}
\noindent The work of KFL is supported by the Research Grants Council of the Hong Kong Special Administrative Region, China [Project No.: HKBU 14302218].

\footnotesize
\bibliographystyle{plain}

\end{document}